\documentclass[paper=letter, 11pt, DIV=14, parskip=half, leqno]{scrartcl}
\usepackage{hyperref}
\usepackage{amsmath, amssymb, amsthm,enumerate}
\usepackage{graphicx, subcaption}
\usepackage[backend=biber, style=alphabetic, maxnames=6, maxalphanames=6]{biblatex}
\usepackage{fixdif}
\usepackage[svgnames]{xcolor}
\hypersetup{colorlinks, citecolor=DarkTurquoise, linkcolor=DarkRed}
\usepackage{subfiles,dsfont}
\usepackage{verbatim}
\usepackage{libertinus-type1, libertinust1math}
\usepackage{physics2}
\usepackage{commath}
\usepackage[normalem]{ulem}
\usephysicsmodule{ab}
\newtheorem{theorem}{Theorem}[section]

\theoremstyle{definition}

\newtheorem{remark}[theorem]{Remark}
\newtheorem{definition}[theorem]{Definition}
\newtheorem{lemma}[theorem]{Lemma}
\newtheorem{proposition}[theorem]{Proposition}

\newcommand{\bR}{\mathbb{R}}
\newcommand{\bZ}{\mathbb{Z}}
\newcommand{\bN}{\mathbb{N}}
\newcommand{\bQ}{\mathbb{Q}}
\newcommand{\bC}{\mathbb{C}}
\newcommand{\bE}{\mathbb{E}}
\newcommand{\bbD}{\mathbb{D}}
\newcommand{\bbE}{\mathbb{E}}
\newcommand{\bbP}{\mathbb{P}}

\newcommand{\bP}{\mathbb{P}}
\newcommand{\bbR}{\mathbb{R}}
\newcommand{\bbZ}{\mathbb{Z}}
\newcommand{\bbQ}{\mathbb{Q}}
\newcommand{\bbH}{\mathbb{H}}
\newcommand{\bH}{\mathbb{H}}
\newcommand{\fh}{\mathfrak{h}}

\newcommand{\cS}{\mathcal{S}}
\newcommand{\cG}{\mathcal{G}}
\newcommand{\cA}{\mathcal{A}}
\newcommand{\cC}{\mathcal{C}}
\newcommand{\cD}{\mathcal{D}}
\newcommand{\cR}{\mathcal{R}}
\newcommand{\cT}{\mathcal{T}}
\newcommand{\cL}{\mathcal{L}}
\newcommand{\rmW}{\mathrm{W}}
\newcommand{\rmE}{\mathrm{E}}
\newcommand{\rmS}{\mathrm{S}}
\newcommand{\rmL}{\mathrm{L}}
\newcommand{\rmR}{\mathrm{R}}
\newcommand{\fA}{\mathfrak{A}}
\newcommand{\sfL}{\mathsf{L}}
\newcommand{\sfR}{\mathsf{R}}
\newcommand{\sfZ}{\mathsf{Z}}
\newcommand{\sfV}{\mathsf{V}}
\newcommand{\sfU}{\mathsf{U}}
\newcommand{\sfS}{\mathsf{S}}
\newcommand{\bfc}{\mathbf{c}}
\newcommand{\ol}{\overline}
\newcommand{\e}{\varepsilon}
 
\newcommand{\cros}{\mathrm{cross}}
\newcommand{\curve}{\mathrm{curve}}
\newcommand{\field}{\mathrm{field}}
\newcommand{\SLE}{\mathrm{SLE}}

\newcommand{\LF}{\mathrm{LF}}

\newcommand{\QA}{\mathrm{QA}}
\newcommand{\MA}{\mathrm{MA}}
\newcommand{\MD}{\mathrm{MD}}
\newcommand{\MT}{\mathrm{MT}}
\newcommand{\MC}{\mathrm{MC}}
\newcommand{\QT}{\mathrm{QT}}

\newcommand{\Leb}{\mathrm{Leb}}
\newcommand{\IG}{\mathrm{IG}}

\newcommand{\wt}[1]{\widetilde{#1}}
\newcommand{\wh}[1]{\widehat{#1}}

\newcommand{\Md}{\mathcal{M}^{\mathrm{disk}}}

\DeclareMathOperator{\re}{Re}
\DeclareMathOperator{\im}{Im}

\DeclareMathOperator{\supp}{supp}

\numberwithin{equation}{section}

\newcommand{\zhenfeng}[1]{
    {\color{BlueViolet}\textbf{Zhenfeng: }#1}
}

\newcommand{\updated}[1]{
    {\color{Indigo}#1}
}
\title{Oriented planar maps and \\ annular mating of trees}

\author{Xingjian Di\thanks{NYU--ECNU Institute of Mathematical Sciences, New York University Shanghai.} 
\footnotemark[2]
\qquad Nina Holden\thanks{Courant Institute of Mathematical Sciences, New York University.}
\qquad Zhenfeng Tu\footnotemark[2]
\qquad Pu Yu\footnotemark[2]
}
\date{}

\begin{document}

\maketitle

\begin{abstract}
We study a class of oriented annular planar maps which are closely related to bipolar-oriented maps. We prove that such maps can be encoded by a 2D lattice walk which converges to a 2D Brownian path in the scaling limit. We further prove a mating-of-trees result for the annulus, where the aforementioned Brownian path  and a generalization encodes an LQG annulus  decorated by a counterclockwise space-filling SLE loop  with parameter $\gamma=16\kappa^{-1/2}\in(0,2)$. 

\end{abstract}

\tableofcontents

\section{Introduction}
Random planar maps are natural models for discrete random surfaces which are studied in several branches of both mathematics and physics. 
Liouville quantum gravity (LQG) is a model for a random surface originating in work of Polyakov~\cite{polyakov1981quantum} which has been extensively studied in probability theory in the past two decades. 
One research direction is to prove that random planar maps converge in the scaling limit to LQG surfaces.

In these scaling limit results, several different notions of convergence have been considered. For example, one can consider the \emph{Gromov-Hausdorff} topology where the planar map is viewed as a metric space by equipping it with its graph distance \cite{LG13, Mie13}. Equivalence of the Brownian surfaces considered in the cited works and LQG was proven in \cite{LQGBM1, LQGBM2, LQGBM3}. One can also consider convergence under \emph{discrete conformal embedding}, where one proves convergence of the counting measure induced in the plane by the embedded planar map \cite{GMS21, HS23}. Finally, one can consider convergence in \emph{peanosphere sense}, which we are doing 
in this paper. Several planar maps with a statistical physics model can be encoded by a 2D lattice walk. In the theory of \emph{mating-of-trees} \cite{MoT}, a 2D Brownian motion encodes an LQG surface decorated by a Schramm-Loewner evolution (SLE) \cite{SchrammSLE}. SLE has a single parameter $\kappa>0$ and has been proved or conjectured to describe the scaling limits of a large class of two-dimensional lattice models at criticality, e.g.~\cite{smirnov2001critical,lawler2011conformal,schramm2009contour,chelkak2014convergence}. The convergence in peanosphere sense means convergence of the 2D lattice walk encoding random planar maps to the 2D Brownian motion encoding space-filling SLE decorated LQG surfaces. Once convergence has been established in the peanosphere sense, it is natural to conjecture that convergence also holds in the other two senses. 
Examples include~\cite{She16} on the Fortuin-Kastelyn model, \cite{KMSW19} on bipolar orientations, ~\cite{BHS18} on percolation on triangulations, and ~\cite{LSW24} on Schynder woods.

Most scaling limit results are obtained for random planar maps with disk, sphere or whole-plane topology. A notable exception is \cite{bettinelli2017compact,BM22}, which considers uniform quadrangulations on general compact surfaces with boundary. One can also get natural multiply connected planar map models by considering a planar map with disk topology decorated by a loop model and ``removing'' the interior of some appropriately chosen loops \cite{ModAnn, Wu2023}; see also \cite{legall-drilling-annulus}. 
Other aspects of non-simply connected random planar map models are studied in e.g.~\cite{BG14,bonzom2022enumeration,fusy2021maps, ELT23,BGM24}, where enumerative results for  certain classes of planar maps of various topology have been obtained, often related to topological recursion~\cite{eynard2007invariants,eynard2008algebraic}. There have also been a number of works on planar maps in non-simply connected settings, e.g.,\cite{sun2024annulus} on annulus crossing formulae for critical planar percolation and \cite{DGL17,bernardi2025grand} on Schnyder woods in general settings. The mating-of-trees theory was first established in~\cite{MoT} for an infinite volume LQG surface called $\gamma$-quantum cone with whole-plane topology, and later extended to other simply connected LQG surfaces of finite volume in~\cite{MS19Finite,MoTboundary,ConfWeldDisks,AY23wholeplane,ASYZ24}.  {Before this work,} there was no mating-of-trees results for LQG surfaces of non-simply connected topology. One reason is that there is no general theory of space-filling SLE in non-simply connected domains.

In this paper, we study a certain class of oriented random planar maps of annular topology  that are closely related to the bipolar-oriented maps in~\cite{KMSW19}. We prove that such planar maps of annular topology can be encoded by 2D lattice walks, which converges to a Brownian path in the scaling limit (Theorem~\ref{thm: annular orientation scaling limit}). Based on the imaginary geometry~\cite{ImagGeo1,ImagGeo4}, we construct a version of space-filling SLE loops in the annulus and further  prove a mating-of-trees result for this version of SLE loop on the LQG annulus (Theorem~\ref{thm: quantum annulus mot}).  The SLE and LQG on annulus encoded by the Brownian path agrees with the scaling limit of the lattice walks encoding our  planar maps of annular topology, and therefore can be viewed as the scaling limit of the annular maps. There is also a random moduli of the annulus in the scaling limit, which is computed in Theorem~\ref{thm: law of the random modulus}.

\subsection{Main results and proof outline}
\label{sec: intro precise statements and outline}

Let $\gamma\in (0,2)$. Throughout this paper, we  work with   Brownian motions  $(L_t,R_t)_{t\in\bbR}$ with covariance
\begin{equation}
    \label{eq: gamma correlated BM}
    \mathrm{Var} (L_t) = \mathrm{Var} (R_t) = \mathbb{a}^2 |t|, \quad
    \mathrm{Cov}(L_t, R_t) = -\mathbb{a}^2 |t|\cos\theta, \ \quad \text{where $\theta =\frac{\pi \gamma^2}{4}$ and $\mathbb{a}^2 = \frac{2}{\sin\theta}$.}  
\end{equation}
We shall also work with Brownian path measures  induced by $(L,R)$; see e.g.~\cite{LW04}. For a domain $D$ and $x,y\in \overline D$, we write $B_{x,y}^{\gamma,D}$ for the Brownian path $(L,R)$ in $D$ from $x$ to $y$ with covariance matrix~\eqref{eq: gamma correlated BM}. See Section~\ref{sec: prelim Brownian path measure} for details.

A \emph{planar map} is a 
connected planar graph with finitely many edges together with a proper embedding into $\mathbb{S}^2 \cong \wh{\bC}:=\bC\cup\{\infty\}$. 
Two embeddings are equivalent if there exists an orientation-preserving homeomorphism from $\mathbb{S}^2$ to $\mathbb{S}^2$ that takes the first embedded graph to the second. All the planar maps we consider will be rooted, meaning that there is a distinguished oriented edge $e$. The tail vertex of $e$ is called the root vertex and the face to the right of $e$ is called the root face. An \emph{annular planar map} is a planar map with two distinguished faces $f_{\mathrm{inner}}$ and $f_{\mathrm{outer}}$ whose boundaries are simple loops, and we assume that $f_{\mathrm{outer}}$ is the root face.

In this paper, we work with oriented planar maps or  annular planar maps, where in the latter case we will assume that both the boundaries of $f_{\mathrm{inner}}$ and $f_{\mathrm{outer}}$ are oriented counterclockwise. 
A \emph{source} (resp.\ \emph{sink}) is a vertex with no incoming (resp.\ outgoing) edges, and an annular planar map is pole-free if there is no sink or source. We write $\cG^0$ for the set of pole-free annular planar maps as above with the additional condition that there are no contractible (directed) loops, i.e., all directed loops are non-contractible if we embed the annular planar maps and remove $f_{\mathrm{outer}}$ and $f_{\mathrm{inner}}$.
A \emph{bipolar-oriented map} is a planar map with a unique source and a unique sink which are both incident to the root face and there are no (directed) loops. In~\cite[page 1242]{KMSW19}, it is shown that for every vertex $v$ and its incidental edges, in counterclockwise order, there is only one group of incoming edges and one group of outgoing edges. In Lemma~\ref{lemma: good orientation two groups}, we prove that the same property holds for $G\in\cG^0$.   See Figure~\ref{fig: annular map example} for an illustration. This allows us to define the \emph{west edge} started from $v$ to be the last outgoing edge in its group. A \emph{west path} is a path such that all of its edges are west edges. Let $\cG $ be the maps in $\cG^0$ such that there exists a crossing west path,  i.e., a west path from  $f_{\mathrm{outer}}$ to $f_{\mathrm{inner}}$. See Figure~\ref{fig: tree encoding} (left) for an illustration.

\begin{figure}[t]
    \centering
    \includegraphics[height=4cm, page=3]{img/tree-encoding.pdf}
    \hspace{1cm}
    \includegraphics[height=4cm, page=2]{img/tree-encoding.pdf}
    \hspace{1cm}
    \includegraphics[height=4cm, page=1]{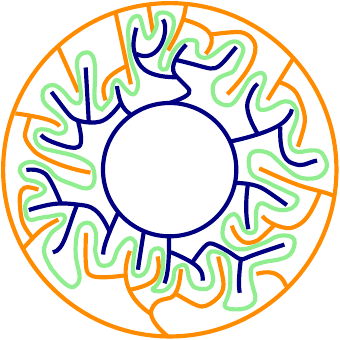}
    \caption{
    \textbf{Left:} A pole-free oriented annular map in $\cG$ where the boundaries are oriented counterclockwise, and there exists a crossing west path (colored in red) starting from the root vertex  and no contractible loops. 
    \textbf{Middle:} In Section~\ref{sec: random conditioned orientation}, we cut all the non-west edges near their starting points to generate a west tree (blue). Likewise, we can reverse the orientations to apply the same procedure to generate an east tree (orange). This generate a space-filling curve (green) moving along faces and edges of the map. 
    \textbf{Right:} The curve converges to a space-filling loop in the scaling limit, which could be thought of as a space-filling $\SLE_{16/\gamma^2}$ in the annulus.}
    \label{fig: tree encoding}
\end{figure}

It is known that bipolar-oriented maps can be encoded by lattice walks~\cite{KMSW19} which are also known as \emph{height functions}, and one can generate a space-filling path on the map that 
explores the map
in a natural way. It is shown in~\cite{KMSW19} that the boundary of each non-root face can be decomposed as a union of a clockwise directed path of length $i+1$  and a counterclockwise directed path of length $j+1$ on the right, and such a face is said to be \emph{type-$(i,j)$}. Our first result is an extension of these conclusions to  the annular maps in $\cG$. See Figure~\ref{fig: tree encoding} (middle) and Figure~\ref{fig: height function encoding} for an illustration. In Lemma~\ref{lem:type-ij}, we prove that the notion of type-$(i,j)$ can be extended to maps $\cG$.  Let $\{b_e,b_{i,j}: i,j\in \mathbb{Z}_{\geq 0}\}$ be a set of non-negative real numbers with the following conditions. 
    \begin{enumerate}[(A)] \label{conditions for map weights}
        \item \emph{Bounded Degree}: There exists a constant $M>1$ such that $b_{i,j} = 0 $ if $i+j \geq M$. \label{cond: b-A}
        \item \emph{Symmetry}: $b_{i,j} = b_{j,i}$. \label{cond: b-B}
        \item \emph{Zero Drift}: $b_e(-1,1) + \sum_{i,j}b_{i,j}(j,-i) = (0,0)$. \label{cond: b-C}
        \item \textit{Defines a probability measure}:  $b_e + \sum_{i,j} b_{i,j} = 1$. \label{cond: b-D}
    \end{enumerate} 
  For each $G\in\cG$, let $E(G)$, $F(G)$ be its edge set and face set. We set its weight to be 
 \begin{equation}
        \label{eq: planar map weight}
        W(G) = 
        b_e^{|E(G)|-|F\degree(G)|} \prod_{i, j\in \bZ_{\ge 0}} b_{i,j}^{n_{i, j}},
    \end{equation}
where $F\degree(G)=F(G)\setminus\{f_{\mathrm{inner}},f_{\mathrm{outer}}\}$  and $n_{i,j}$ is the number of type $(i,j)$-faces in $G$. We use the convention $0^0=1$.   For $\ell$, $n\in\bN$ and $m\in\bZ_{\ge 0}$, we let $\mathcal{H}_{k}(\ell, m, n)$ be the set of lattice paths $\sfZ = (\mathsf{L}_j,\mathsf{R}_j)_{0\leq j\leq n}:[0,n]\cap\bbZ \to\bbZ^2$ that satisfy the following conditions.

    \begin{enumerate}[(I)]
        \item Each step of the lattice path is equal to $(-1, 1)$ or $(j, -i)$ for some $i$, $j\in \bZ_{\ge 0}$.\label{cond: Hk lmn cond 1}
        \item We have $\sfZ(0)=(0, 0)$ and $\sfR(j)\geq0$ for $0\leq j\leq n$.
        \item We have $\sfL(n)=-\ell$, $\inf_{0\le j\le n} \sfL(j) = -m-\ell$ and $\sfR(n)=k$.
        \label{cond: Hk lmn cond 4}
    \end{enumerate}

    We also write $\mathcal{H}_k = \bigcup_{\ell, n\in \bN, m\in \bZ_{\ge 0}}\mathcal{H}_k(\ell, m, n)$, and $\mathcal{H} = \bigcup_{k\in \bN} \mathcal{H}_k$. Note that given a lattice path $\sfZ\in\mathcal{H}$, the parameters $\ell$, $m$, $n$ and $k$ may be recovered from the path, and we set its   weight to be
    \begin{equation}
        \label{eq: lattice path weight}
        W(\mathsf{Z}) = b_e^{n_e}\prod_{i,j\in \mathbb Z_{\geq 0}} b_{i,j}^{n_{i,j}},
    \end{equation}
    where $n_e$ (resp.\ $n_{i,j}$) is the number of $(-1, 1)$  (resp.\ $(j, -i)$) steps.  Now we are ready to state our first result. Throughout this paper we write $\bbH$ for the upper half plane $\{z\in\bC: \mathrm{Im}\,z> 0\}.$   For two processes $w_1,w_2$, we will also use the metric
    \begin{equation}\label{eq:curve-metric}
        d(w_1,w_2) = \inf_{\psi}\sup_{t\in [0,T_1]} \ab( |w_1(t) - w_2(\psi(t))| + |\psi(t) - t|),
    \end{equation}
    where $w_j:[0,T_j]\to \mathbb{R}^2$ and the infimum is over all increasing bijections $[0,T_1]\to [0,T_2]$.

\begin{figure}[t]
    \centering
    \includegraphics[height=5cm, page=1]{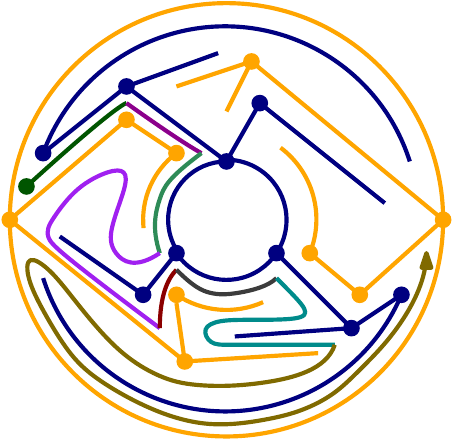}
    \hspace{1cm}
    \includegraphics[height=5cm, page=2]{img/height-function-encoding.pdf}
    \caption{
    \textbf{Left:} The oriented annular map and the space-filling curve as in Figure~\ref{fig: tree encoding}, where each color represents a step of the curve. \textbf{Right:} The height function $(\mathsf L, \mathsf R)$ encodes the change of height in the west and east forest, respectively, at each step. The dark green and dark gold steps are slightly bended for illustration purpose. For a step during which the Peano curve does not cross a face of the map or trace its boundary, the associated step of the walk will be $(-1,1)$. When the Peano curve crosses a type-$(i,j)$ face, the associated step of the walk will be $(i,-j)$.}
    \label{fig: height function encoding}
\end{figure}

    \begin{theorem}   
\label{thm: annular orientation scaling limit}
Let $k\in\bN$ and let $\cG_k$ be the set of  maps in $\cG$ with $k$ edges on the outer boundary.  Then 
\begin{enumerate}
    \item For each $k\geq 1$, there exists a   bijection between $\mathcal G_k$ and $\mathcal H_k$ which sends an oriented annular map to the associated pair of height functions and preserves the weights~\eqref{eq: planar map weight} and~\eqref{eq: lattice path weight}.
    \item Fix $b>0$. Let $G$ be sampled from $\mathcal G_{\lfloor b\sqrt{N}\rfloor} $ according to the weights~\eqref{eq: planar map weight} with associated height function $(\sfZ_j)_{0\leq j\leq n}$. Then as $N\to \infty$, the law of $(\frac{1}{\sqrt N}\sfZ_{\lfloor Nt \rfloor })_{0\leq \lfloor Nt \rfloor\leq n}$ converges in distribution to $\int_0^\infty \d a B^{\gamma, \bH}_{0, -a+ib}$ with respect to the metric~\eqref{eq:curve-metric}. The   parameter $\gamma$ is given by
    \begin{equation}
        \label{eq: relation between gamma and weights}
        \cos \ab(\frac{\pi\gamma^2}{4}) = \frac{2b_e + 2\sum_{i,j \in \bZ_{\ge 0}} ijb_{i,j}}{2b_e + \sum_{i,j\in \bZ_{\ge 0}}(i^2+j^2) b_{i,j}}
    \end{equation}
    and by choosing  numbers $\{b_e,b_{i,j}:i,j\in\bZ_{\geq0}\}$ properly, $\gamma$ may take every value in $(0,\sqrt{2})$.
\end{enumerate}
\end{theorem}
Note that the law of $(\frac{1}{\sqrt N}\sfZ_{\lfloor Nt \rfloor })_{0\leq \lfloor Nt \rfloor\leq n}$ and $\int_0^\infty \d a B^{\gamma, \bH}_{0, -a+ib}$ are both non-probability measures, and we will show that the partition function for $(\frac{1}{\sqrt N}\sfZ_{\lfloor Nt \rfloor })_{0\leq \lfloor Nt \rfloor\leq n}$ converges to the total mass of $\int_0^\infty \d a B^{\gamma, \bH}_{0, -a+ib}$.
The proof is based on the observation that, as in Lemma~\ref{lemma: cut to be bipolar}, if we cut along the crossing west path starting from the root vertex, we obtain a   bipolar-oriented map with disk topology and certain boundary conditions. This allows us to extend the properties of bipolar-oriented maps  to $\cG$ and modify the bijection in~\cite{KMSW19} to prove the first part of the theorem.  The second part follows from a standard Donsker-type argument.

Next we recap the theory of mating-of-trees for Liouville quantum gravity. Let $\gamma\in(0,2)$ and let $h$ be some variant of Gaussian free field (GFF). The Liouville quantum gravity, roughly speaking, is the geometry induced by $e^{\gamma h}$, where the LQG area is $e^{\gamma h(z)}\,d^2z$ and the LQG length is $e^{\gamma h(z)/2}\,dz$. The LQG area and length can be made rigorous via a limiting procedure, see e.g.\ \cite{DS11}. Consider a whole-plane $\gamma$-LQG surface called the $\gamma$-quantum cone introduced in~\cite[Definition 4.10]{MoT} and a whole-plane space-filling $\SLE_{16/\gamma^2}$ curve $\eta':(-\infty,\infty)\to\bC$  from $\infty$ to $\infty$ parameterized by   LQG area.  Consider the LQG length of the left and right boundaries of $\eta'((-\infty,t]))$, and let $(L_t,R_t)$ be the change of these LQG lengths relative to time 0. Then it is shown in~\cite{MoT} that $(L,R)$ evolves as a Brownian motion with covariance~\eqref{eq: gamma correlated BM} (the variance $\mathbb{a}^2$ was computed in~\cite{ARS23}), and further determines the  $\gamma$-quantum cone and the SLE curve.  This type of results was later extended to other simply connected LQG surfaces of finite volume in~\cite{MS19Finite,MoTboundary,ConfWeldDisks,AY23wholeplane,ASYZ24} where $(L,R)$ is a Brownian path. Many of these LQG surfaces can be described in terms of \emph{Liouville conformal field theory}~\cite{AHS17,cercle2021unit,AHS24,QuanTrig}, a 2D quantum field theory rigorously developed in~\cite{DKRV16} and subsequent works~\cite{GKRV20_bootstrap,KRV20,GKRV21_Segal}. The mating-of-trees approach to LQG has been helpful in studying a number of problems related to LQG, including random planar maps and integrability of statistical physics models and SLE. 

\begin{figure}[t]
    \centering
    \includegraphics[scale=0.77]{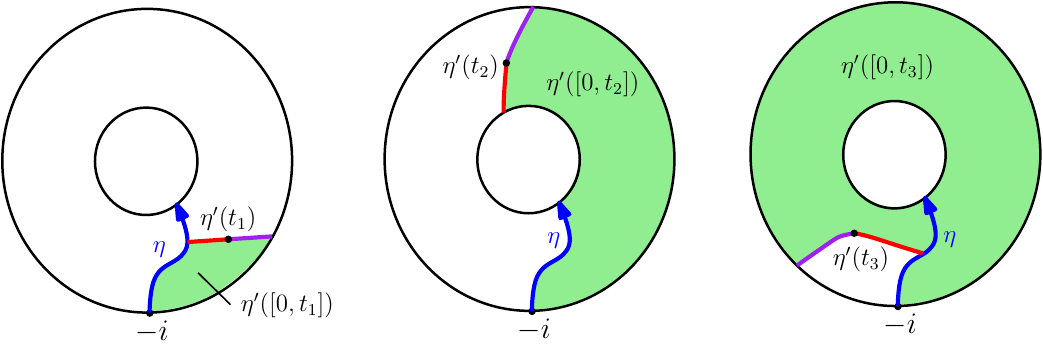}
    \caption{An illustration of the boundary length process $(L,R)$ in Theorem~\ref{thm: quantum annulus mot} where $\gamma\in (0,\sqrt 2)$; the case where  $\gamma\in (\sqrt 2,2)$ is similar. The counterclockwise space-filling  curve $\eta'$ is rooted at $-i$ and parameterized by LQG area.  The interface $\eta$ is the left boundary of the curve $\eta'$  stopped when first hitting the inner boundary. For each $t$, $L_t$ is the LQG length of the red segment plus the intersection between $\eta'([0,t])$ and the left side of $\eta$, minus the LQG length of the intersection between $\eta'([0,t])$ and the inner boundary as well as the right side of $\eta$, while $R_t$ is the LQG of the purple segment plus the   intersection between $\eta'([0,t])$ and the outer boundary.}
    \label{fig:weld-annulus-10}
\end{figure}

Our next result is the mating-of-trees on the annulus, where we prove that a Brownian path sampled from $\int_0^\infty \d a B^{\gamma, \bH}_{0, -a+ib}$ encodes an LQG annulus decorated by a space-filling SLE. The LQG annulus is studied in~\cite{ModAnn} and can be described in terms of LCFT on the annulus~\cite{remy2018liouville}.  For $\tau>0$, let $\cA_\tau := \{z\in \bC: e^{-2\pi\tau}<|z|<1\}$ denote an annulus with modulus $\tau$. Then the LQG annulus  can be described in terms of  $(\cA_\tau,\phi)$ where $\phi$ is a variant of GFF sampled from the measure $\LF_\tau$. We will work with the measure  $\LF_\tau^{(\gamma,-i)}$, where we first weight the law of $\phi$ by the LQG length of $C_0:=\partial\bbD$ and sample a marked point on $C_0$ according to the probability measure proportional to the LQG length measure, and further rotate such that the marked point is located at $-i$. This marked point corresponds to the root vertex of the annular planar maps in Theorem~\ref{thm: annular orientation scaling limit}. See Section~\ref{sec: LQG surfaces} for further details. 

We will also work with the space-filling SLE  loop on the annulus. SLE curves are first introduced for simply connected domains, and there are multiple ways to extend to non-simply connected domains, including Loewner equations~\cite{zhan2004stochastic,zhanannulus,BF04SLE}, imaginary geometry~\cite{ImagGeo4} and Brownian loop measures~\cite{lawler2011defining,zhan2021sle}. The SLE curves are constructed locally in the first two ways (the endpoint continuity for annulus SLE was discussed in~\cite{zhanannulus}), and in the last way the SLE curves are constructed globally (i.e., the curve is constructed all at once) but this works only in the simple regime ($\kappa\leq 4$). For simply connected domains, there is also a version of space-filling SLE~\cite{ImagGeo4}, which agrees with ordinary SLE for $\kappa\geq 8$ but different for $\kappa\in(4,8)$. In Section~\ref{sec: IG annulus}, we study a particular imaginary geometry field on the annulus, and construct a counterclockwise space-filling SLE loop on the annulus, whose law is denoted by $\SLE_\tau^{\mathrm{loop}}$. Different from~\cite{ImagGeo4}, our field on the annulus is multi-valued, and has close relation to a field  with one magnetic insertion from compactified imaginary Liouville theory~\cite{GKR23}. 

\begin{theorem}
    \label{thm: quantum annulus mot}
    For $\gamma \in (0,  {2})\backslash\{\sqrt2\}$, there exists a constant $C_\gamma\in (0, \infty)$ and a $\sigma$-finite measure $m(\d\tau)$ on $(0, \infty)$, such that the following is true. Consider the embedding $(\cA_\tau,\phi,\eta',-i)$  of a sample from
    \begin{equation}
        \label{eq: MA 01 defn intro}
        \MA^{1,0} := 
        \LF_\tau^{(\gamma, -i)}(\d\phi)\, \SLE^{\mathrm{loop}}_\tau(\d\eta')\, m(\d\tau)/\mathord\sim,
    \end{equation}
  where the counterclockwise space-filling $\SLE_{16/\gamma^2}$ loop is rooted at $-i$ and parameterized by the $\gamma$-LQG area with respect to $\phi$.   The boundary length process $(L,R)$ in Figure~\ref{fig:weld-annulus-10}
  has law $C_\gamma\int_0^\infty \d a \int_0^\infty \d b B^{\gamma, \bH}_{0, -a+bi}$.
  Moreover, the curve-decorated annulus is almost surely determined by its boundary length process.
\end{theorem} 

\begin{remark}
    The reason we did not include $\gamma=\sqrt{2}$ in the statement is that the version of the mating-of-trees for quantum triangles that we are using (Theorem~\ref{thm:ASYZ3.9}, from~\cite{ASYZ24}), did not include $\gamma=\sqrt{2}$. The reason for this, is that the proof of Theorem~\ref{thm:ASYZ3.9} relied on~\cite[Proposition 3.6]{ASYZ24}, which used inputs from~\cite{SY23}. If we directly use results from~\cite{SY23}, then one would need to work with Liouville fields of the form $\LF_\bbH^{(\beta_i,s_i)}$ (see Definition~\ref{def:LF-HH}) with 4 insertions and one insertion being $Q^-$, which was not treated in that paper. One way to get around this is to apply~\cite[Proposition 3.8]{ang2025p} and use arguments in~\cite[Section 6.5]{QuanTrig} to take a $W_0\uparrow \frac{\gamma^2}{2}$ limit. Once this is done, both Theorem~\ref{thm:ASYZ3.9} and Theorem~\ref{thm: quantum annulus mot} will hold for $\gamma=\sqrt2$.
\end{remark}


\begin{figure}[t]
    \centering
    \includegraphics[scale=0.76]{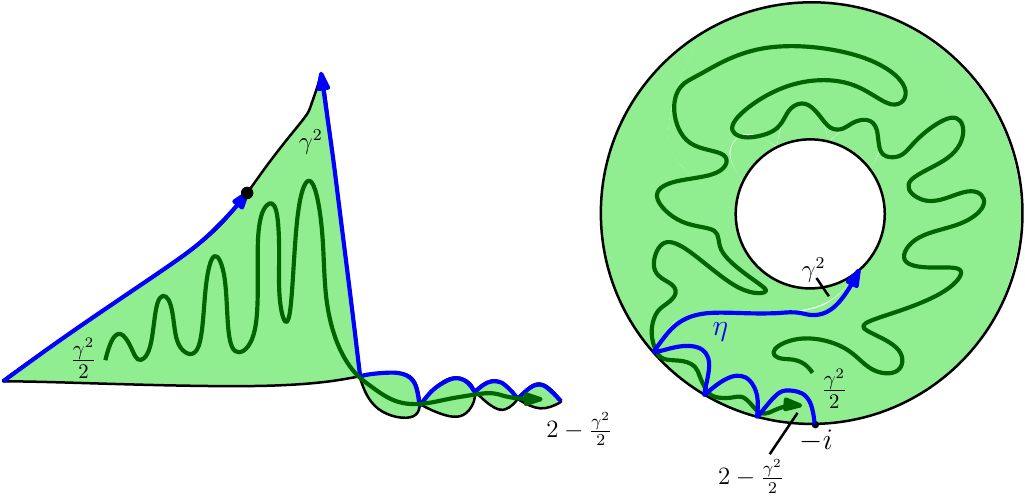}
    \caption{An illustration of the conformal welding of the weight $( 2-\frac{\gamma^2}{2}, \gamma^2,\frac{\gamma^2}{2})$ quantum triangle (left) to itself, where $\gamma\in (\sqrt2,2)$ and there are infinitely many beaded disks near the weight $2-\frac{\gamma^2}{2}$ vertex; there will be no such beads if $\gamma\in (0,\sqrt 2)$ and the quantum triangle is simply connected in this case. The dark green curve is a space-filling $\SLE_{16/\gamma^2}(\frac{8}{\gamma^2}-4;\frac{8}{\gamma^2}-4)$ as in Theorem~\ref{thm:ASYZ3.9}.  We assume that the boundary arc connecting the weight $\gamma^2$ and $\frac{\gamma^2}{2}$ vertices is longer than the the boundary arc connecting the weight $\gamma^2$ and $2-\frac{\gamma^2}{2}$ vertices in terms of $\gamma$-LQG length. After gluing, we obtain an LQG surface with annulus topology. In Theorem~\ref{thm: quantum annulus mot} we will prove that this LQG surface can be described via Liouville CFT on the annulus~\cite{LCFTAnnulus,ModAnn}, and the interface will be a counterclockwise space-filling $\SLE_{16/\gamma^2}$ loop in the annulus as discussed in Section~\ref{sec: IG annulus}.}
    \label{fig:weld-annulus-8}
\end{figure}

Combining Theorem~\ref{thm: annular orientation scaling limit} and Theorem~\ref{thm: quantum annulus mot}, we get a  scaling limit result for the annular planar map and its space-filling curve to \eqref{eq: MA 01 defn intro}. Our proof of Theorem~\ref{thm: quantum annulus mot} is based on \emph{conformal welding} for LQG surfaces, where one can glue two LQG surfaces together according to the LQG boundary length measure. This was discovered in~\cite{She16a,MoT} and later widely extended and applied in e.g.~\cite{ConfWeldDisks,QuanTrig,ACSW24,AHS24,ASYZ24,LSYZ24}. We start from the mating-of-trees for a space-filling $\SLE_\kappa$ decorated \emph{quantum triangle} of weights $( 2-\frac{\gamma^2}{2}, \gamma^2,\frac{\gamma^2}{2})$, an LQG surface with three marked boundary points studied in~\cite{QuanTrig,ASYZ24}; see Theorem~\ref{thm:ASYZ3.9}.  The Brownian path measure in Theorems~\ref{thm: annular orientation scaling limit} and~\ref{thm: quantum annulus mot} appears in this mating-of-trees result. In light of this,  in order to prove Theorem~\ref{thm: quantum annulus mot} it is sufficient to argue that as we conformally weld  the quantum triangle to itself to get an annulus as in Figure~\ref{fig:weld-annulus-8}, the law of the field and the curves on the annulus, denoted by $\MA^{1,0,\dagger}$, is equal to the measure $\MA^{1,0}$  in~\eqref{eq: MA 01 defn intro}. This is done in two parts: the imaginary geometry on the annulus in Section~\ref{sec: IG annulus}, and the resampling property of the fields and curves from $\MA^{1,0} $ and $\MA^{1,0,\dagger} $ in Section~\ref{sec: resampling characterization scaling limit}. 

To construct the counterclockwise space-filling $\SLE_{16/\gamma^2}$ curve in the annulus, we start with the universal cover $\cS_\tau=\bbR\times (0,\tau)$ where the GFF is now periodic. We first use the locality of the flow lines of GFF as in~\cite[Theorem 1.2]{ImagGeo1} along with local absolute continuity between different GFF as in~\cite[Proposition 2.16]{ImagGeo4} to define flow lines of the periodic GFF on $\cS_\tau$ (Proposition~\ref{prop:def-flow-line-strip}).  In Lemma~\ref{lem:flow-line-continuity}, we further prove that the flow lines either merge into one of its translates or $\bbR+i\tau$, and we work with the event $E^\cros$ where there the west-going (i.e., angle $\frac{\pi}{2}$) flow line $\eta_0^\rmW$ started from 0 hits $\bbR+i\tau$ before merging into any of its translates $\eta_0^\rmW+k$. This is in accordance with our condition that there exists a crossing west path in the discrete setting. Our (periodic) space-filling SLE is then the concatenation of the usual chordal space-filling counterflow line of the GFF within each connected component of $\cS_\tau\backslash \{\eta_0^\rmW+k:k\in\bbZ\}$.  We define the flow lines and space-filling counterflow lines of the GFF in the annulus $\cA_\tau$ to be the image of their counterpart in $\cS_\tau$ under the map $p:z\mapsto e^{2\pi iz}$.

To characterize the surface and the interface after conformal welding we use a resampling argument. 
In Proposition~\ref{prop: MC+MD}, we prove that if one starts with the space-filling SLE decorated quantum triangle as in Theorem~\ref{thm:ASYZ3.9} and stops the SLE curve when it hits a 
uniform point on the boundary arc connecting the weight $\frac{\gamma^2}{2}$ and weight $\gamma^2$ vertices,
then one gets a picture that can be understood as the conformal welding of a weight $\frac{\gamma^2}{2}$ quantum disk as discussed in~\cite[Remark 7.10]{ConfWeldDisks} and a quantum cell as discussed in~\cite{ang2023liouville,AY23wholeplane}. See Figure~\ref{fig: two ways of welding} for an illustration. This is done  by considering the associated 2D Brownian paths 
(Lemma~\ref{lemma: first and last passage decomposition}). The conformal welding of the quantum disk and the quantum cell is symmetric, which gives a resampling property for the interfaces as in Proposition~\ref{prop:curve-resample}. Based on this, we construct irreducible Markov kernels $\Lambda^\curve$ and $\Lambda^\field$ and prove that both  $\MA^{1,0} $ and $\MA^{1,0,\dagger} $  are invariant measures, which allows us to conclude by~\cite{MT12} that they must be the same up to a multiplicative constant. This finishes the proof of Theorem~\ref{thm: quantum annulus mot}. 
Resampling arguments and irreducible Markov kernels have been used for SLE~\cite{ImagGeo2,Yu22,Zhan23} and LQG~\cite{ModAnn,QuanTrig,ang2024radial} in some earlier contexts, but we are the first paper to use this for a planar map scaling limit result, and we expect that our methods can be extended to other non-simply connected planar map models. Additionally, we have not seen earlier works that establish a resampling property via a decomposition of a Brownian path measure, or which construct a natural annular SLE-decorated surface by welding an SLE-decorated surface with disk topology to itself.

 Our final result is the exact formula for the law of the modulus $\tau$. The Brownian path measure in Theorem~\ref{thm: quantum annulus mot} gives  {the joint law of} the inner and outer LQG boundary lengths of surfaces from $\MA$. We adopt the following convention for the Jacobi $\vartheta$ function
\begin{equation}
    \label{eq: theta 3 definition}
    \vartheta_3(z, q) = 1+2\sum_{n=1}^\infty q^{n^2} \cos(2 n z).
\end{equation}
 Combined with the integrability of the Liouville fields on the annulus~\cite{LCFTAnnulus,ModAnn}, in Section~\ref{sec: law of the random modulus} we prove the following. 

\begin{theorem}
    \label{thm: law of the random modulus}
    The measure $m(\d\tau)$ in \eqref{eq: MA 01 defn intro} is given by
    \begin{equation}
        C_\gamma \ab[\vartheta_3(0,e^{-\theta\tau}) - \vartheta_3(\theta, e^{-\theta\tau})]\d\tau,
    \end{equation}
    where $C_\gamma$ normalizes $m$ to be a probability measure and $\vartheta_3$ is as defined in \eqref{eq: theta 3 definition}.
\end{theorem}

{
\begin{remark}
 
    Let $g(\theta, \tau) = \vartheta_3(0, e^{-\theta\tau})-\vartheta_3(\theta, e^{-\theta\tau})$.
    In Section~\ref{app: gamma = 0 limit}, we show that for every $\delta>0$,
    \begin{equation*}
        \lim_{\theta\downarrow 0}\frac{\int_\delta^\infty g(\theta, \tau)\d\tau}{\int_0^\infty g(\theta, \tau)\d\tau} = 0.
    \end{equation*}
    Loosely speaking, this implies that $m(\d\tau)$ becomes a point mass at $0$ in the semiclassical limit  {$\gamma\downarrow 0$}.
    In other words, the annulus will be infinitesimally thin. This is consistent with the observation that if $\gamma=0$, the lengths of the inner boundary and the outer boundary will be almost surely equal. It is also possible to scale in other ways to obtain different limits; see Remark~\ref{remark: gammato0} for more details. 
\end{remark}
}

\subsection{Future directions}
\label{sec: future directions}
A first natural extension of our model is to remove the conditioning on the crossing west path in the discrete model. In this case, one can prove that the west paths and east paths are still well-defined, but there might be multiple Peano curves separated by west or east cycles; see Figure \ref{fig: multiple annulus}. In the continuum, this corresponds to removing the conditioning on $E^\cros$ for the imaginary geometry field, and a west-going flow line in the annulus may merge with itself and form an $\SLE_\kappa$-type loop. The limiting surface can be described as the conformal welding of multiple surfaces of the type we consider in this paper or closely related variants. This model will be studied in a future work of the authors.

\begin{figure}[t]
    \centering
    \includegraphics[height=5cm, page=1]{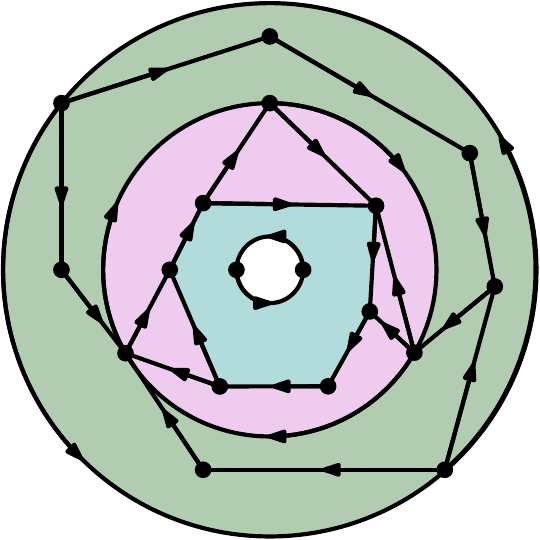}
    \hspace{3cm}
    \includegraphics[height=5cm, page=3]{img/multi-annulus.pdf}
    \caption{
    \textbf{Left:} A pole-free almost acylic annular map without a crossing west path. The annular map is separated into three regions by a west cycle and an east cycle in the bulk, shaded in green, pink and cyan, respectively. A Peano curve may be drawn in each region. 
    \textbf{Right:} In the continuum, we obtain the conformal welding of a random number of (variants of) the annuli we obtain in this paper.
    The welding interfaces are $\SLE_{\kappa}$-type loops, generated by west-going (blue) and east-going (orange) flow lines.
    Each annulus is also decorated with a space-filling loop, in green, pink and cyan, respectively.
    }
    \label{fig: multiple annulus}
\end{figure}


\begin{figure}
    \centering
    \includegraphics[scale=1]{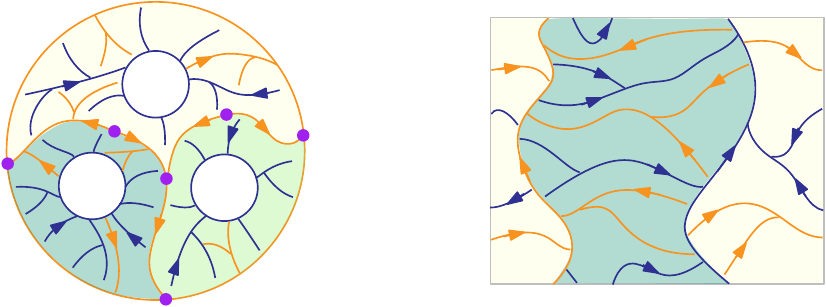}
    \caption{ Illustrations of potential scaling limits of planar maps with orientations in the case of multiply connected (left) and torus (right) topologies. The blue and orange curves represent west-going and east-going, respectively, flow lines. These flow lines divide the surfaces into three (left) or two (right) annuli in the examples in the figure, but the number of annuli is random. The purple points in the left figure are special points for the flow lines.}
    \label{fig: pfaa maps general topology}
\end{figure}

Another natural extension is to consider random maps of other topologies, for example higher genus or multiply connected topologies.  One can also add singularities to the surface by allowing a fixed set of sources and/or sinks, where the case of a single sink can be viewed as an annulus with the inner boundary shrunk to a single vertex.
In at least some cases we expect that, both in the discrete and the continuum, the surfaces can be decomposed into simpler surfaces with annular topology by ``cutting'' along suitably chosen flow lines, including flow lines forming loops. See Figure \ref{fig: pfaa maps general topology} for an example.  We expect that the limiting surface can be described by a Liouville field and a certain imaginary geometry field. The imaginary geometry field might be related to  
a compactified GFF, which describes the scaling limit of dimers on Riemann surfaces 
\cite{DimersIG,DimersRiemannI,DimersRiemannII,DimersDGFF} and has been studied recently in the context of compactified imaginary Liouville theory \cite{GKR23}. In settings where the surface decomposes into simpler annular surfaces one can attempt to compute the partition function of the surface by using the partition functions we get for the annular surfaces via Brownian computations.


In the simply-connected setting, there has been a number of planar maps decorated by a statistical physics model whose scaling limit can be described via mating-of-trees. One can also study similar questions to the ones above for the other models, 
such as an FK model or a tree/forest. It is also natural to consider models on regular lattices, where one might guess that in the annulus setting, the scaling limit of the interfaces can be described via imaginary geometry.

Another interesting future direction is the study of the SLE-type curves which arise in this paper and in the generalized settings discussed above. Some earlier works on SLE in non-simply connected settings include e.g.~\cite{zhan2004stochastic,lawler2011defining,jahangoshahi2018multiple,zhan2021sle,aru2024sle}. It is of interest to investigate a potential relationship between our curves $\eta$ and annulus SLE \cite{zhan2004stochastic}, and the SLE-type loops discussed above might be related to Brownian loop measures and SLE loop measures. It is also of interest to get explicit Loewner equation descriptions or other characterizations of these SLE-type curves. We expect that our west-going flow lines in the annulus are locally absolutely continuous with respect to annulus variants of $\SLE_\kappa(-\kappa;\kappa-2)$ curves.


An interesting feature of our model is that it makes sense for the  range $\gamma\in(0,2)\backslash\ \{\sqrt 2\}$ and in particular it allows to study the limit $\gamma\downarrow 0$. It is known that both SLE \cite{wang2019energy, PW23} and LCFT \cite{LRV22semiclassical} have a rich behavior in the semiclassical limit and our model allows to study this from a planar map and mating-of-trees perspective. A particular direction of interest is to see how the partition function of the surface behaves in the semiclassical limit in the setting of more general topologies and investigate whether there is a relationship with the Weil-Petersson volume.
We remark that several earlier works have established a close relationship between the latter volumes and random planar maps via topological recursion; see \cite{CoSur} for a survey. See also \cite{budd2022irreducible,Wu2023}. 

\subsection*{Convention}
We collect some equalities relating different parameters that will be used in this paper for readers' convenience. 
Here, $\gamma$ is a parameter that we take in $(0, 2)\backslash\{\sqrt2\}$, and $\beta\in \bR$ refers to the strength of an additive log singularity (at the boundary) of a distribution.
\begin{equation}
    Q = \frac{\gamma}{2}+\frac{2}{\gamma}, \quad W = \gamma \ab(\gamma + \frac{2}{\gamma}-\beta);
\end{equation}
\begin{equation}\label{eq:ig-parameters}
    \kappa = \gamma^2, \quad \kappa'=\frac{16}{\kappa}, \quad \chi = \frac{2}{\sqrt{\kappa}}-\frac{\sqrt{\kappa}}{2}, \quad \lambda = \frac{\pi}{\sqrt{\kappa}}, \quad \lambda'=\frac{\pi\sqrt{\kappa}}{4};
\end{equation}
\begin{equation}
    \theta = \frac{\pi\gamma^2}{4}.
\end{equation}
Throughout the paper, we will  use $\|\mu\|$ to denote the total mass of a measure $\mu$, and use $\overset{d}{=}$ to denote equal in distribution. A curve is a map $\gamma:I\to \bC$ for some interval $I$. 

In this paper we work with non-probability measures and extend the terminology of ordinary probability to this setting. For a finite or $\sigma$-finite  measure space $(\Omega, \mathcal{F}, M)$, we say $X$ is a random variable if $X$ is an $\mathcal{F}$-measurable function with its \textit{law} defined via the push-forward measure $M_X=X_*M$. In this case, we say $X$ is \textit{sampled} from $M_X$ and write $M_X[f]$ for $\int f(x)M_X(dx)$. \textit{Weighting} the law of $X$ by $f(X)$ corresponds to working with the measure $d\tilde{M}_X$ with Radon-Nikodym derivative $\frac{d\tilde{M}_X}{dM_X} = f$.     {\emph{Restricting} to some event  $E\in\mathcal{F}$ refers to the measure $M[E\cap\cdot]$, and \emph{conditioning on}   $E$ where $M[E]\in (0,\infty)$  refers to the measure $M[E\cap\cdot]/M[E]$. }

For a planar map $G$, we will use $V(G)$, $E(G)$ and $F(G)$ to represent the set of vertices, edges, and faces, respectively, of $G$. For a finite set, we always use $\abs{ \cdot }$ to represent its cardinality.

\medskip

\textbf{Organization of the paper.}
The rest of the paper is organized as follows. In Section~\ref{sec: prelim}, we gather some preliminaries on GFF, SLE, LQG and the Brownian path measure. In Section~\ref{sec: random conditioned orientation},  we define the annular oriented maps to be considered in this paper and prove Theorem~\ref{thm: annular orientation scaling limit}. Sections~\ref{sec: resampling property of scaling limit}-\ref{sec: resampling characterization scaling limit} are devoted to the proof of Theorem~\ref{thm: quantum annulus mot}, where we first review the mating of trees for quantum triangles and prove a resampling property for the conformal welding in Figure~\ref{fig:weld-annulus-8} in Section~\ref{sec: resampling property of scaling limit}, and then work with the imaginary geometry in the annulus in Section~\ref{sec: IG annulus}, and finally conclude the proof of Theorem~\ref{thm: quantum annulus mot} in Section~\ref{sec: resampling characterization scaling limit} via a resampling argument. Finally in Section~\ref{sec: law of the random modulus}, we calculate the law of the random modulus, prove Theorem~\ref{thm: law of the random modulus}, and further discuss the semiclassical limit $\gamma\downarrow 0$. 

\medskip

\textbf{Acknowledgements.}
We thank Morris Ang and Zijie Zhuang for discussion during the early stage of this project. 
X.D.\ was supported by the NSFC NYTP20220903, MOST 2021YFA1002700, and the NYU Shanghai Boost Funds. 
X.D., N.H.\ and Z.T.\ were 
supported by grant DMS-2246820 of the National Science Foundation. 
N.H.\ and Z.T.\ were supported by the Simons Collaboration Grant on 
Probabilistic Paths to Quantum Field Theory.

\section{Preliminaries}
\label{sec: prelim}

\subsection{Gaussian free fields}
\label{sec: prelim gff}
We first give a brief review of the two-dimensional Gaussian free field (GFF). Readers are referred to  e.g.~\cite{GFFmathematicians,werner2021lecture,berestycki2024gaussian}  for more background.
 Let $D\subsetneq \bC$ be a domain that is conformally equivalent to the unit disk or the annulus $\cA_\tau$ for some $\tau\in (0, \infty)$. 
For $f$, $g\in C^\infty(D)$, define 
the inner product 
\begin{equation}\label{eq:inner}
    \langle f, g\rangle = (2\pi)^{-1}\int_D (\nabla f\cdot\nabla g)\d x.
\end{equation}

Let $H_0(D)$ be the completion of compactly supported smooth functions on $D$, and $\overline{H}(D)$ be the completion of $C^\infty(D)/\mathord\sim$, where $\sim$ identifies functions up to an additive constant. 
Let $(f_j)_{j\in\bN}$ be an orthonormal basis of $\overline{H}(D)$, let $(g_j)_{j\in\bN}$ be an orthonormal basis of $H_0(D)$, and let 
$(\alpha_j)_{j\in\bN}$
be i.i.d.~standard normal variables.
The random distribution $\sum_{j\in \mathbb N} \alpha_j f_j$ is called a \emph{free boundary GFF} on $D$, defined up to an additive constant, and $\sum_{j\in \mathbb N} \alpha_j g_j$ a \emph{zero boundary GFF} on $D$. Moreover, for a deterministic harmonic function $\fh$ on $D$, we say that $\sum_{j\in\bN} \alpha_j g_j + \fh$ is a \emph{Dirichlet GFF} with boundary data specified by $\fh$.

One way to fix the additive constant for a free boundary GFF $h$ is to require $\int_D h \d\rho=0$ where  
$\rho(\d x)$ is a compactly supported probability measure on $\overline{D}$ such that
\begin{equation}
    \label{eq: log integrability}
    \int_{\overline{D}\times \overline{D}} -\log\ab|x-y|\rho(\d x)\rho(\d y)<\infty.
\end{equation}
Equivalently, let $H(D; \rho)$ be the completion of $\{f\in C^\infty(D) : \int_D f \d\rho=0\}$ and let $(f^\rho_j)_{j\in\bN}$ be an orthonormal basis. Then $\sum_{j\in \bN} \alpha_j f^\rho_j$ is called a \emph{free boundary GFF} on $D$ normalized by $\rho$.

We also consider a mixed type of boundary data. Following the construction in e.g.\ \cite[Section 2.1]{ModAnn}, 
let $\partial_1 D$ be a non-empty finite union of open or closed subset of $\partial D$, and $\partial_2 D = \partial D \setminus \partial_1 D$. Write $H(D; \partial_1 D)$ for the Hilbert completion of
\begin{equation*}
    \{f\in C^\infty(D): \supp(f) \cap \partial_1 D = \emptyset\},
\end{equation*}
under the inner product~\eqref{eq:inner} and let $(h_j)_{j\in \mathbb N}$ be an orthonormal basis of $H(D; \partial_1 D)$.
The random distribution $\sum_j \alpha_j h_j$ is called a \emph{Dirichlet--Neumann GFF} with zero boundary data on $\partial_1 D$ and free boundary data on $\partial_2 D$.

We will use the following version of the Markov property of the GFF. Let $\tau>0$, $D = \cA_\tau $, $\rho$ be the uniform measure on $\{0\}\times[-1,-e^{-2\pi\tau}]$ and $h$ be a free boundary GFF on $D$ normalized such that $\int_D h \d \rho=0$. Let $A =D\backslash \{re^{i\theta}: r\in [e^{-2\pi\tau},1],\theta\in[5\pi/4,7\pi/4] \}$. Let $\mathrm{Harm}(D,A,\rho)$ be the space of functions in  $H(D,\rho)$ that are harmonic on $A$ and have normal derivative zero on $\partial D\cap\partial A$. Let $\partial_1A = \partial A\backslash\partial D$. The following is from the standard argument from~\cite[Section 2.6]{GFFmathematicians}.

\begin{lemma}\label{lem:GFF-Markov-0}
 Let $D,h,A,\rho$ be described as above.   We  have the orthogonal decomposition
    \begin{equation}
        H(D) = H(A;\partial_1A)\oplus \mathrm{Harm}(D,A,\rho).
    \end{equation}
    Furthermore, let $h^{\mathrm{har}}$ be the harmonic  extension of $h|_{D\backslash A}$ onto $A$ with zero normal derivative on $\partial A\cap\partial D$ and let $h^A = h-h^{\mathrm{har}}$. Then $h^A$ is a GFF on $A$ with zero boundary conditions on $\partial A\backslash\partial D$ and free boundary conditions on $\partial A\cap \partial D$. Moreover, $h^A$ is independent from $h^{\mathrm{har}}$.
\end{lemma}

Next we recall the definition of  \emph{local sets} from \cite{SS13}. Following \cite[Section 3.3]{SS13}, we let $d_*$ be the metric on $D$ induced by the Euclidean metric when we conformally send $D$ to the unit disk or $\cA_\tau$ for some $\tau\in\bR$, let $\Gamma(D)$ be the space of closed (with respect to $d_*$) nonempty subsets of $\overline{D}$, endowed with the Hausdorff metric induced by $d_*$. Let $h$ be a Dirichlet GFF on $D$. Given $A\in \Gamma(D)$, let $A_\delta$ denote the closed set containing all points in $D$ whose $d_*$-distance from $A$ is at most $\delta$, $\fA_\delta$ the smallest $\sigma$-algebra in which $A$ and the restriction of $h$ to the interior of $A_\delta$ are measurable, and $\fA = \cap_{\delta\in \bQ_{>0}} \fA_\delta$.

\begin{definition}
    \label{defn: local set Dirichlet GFF}
    In the setting above, suppose $(h, A)$ is a random variable which is the coupling of an instance $h$ of a  GFF with a random element $A\subseteq D$ of $\Gamma(D)$. We say that $A$ is a \emph{local set} for $h$ if the following is true. Conditioned on $\fA$, (a regular version of) the conditional law of $h$ is that of $h_0+\fh$ where $h_0$ is the GFF with zero boundary data on $D\setminus A$ and the original boundary data elsewhere, and $\fh$ is an $\fA$-measurable random distribution 
    which is a.s.~harmonic on $D\setminus A$.
\end{definition}

\subsection{SLE and Imaginary geometry}
\label{sec: prelim ig}

	The SLE$_\kappa$ curves are introduced in \cite{SchrammSLE}. For a curve $\eta$ on the upper half plane starting from 0,  let $H_t$ be the unbounded connected component of $\bbH\backslash \eta([0,t])$, and call $K_t:=\bbH\backslash H_t$ the \emph{hull} of $\eta$ at time $t$. For $\kappa>0$, $\SLE_\kappa$ is a conformally invariant measure on continuously growing   curves $\eta$   with the Loewner driving function $W_t=\sqrt{\kappa}B_t$ (where $B_t$ is the standard Brownian motion).   Then  the $\SLE_\kappa$ curve from 0 to $\infty$ on the upper half plane  can be described by 
	\begin{equation}\label{eqn-def-sle}
	g_t(z) = z+\int_0^t \frac{2}{g_s(z)-W_s}ds, \ z\in\mathbb{H},
	\end{equation} 
	where $g_t$ is the unique conformal transformation from $H_t$ to $\mathbb{H}$ such that $\lim_{|z|\to\infty}|g_t(z)-z|=0$. $\SLE_\kappa$ is scale invariant, and hence its definition can be extended to other simply connected domains via conformal maps. 
	
	SLE$_\kappa$ curves also has a natural variant called SLE$_\kappa(\underline{\rho})$, which first appeared in \cite{LSW03,Dub05} and was studied in \cite{ImagGeo1}. 
Fix force points $x^{k,L}<...<x^{1,L}<x^{0,L}= 0^-<x^{0,R}= 0^+< x^{1,R}<...<x^{\ell, R}$ and weights $\rho^{i,q}\in\bbR$, The $\SLE_\kappa(\underline{\rho})$ process is defined in the same way as $\SLE_\kappa$, except that its Loewner driving function $(W_t)_{t\ge 0}$ is now defined by 
\begin{equation}\label{eq:def-sle-rho}
\begin{split}
&W_t = \sqrt{\kappa}B_t+\sum_{q\in\{L,R\}}\sum_i \int_0^t \frac{\rho^{i,q}}{W_s-g_s(x^{i,q})}ds,
\end{split}
\end{equation}
where $B_t$ is standard Brownian motion. Here if $x^{i,\rmL}$ (resp.\ $x^{i,\rmR}$) is not on $\partial H_t$, then we define $g_t(x^{i,\rmL})$ (resp.\ $g_t(x^{i,\rmR})$) to be the $g_t$-image of the leftmost (resp.\ rightmost) point of $K_t\cap\bbR$.
It has been proved in \cite{ImagGeo1} that the SLE$_\kappa(\underline{\rho})$ process a.s.\ exists, is unique and generates a continuous curve until the \textit{continuation threshold}, the first time $t$ such that $W_t = g_t(x^{j,q})$ with $\sum_{i=0}^j\rho^{i,q}\le -2$ for some $j$ and $q\in\{L,R\}$. 
    In the rest of the paper, we let $\kappa\in (0, 4)$  and recall the conventions~\eqref{eq:ig-parameters}. 
	
	 Now we recall the notion of \emph{the  GFF flow lines}. Heuristically, given a GFF $h$, $\eta$ is a flow line of angle $\theta$ if
	\begin{equation}
	 {\partial_t\eta(t)} = e^{i(\frac{h(\eta(t))}{\chi}+\theta)}\ \text{for}\ t>0.
	\end{equation} 
	Rigorously, the GFF flow lines are defined via the following result from~\cite[Theorem 1.1, Theorem 1.2]{ImagGeo1}. Also recall the notion of GFF local sets in Definition~\ref{defn: local set Dirichlet GFF}.
	
	\begin{theorem}\label{thm-ig1}
	 	Fix $\kappa>0$, a vector $\underline{\rho}$ of weights and a vector $\underline{x}$ of force points. Let $(K_t)_{t\ge 0}$ be the hull at time $t$ of the SLE$_\kappa(\underline{\rho})$ process $\eta$ described by the Loewner flow \eqref{eqn-def-sle} with $(W_t)_{t\geq0}$ solving \eqref{eq:def-sle-rho}. Let $\mathfrak{h}_t^0$ be the harmonic function on $\mathbb{H}$ with boundary values 
	 	$$
	 	-\lambda(1+\sum_{i=0}^j \rho^{i,L})\ \ \text{on} \ [V_t^{j+1, L}-W_t, V_t^{j,L}-W_t)\ \  \text{and}\ \ \lambda(1+\sum_{i=0}^j \rho^{i,R})\ \text{on}\ \ [V_t^{j, R}-W_t, V_t^{j+1,R}-W_t)
	 	$$
	 	where $\rho^{0,R} = \rho^{0,L}=0$, $x^{0,L} = 0^-, x^{0,R} = 0^+, x^{k+1, L} = -\infty, x^{\ell+1, R} = +\infty$, $V_t^{0,L} = g_t(0^-)$ and $V_t^{0,R} = g_t(0^+)$. Set $\mathfrak{h}_t(z) = \mathfrak{h}_t^0( {g_t}(z))-\chi\arg  {g_t'}(z)$. Let $\mathcal{F}_t$ be the filtration generated by $(W, V^{i,q})$. Then there exists a coupling $(K,h)$ where $h = \tilde{h}+ \mathfrak{h}_0$ with $\tilde{h}$ being a zero boundary GFF on $\mathbb{H}$ such that the following is true. For any $\mathcal{F}_t$-stopping time $\tau$ before the continuation threshold, $K_\tau$ is a local set for $h$ and the conditional law of $h|_{\mathbb{H}\backslash K_\tau}$ given $\mathcal{F}_\tau$ is the same as the law of $\mathfrak{h}_\tau+\tilde{h}\circ  {g_\tau}$. Moreover, the curve $\eta$ is measurable with respect to $h$.
	 \end{theorem}
	 Following~\cite{ImagGeo1}, a chordal $\SLE_\kappa(\underline\rho)$ curve $\eta$ with $\kappa\in(0,4)$ coupled with the GFF $h$ as above is called a \emph{flow line} of $h$, while a chordal $\SLE_{\kappa'}(\underline\rho)$ curve $\eta$ with $\kappa'>4$ coupled with the GFF  {$-h$} is called a \emph{counterflow line} of $h$. For $\theta\in\bbR$, we say $\eta$ is a \emph{flow line of $h$ of angle $\theta$}, if $\eta$ can be coupled with $h+\chi\theta$ in the above way.
	 One important consequence  is that, as argued in \cite[Section 6]{ImagGeo1}, given $\eta$, the field $h$ in each component of $\bbH\backslash\eta$ is a Dirichlet GFF with the \emph{flow line boundary conditions}; see Figure~\ref{fig: flow line boundary data} and also e.g.\ \cite[Figures 1.10 and 1.11]{ImagGeo1}.

     \begin{figure}[ht]
    \centering
    \includegraphics[height=4cm, page=1]{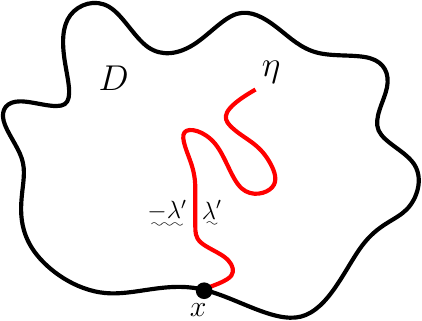}
    \hspace{2cm}
    \includegraphics[height=4cm, page=2]{img/flowline.pdf}
    \caption{Boundary data modulo $2\pi\chi$ of $\fh$ on two sides of $\eta$. The symbol $\uwave{\lambda'}$ is the shorthand for $\lambda' + \chi\cdot\text{winding}$. In other words, whenever $\eta$ makes a quarter turn to the left, the boundary data of $\fh$ goes up by $\frac{\pi}{2}\chi$, and whenever $\eta$ makes a quarter turn to the right, the boundary data of $\fh$ goes down by $\frac{\pi}{2}\chi$. The same applies to $\uwave{-\lambda'}$.} 
    \label{fig: flow line boundary data}
\end{figure}

One can also define GFF flow lines started from interior points; see~\cite[Theorems 1.1 and 1.2]{ImagGeo4} for a precise statement.  Let $h$ be a GFF on $\bbH$ with piecewise constant boundary conditions. We work with the setting where $h$ is equal to $\lambda'$ on $[y^\rmL,0]$ and $-\lambda'+\pi\chi$ on $(-\infty,y^\rmL)$, and equal to $-\lambda'$ on $(0,y^\rmR]$ and $\lambda'-\pi\chi$ on $(y^\rmR,\infty)$, with $y^\rmL\in (-\infty,0)\cup\{0^-\}$ and $y^\rmR\in (0,\infty)\cup\{0^+\}$. Flow lines of ${h}$ with the same angle started at different points of $\mathbb{Q}^2\cap\bbH$ merge into each other when intersecting and form a tree \cite[Theorem 1.9]{ImagGeo4}. This gives an ordering of $\mathbb{Q}^2\cap\bbH$, where $x\preceq y$ whenever {the $\frac{\pi}{2}$-angle flow line} from $x$ merges into the $\theta$-angle flow line from $y$ on the left side. Then one can construct a Peano curve which visits points of $\mathbb{Q}^2$ with respect to this ordering. We call this curve the \textit{space-filling   counterflow line  of ${h}$.} These curves are space-filling versions of  $\SLE_{\kappa'}(\frac{\kappa'}{2}-4;\frac{\kappa'}2-4)$ processes with force points located at $y^\rmL$ and $y^\rmR$, which agrees with the ordinary version of $\SLE_{\kappa'}(\frac{\kappa'}{2}-4;\frac{\kappa'}2-4)$ when $\kappa'\geq 8$. One can also work with the case where $h$ is a \emph{whole-plane GFF modulo $2\pi\chi\bbZ$}, and the curve generated in this way is the space-filling whole-plane $\SLE_{\kappa'}$. See~\cite[Section 4.3]{ImagGeo4} for more details.

	 Finally, we can  extend the notion of GFF flow lines to other simply connected domains as explained in~\cite{ImagGeo1}. For a conformal map $f:D\to\bbH$ and a GFF $h$ on $\bbH$, if $\eta$ is a flow line of $h$, then $f^{-1}\circ \eta$ is a flow line of the field 
	  \begin{equation}\label{eq:ig-change-coord}
	      h\circ f - \chi\arg f'.
	   \end{equation}

\subsection{Liouville fields and LQG surfaces}
\label{sec: LQG surfaces} 
We now review LQG surfaces with disk topology and annular topology. Let $\gamma \in (0, 2)$ and recall that $Q = \frac{\gamma}{2}+\frac{2}{\gamma}$.
Let $D\subseteq \bC$ be a planar domain.
For conformal maps $\psi:D'\to D$ and a generalized function $\phi$ on $D$, we define the \emph{LQG pullback} by
\begin{equation}
    \label{eq: LQG field pullback}
    \psi^*_\gamma\phi = \phi \circ \psi + Q\log\ab|\psi'|,
\end{equation}
which defines the following equivalence relation on tuples with $z_i\in \overline{D}$
\begin{equation}
    \label{eq: LQG surface equivalence}
    (D, \phi, \{z_i\}_{i=1}^n) \sim_\gamma (D', \psi^*_\gamma\phi, \{\psi^{-1}(z_i)\}_{i=1}^n).
\end{equation}
The equivalence class $(D, \phi, \{z_i\}_{i=1}^n)/\mathord\sim_\gamma$ is called a \emph{$\gamma$-LQG surface} (or just \emph{quantum surface}) with $n$ marked points. A particular tuple in 
the equivalence class is called an \emph{embedding} of the LQG surface.

We consider the \emph{LQG area measure} and the \emph{LQG length measure}, formally defined to be $e^{\gamma\phi} \d^2 z$ and $e^{\gamma\phi/2}\d z$, respectively, where $\d^2 z$ is the Lebesgue measure on $D$ and $\d z$ is the Lebesgue measure on $\partial D$. This can be made rigorous by taking the limit $\lim_{\e\to 0}\e^{\frac{\gamma^2}{2}}e^{\phi_\e(z)}$ and $\lim_{\e\to 0}\e^{\frac{\gamma^2}{4}}e^{\phi_\e(z)/2}$ if $\phi$ is sampled from a variant of the GFF on $D$, where $\phi_\e(z)$ is the average of $\phi$ over $\partial B(z;\e)\cap\bbH$. See \cite{berestycki2024gaussian} for a   review. 
The definition \eqref{eq: LQG field pullback} ensures the invariance of the LQG area measure and the LQG length measure under LQG pullback. We will write $\mu_\phi$ (resp.~$\nu_\phi$) for the LQG area (resp.~length) measure associated to $\phi$.

Let $\bP_{\bH}$ be the law of a free boundary GFF on $\bH$, where the additive constant is fixed by taking $\rho$ to be the uniform measure on the unit semicircle $\partial\bbD\cap\bbH$ in the definition.
For $z\in \bC$, define $|z|_+ = \max\{|z|, 1\}$.
The free boundary GFF has a covariance kernel 
\begin{equation}
    G_{\bH}(z, w) := \bE[h_{\bH}(z) h_{\bH}(w)] = -\log\ab|z-w|-\log\ab|z-\overline{w}| +2\log\ab|z|_+ +2\log\ab |w|_+.
\end{equation}
We use the convention $G_{\bH}(\infty, w)=2\log\ab|w|_+$.

\begin{definition}[Liouville field on $\bH$]\label{def:LF-HH}
    Sample $(h, \bfc)$ from the infinite measure $\bP_\bH\times \Leb_{\bR}$, where $\Leb_\bR$ is the Lebesgue measure on $\bR$, and take $\phi = h+\bfc$. We write $\LF_{\bH}$ for the law of $\phi$. 
\end{definition}

We also consider the Liouville field with boundary insertions on $\bH$. For $m > 0$, let $(\beta_i, s_i)\in \bR\times (\bR\cup(\infty))$ for $i=1, \dots, m$ where the $s_i$ are distinct, and $s_i\neq\infty$ for $i\ge 2$. We formally define $\LF_{\bH}^{(\beta_i, s_i)_{i=1}^m}(\d\phi)$ by $\exp(\sum_{i=1}^m \beta_i \phi(s_i)/2)\LF_\bH(\d\phi)$, which leads t the following rigorous definition; see ~\cite[Definitions 2.5, 2.8 and Lemma 2.6]{AHS24}.

\begin{definition}[Liouville field with boundary insertions on $\bH$]
    In the setting above, sample $(h, \bfc)$ from the infinite measure
    \begin{equation}
        C_{\bH}^{(\beta_i, s_i)_{i=1}^m}
        \bP_{\bH}
        \times \ab[ e^{(\frac{1}{2}\sum_{i=1}^m\beta_i-Q)c}\Leb_{\bR}(\d c)],
    \end{equation}
    where
    \begin{equation}
        C_{\bH}^{(\beta_i, s_i)_{i=1}^m} = 
        \begin{cases}
        \prod_{i=1}^m \ab|s_i|_+^{-\beta_i(Q-\frac{\beta_i}{2})} \exp\ab(\frac{1}{4}\sum_{j=i+1}^m \beta_i\beta_j G_{\bH}(s_i, s_j)) \text{ if } s_1\neq \infty\\
        \prod_{i=2}^m \ab|s_i|_+^{-\beta_i(Q-\frac{\beta_i}{2}-\frac{\beta_1}{2})} \exp\ab(\frac{1}{4}\sum_{j=i+1}^m \beta_i\beta_j G_{\bH}(s_i, s_j)) \text{ if } s_1= \infty
        \end{cases},
    \end{equation}
    and take
    \begin{equation}
        \label{eq: Liouville field}
        \phi(z) = h(z)-2Q\log\ab|z|_+ +\frac{1}{2}\sum_{i=1}^m \beta_i G_{\bH}(s_i, z)+\bfc.
    \end{equation}
    We write $\LF_{\bH}^{(\beta_i, s_i)_{i=1}^m}$ for the law of $\phi$.
\end{definition}

Note that if $\beta_i\geq Q$, following  \cite[Theorem 3.1]{huang2018liouville}, the LQG length measure with respect to $\phi$ has infinite mass near $s_i$. On the other hand,   the limit  $\lim_{\beta_3\uparrow Q}\frac{1}{Q-\beta_3}\LF^{(\beta_1,s_1),(\beta_2,s_2),(\beta_3, s_3)}_{\bH}$ exists as measures on the space of distributions, where we equip the space of distributions with the weak-$*$ topology from testing against smooth compactly supported functions. The limiting field, whose law is denoted by $\LF^{(\beta_1,s_1),(\beta_2,s_2),(Q^-, s_3)}_{\bH}$, has finite LQG boundary length near $s_3$, and   in this case we say the Liouville field has a $Q^-$ insertion at the point $s_3$. The same can be defined analogously if one or more $\beta_1,\beta_2,\beta_3$ is $Q^-$. See~\cite[Definition 2.28 and Proposition 2.32]{QuanTrig} for more details.

\begin{definition}[\cite{QuanTrig}]
    \label{defn: quantum triangle}
    Fix $W_1,W_2, W_3 \ge \gamma^2/2$, and let $\beta_i= \gamma+\frac{2-W_i}{\gamma}$ if $W_i\neq \frac{\gamma^2}{2}$ and $\beta_i=Q^-$ if $W_i=\frac{\gamma^2}{2}$.
    Consider a distribution $\phi$ sampled from 
    \begin{equation*}
        \prod_{\beta_i\neq Q^-} (Q-\beta_i)^{-1} \LF_{\bH}^{(\beta_1,0), (\beta_2, 1),(\beta_3, \infty)}.
    \end{equation*}
    The infinite measure $\QT(W_1,W_2,W_3)$ is defined to be the law of $(\bH, \phi, 0, 1, \infty)/\mathord\sim_\gamma$. A sample from $\QT(W_1,W_2,W_3)$ is called a quantum triangle with weights $W_1$, $W_2$ and $W_3$. 
\end{definition}

We will also need the following definition of quantum disks. Let $H(\bbH,\rho)$ be as in Section~\ref{sec: prelim gff} where $\rho$ is the uniform measure on $\partial\bbD\cap\bbH$.  This space admits a natural decomposition $H(\bbH,\rho) = H_1(\bbH)\oplus H_2(\bbH)$, where $H_1(\bbH)$ (resp.\ $H_2(\bbH)$) is the set of functions in $H(\bbH)$ with constant value (resp.\ average zero) on the semicircle $\{z\in\bbH:\ |z| = r\}$ for each $r>0$.

\begin{definition}[Thick quantum disk]\label{def-quantum-disk}
Fix a weight parameter $W\ge\frac{\gamma^2}{2}$ and let $\beta = \gamma+ \frac{2-W}{\gamma}\le Q$. 
		Let $(B_s)_{s\ge0}$ and $(\wt{B}_s)_{s\ge0}$ be independent standard one-dimensional Brownian motions conditioned on  $B_{2t}-(Q-\beta)t<0$ and  $ \wt{B}_{2t} - (Q-\beta)t<0$ for all $t>0$.   Let $\mathbf{c}$ be sampled from the infinite measure $\frac{\gamma}{2}e^{(\beta-Q)c}dc$ on $\bbR$ independently from $(B_s)_{s\ge0}$ and $(\wt{B}_s)_{s\ge0}$.
			Let 	
			\begin{equation*}
				Y_t=\left\{ \begin{array}{rcl} 
					B_{2t}+\beta t+\mathbf{c} & \mbox{for} & t\ge 0,\\
					\wt{B}_{-2t} +(2Q-\beta) t+\mathbf{c} & \mbox{for} & t<0.
				\end{array} 
				\right.
			\end{equation*}
			 Let $h$ be a free boundary  GFF on $\mathbb{H}$ independent of $(Y_t)_{t\in\bbR}$ with projection onto $H_2(\mathbb{H})$ given by $h_2$. Consider the random distribution
			\begin{equation*}
				\psi(\cdot)=Y_{-\log|\cdot|} + h_2(\cdot) \, .
		\end{equation*}
		Let the infinite measure $\mathcal{M}_{2}^{\mathrm{disk}}(W)$ be the law of $({\mathbb{H}}, \psi,0,\infty)/\mathord\sim_\gamma $. 
		We call a sample from $\mathcal{M}_{2}^{\textup{disk}}(W)$ a \emph{quantum disk} of weight $W$ with two marked points.
		
		We call $\nu_\psi((-\infty,0))$ and $\nu_\psi((0,\infty))$ the left and right{, respectively,} quantum   {boundary} length of the quantum disk  $(\mathbb{H}, \psi, 0, \infty)/{\sim_\gamma}$.
	\end{definition}

    Next we define quantum disks and quantum triangles with thin weights.
	\begin{definition}[Thin quantum disk]\label{def-thin-disk}
	For $W\in(0, \frac{\gamma^2}{2})$, the infinite measure $\mathcal{M}_{2}^{\textup{disk}}(W)$ is defined as follows. First sample a random variable $T$ from the infinite measure $(1-\frac{2}{\gamma^2}W)^{-2}\textup{Leb}_{\mathbb{R}_+}$. Then sample a Poisson point process $\{(u, \mathcal{D}_u)\}$ with intensity $\mathds{1}_{t\in [0,T]}dt\times \mathcal{M}_{2}^{\textup{disk}}(\gamma^2-W)$. Finally consider the ordered (according to the order induced by $u$) collection of doubly-marked thick quantum disks $\{\mathcal{D}_u\}$, called a \emph{thin quantum disk} of weight $W$.
		
		Let $\mathcal{M}_{2}^{\textup{disk}}(W)$ be the law of this ordered collection of doubly-marked quantum disks $\{\mathcal{D}_u\}$.
		The left (resp.\ right) boundary length of a sample from $\mathcal{M}_{2}^{\textup{disk}}(W)$ is defined to be the sum of the left (resp.\ right) boundary lengths of the quantum disks $\{\mathcal{D}_u\}$. A surface sampled from $\mathcal{M}_{2}^{\textup{disk}}(W)$ has two marked points; see Figure~\ref{fig-qt} (left). 
	\end{definition}

    \begin{definition}[Quantum triangles with  {possibly} thin vertices]\label{def-qt-thin}
	Fix $W_1, W_2, W_3>0$. Let $I:=\{i\in\{1,2,3\}:W_i<\frac{\gamma^2}{2}\}$. Let $\tilde W_i = W_i$ if $i \not \in I$ and $\tilde W_i = \gamma^2 - W_i$ if $i \in I$. Sample $(S_0, (S_i)_{i \in I})$ from 
 \[\QT(\tilde W_1, \tilde W_2, \tilde W_3) \times \prod_{i\in I } (1-\frac{2W_i}{\gamma^2}) \Md_2(W_i).  \]
 Embed $S_0$ as $(\tilde D, \phi, \tilde a_1, \tilde a_2, \tilde a_3)$, for each $i \not \in I$ let $a_i = \tilde a_i$, and for each $i \in I$ embed $S_i$ as $(\tilde D_i, \phi, \tilde a_i, a_i)$ in such a way that the $\tilde D_i$ are disjoint and $\tilde D_i \cap \tilde D = \tilde a_i$. Let $D = \tilde D \cup \bigcup_{i \in I} \tilde D_i$ and let $\QT(W_1, W_2, W_3)$ be the law of $(D, \phi, a_1, a_2, a_3)/{\sim_\gamma}$.
\end{definition} 

\begin{figure}[ht]
	\centering
    \begin{tabular}{cc}
    \includegraphics[scale=0.43]{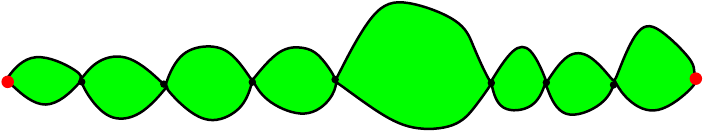}     &  \includegraphics[scale=0.43]{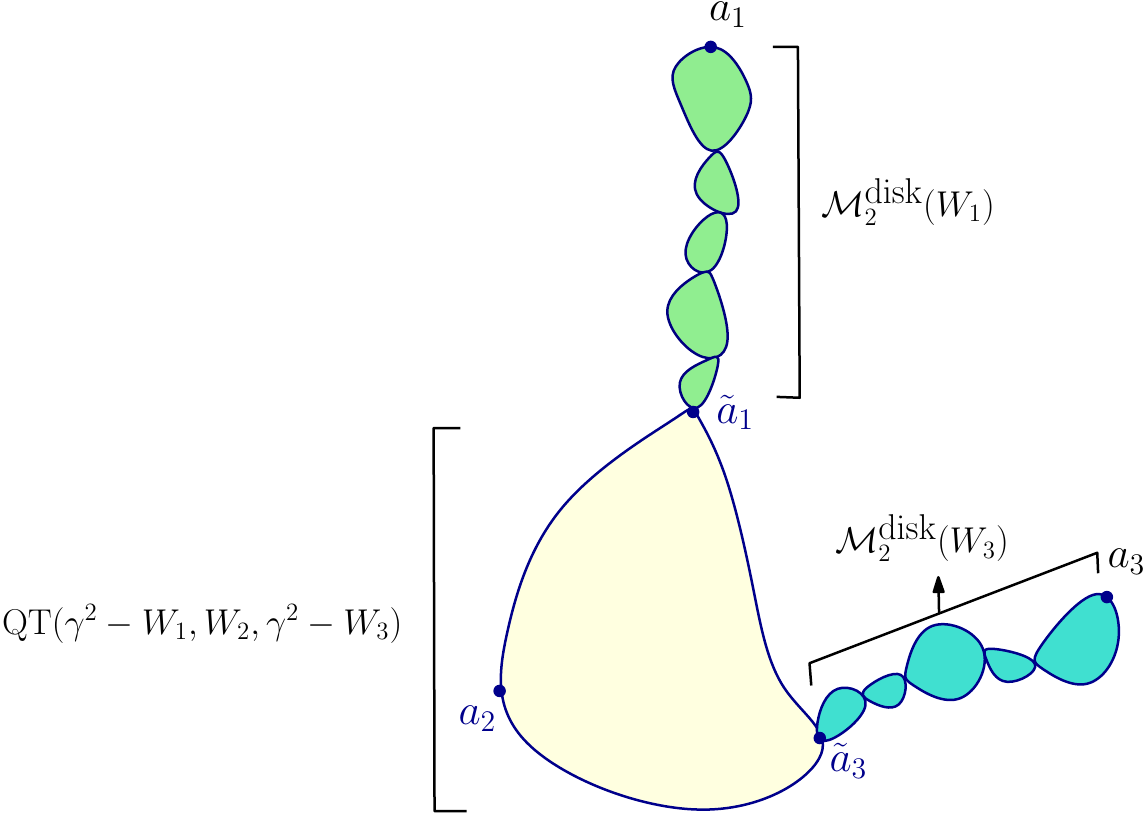}
    \end{tabular}
	
	\caption{\textbf{Left}: A thin quantum disk of weight $W\in(0,\frac{\gamma^2}{2})$. Note that there are infinitely many disks between the different beads and near the two  marked points (red). \textbf{Right:} A quantum triangle with $W_2\geq\frac{\gamma^2}{2}$ and $W_1,W_3<\frac{\gamma^2}{2}$ embedded as $(D,\phi,a_1,a_2,a_3).$  The two thin disks (colored green) are concatenated with the thick triangle (colored yellow) at points $\tilde{a}_1$ and $\tilde{a}_3$.}\label{fig-qt}
\end{figure}

    We will also work with LQG surfaces with fixed boundary lengths. For instance, for quantum disks we have 
\begin{equation}\label{eq:QD-dis}
 \mathcal{M}_{2}^{\textup{disk}}(W)=\int_0^\infty\int_0^\infty    \mathcal{M}_{2}^{\textup{disk}}(W;\ell_1,\ell_2)\,d\ell_1\,d\ell_2,
\end{equation}
where $\mathcal{M}_{2}^{\textup{disk}}(W;\ell_1,\ell_2)$ is supported on the set of quantum disks with left boundary $\ell_1$ and right boundary $\ell_2$. The quantum triangles $\QT(W_1,W_2,W_3;\ell_1,\ell_2;\ell_3)$ with boundary arc lengths $\ell_1,\ell_2,\ell_3$ can be defined analogously. These measures can be defined for a.e.\ -$\ell$ by standard disintegration and extended to every $\ell$ in a continuous way; see~\cite[Section 2.6]{ConfWeldDisks} for details.

We now turn to domains with annular topology. Let $\bP_{\tau}$ be the law of the free boundary GFF on $\cA_\tau$, with the additive constant fixed by requiring the field to have zero average on $\{0\}\times [-1,-e^{-2\pi \tau}]$.

\begin{definition}[Liouville field on $\cA_\tau$]
    Sample $(h, \bfc)$ from the infinite measure $\bP_\tau\times \Leb_{\bR}$ and take $\phi=h+\bfc$. We write $\LF_\tau$ for the law of $\phi$.
\end{definition}

 We write $G_\tau$ for the covariance kernel of $h\sim \bbP_\tau$, i.e., $G_\tau(z,w) = \bbE[h(z)h(w)]$. The exact expression for $G_\tau$ is not needed. The following lemma can be proved similarly as~\cite[Lemma 2.6]{AHS24} by using the Girsanov theorem. 
\begin{lemma}
    Let $\beta\in\bbR$ and $s\in\partial \cA_\tau$. Let $\phi_\e(s)$ be the average of $\phi$ over $A_\tau\cap \partial B(s;\e)$. In the sense of vague convergence of measures on the space of distributions, the limit
    \begin{equation}\label{eq:def-LF-ann-marked-pt-0}
        \LF_\tau^{(\beta,s)}:=\lim_{\e\to0}\e^{\frac{\beta^2}{4}}e^{\frac{\beta}{2}\phi_\e(s)} \LF_\tau  
    \end{equation}
    exists. Furthermore, there exists some constant $C(\tau,s)>0$, such that if we sample $(h,\mathbf{c})$ from $C(\tau,s)\LF_\tau \times [e^{\beta c/2}\mathrm{Leb}_\bbR(dc)]$, then 
    \begin{equation}\label{eq:def-LF-ann-marked-pt}
        \phi(z)=h(z)+\frac{\beta}{2}  \beta_i G_{\tau}(s, z) + \bfc
    \end{equation}
    has the same law as $\LF_\tau^{(\beta,s)}$.
\end{lemma}
The precise value of the constant $C(\tau,s)$ has been computed in~\cite[Equation (3.14)]{remy2018liouville}, but we will not use it here. 
Following~\cite[Lemma 2.3]{ModAnn}, the  measure $\LF_\tau$ does not depend on the normalizing measure $\rho$. Therefore it follows from~\eqref{eq:def-LF-ann-marked-pt-0} that the measure  $\LF_\tau^{(\beta,s)}$ does not depend on $\rho$ either.

\begin{definition}
    Fix $\tau\in (0, \infty)$. We write $\QA_\tau$ for the law of the LQG surface $(\cA_\tau, \phi)/\mathord\sim_\gamma$ where $\phi\sim \LF_\tau$. Given an embedding  $(\cA_\tau, \phi)$ of $\QA_\tau$, we write $\cL_0^\phi$ and $\cL_1^\phi$ for the LQG length measure induced by $\phi$ on $C_0 = \{|z|=1\}$ and $C_\tau = \{|z|=e^{-2\pi\tau}\}$, respectively. We further weight the law of $\phi$ by  $\|\cL_0^\phi\|$, the LQG length of the outer circle, and sample a point $x\in C_0$ according to the probability measure proportional to the LQG length measure   $\nu_\phi$. We write $\QA_\tau^1$ for the law of the resulting LQG surface   $(\cA_\tau, \phi,x)/\mathord\sim_\gamma$.
\end{definition}

Following~\cite[Lemma 2.12]{ModAnn}, under the measure $\QA_\tau$, $\|\cL_0^\phi\|,\|\cL_1^\phi\|<\infty$ except for a zero measure set, so the above definition makes sense. Furthermore, for every compact interval $I\subsetneq (0,\infty)$, $\QA_\tau[\{ \|\cL_0^\phi\|\in I \}]<\infty$. One can check this by e.g., setting $f = \mathds{1}_I$ in ~\cite[Lemma 2.12]{ModAnn}.

The following result is an analog of \cite[Proposition 2.18]{AHS24} in the context of LQG annuli and the proof is similar. 
\begin{proposition}
    \label{prop: adding a gamma singularity to QA}
    Let $\phi$ be a sample from $ \LF^{(\gamma, -i)}_\tau$. Then there is some constant $C>0$, such that the law of $(\cA_\tau, \phi, -i)/\mathord\sim$ is $C \QA_\tau^1$.
\end{proposition}

The domain Markov property of the GFF implies a domain Markov property for $\LF_\tau$. 
We first define conditioning for infinite measures via Markov kernels.
\begin{definition}
    \label{defn: conditioning Markov kernel definition}
    Let $(\Omega, \mathcal{F})$ and $(\Omega', \mathcal{F}')$ be measurable spaces. We say $\Lambda: \Omega\times \mathcal{F}'\to [0, 1]$ is a Markov kernel if $\Lambda(\omega, \cdot)$ is a probability measure on $(\Omega', \mathcal{F}')$ for any $\omega\in\Omega$, and $\Lambda(\cdot, A)$ is $\mathcal{F}$-measurable for any $A\in\mathcal{F}'$. If $(X, Y)$ is a sample from $\Lambda(x, \d y)\mu(\d x)$ for a measure $\mu$ on $(\Omega, \mathcal{F})$, we say that $\Lambda(X, \cdot)$ is the conditional law of $Y$ given $X$. 
\end{definition}

\begin{lemma}\label{lem:GFF-Markov}
Let $D=\cA_\tau$ and $A$ be as in Lemma~\ref{lem:GFF-Markov-0}, and $\phi$ be sampled from $\LF_\tau^{(\gamma,-i)}$. Let $\phi^{\mathrm{har}}$ be the harmonic extension of $\phi|_{D\backslash A}$ onto $A$ with zero normal derivative on $\partial A\cap\partial D$. Let $\phi_A = \phi- \phi^{\mathrm{har}}$.   Then conditioned on $\phi|_{D\backslash A}$, the conditional law of  $\phi_A$ is a GFF on $A$ with zero boundary condition on $\partial A\backslash\partial D$ and free boundary condition on $\partial A\cap\partial D$.
\end{lemma}

\begin{proof}
    Let $\phi$ be constructed via $(h,\mathbf{c})$ as in~\eqref{eq:def-LF-ann-marked-pt}. By our choice of the normalization of $h$ in $\bbP_\tau$, the constant $\mathbf{c}$ is the average of $\phi$ over $\{0\}\times [-1,-e^{-2\pi\tau}]$ and is therefore measurable with respect to $\phi|_{D\backslash A}$. The claim therefore follows from the domain Markov property of $h$ as in Lemma~\ref{lem:GFF-Markov-0}.
\end{proof}

\subsection{Brownian path measure}
\label{sec: prelim Brownian path measure}
In this subsection we define the various 2D Brownian path measure that will be used in this paper. We start with the standard Brownian  motion case, and later use the shear matrix in~\eqref{eq: m theta definition} to transfer to the setting where the covariance matrix is~\eqref{eq: gamma correlated BM}. 

Let $\Gamma$ be the set of curves $Z: [0, t]\to \bR^2$, where $t>0$  is called the duration of $Z$. For curves $Z_1$ and $Z_2$ with duration $t_1$ and $t_2$, respectively, we define the uniform metric on $\Gamma$ by
\begin{equation}
    \label{eq: uniform metric curves}
    d(Z_1, Z_2) = \inf_\psi \sup_{t\in [0, t_1]}(|Z_1(t)-Z_2(\psi(t))|+|\psi(t)-t|),
\end{equation}
where the infimum is over all increasing bijiections $[0, t_1]\to [0, t_2]$.
This metric defines a topology and corresponding Borel $\sigma$-algebra on $\Gamma$. The Brownian path measure is a family of infinite measures indexed by $x$, $y\in\bR^2$ such that
\begin{equation}
    \label{eq: bpm defn}
    B_{x, y} = \int_0^\infty B_{x, y, s}^{\#} p_s(x-y)\d s,
\end{equation}
where $B_{x, y, s}^{\#}$  is the Brownian bridge (probability) measure from $x$ to $y$ with duration $s$, and $p_s(z)=(2\pi s)^{-1}e^{-|z|^2/2s}$ is the Poisson kernel on $\bR^2$ with time $s$.   

Let $D\subsetneq \bC$ be a simply connected domain. For simplicity we assume that $\partial D$ consists of a finite number of lines, half-lines and line segments. For $x$, $y\in D$, write $B_{x, y}^D$ for the restriction of $B_{x, y}$ to the curves that stay in $D$.

We define an interior-to-boundary measure as follows. For $x\in D$ and $y\in\partial D$, let $\mathbf{n}_y$ be the inward unit normal at $y$ into $D$. Let 
\begin{equation}
    \label{eq: BPM interior to boundary defn}
    B_{x, y}^D = \lim_{\epsilon\downarrow 0} \frac{1}{2\epsilon} B^D_{x, y+\epsilon\mathbf{n}_y}.
\end{equation}
The  boundary-to-boundary measure can be defined considering the limit twice, i.e., for $y,z\in\partial D$,
\begin{equation*}
    B_{y, z}^D = \lim_{\epsilon\downarrow 0} \frac{1}{2\epsilon} B^D_{y+\epsilon\mathbf{n}_y, z} = \lim_{\epsilon\downarrow 0} \frac{1}{4\epsilon^2} B^D_{y+\epsilon\mathbf{n}_y, z+\epsilon\mathbf{n}_z}.
\end{equation*}
It follows from~\cite[Section 3]{LW04} that the limits above are well-defined, and that the total mass $\|B^{D}_{x, y}\|$  is equal to the exit density at $y\in\partial D$ of a Brownian motion starting at $x \in D$.

For $\eta\in\Gamma$, let $\eta^R$ be the reversal of $\eta$. The Brownian path measures are by definition reversible, i.e., $B_{y, x}(C) = B_{x, y}(\{\eta\in \Gamma: \eta^R\in C\})$. We define boundary-to-interior measure by reversing the paths, i.e., for  $x\in D$ and $y\in\partial D$.
\begin{equation*}
    B_{y, x}^D(C) = B^D_{x, y}(\{\eta\in \Gamma: \eta^R\in C\}).
\end{equation*}

Now we turn to the $\gamma$-correlated Brownian motion $(L, R)$ as defined in~\eqref{eq: gamma correlated BM}, related to the standard Brownian motion $(X, Y)$ via the transform
\begin{equation}
    \label{eq: m theta definition}
    \begin{pmatrix}
        L \\ R
    \end{pmatrix} =
    M_\theta 
    \begin{pmatrix}
        X \\ Y
    \end{pmatrix},
    \qquad
    M_\theta :=
    (\sin\theta)^{-1/2}
    \begin{pmatrix}
        \sin \frac{\theta}{2} + \cos \frac{\theta}{2} 
        & \sin \frac{\theta}{2} - \cos \frac{\theta}{2}\\
        \sin \frac{\theta}{2} - \cos \frac{\theta}{2} 
        & \sin \frac{\theta}{2} + \cos \frac{\theta}{2}
    \end{pmatrix}.
\end{equation}

Let $m_\theta: \bR^2 \to \bR^2$ denote multiplication by $M_\theta$. It induces a map on curves $\Gamma(x, y)\to \Gamma(m_\theta(x), m_\theta(y))$ by composition. Consequently, given a path measure on $\Gamma(x, y)$, we may consider its pushforward measure on $\Gamma(m_\theta(x), m_\theta(y))$. Write $B^{\gamma, D}_{x, y}$ for the pushforward of $B^{m_\theta^{-1}(D)}_{m_\theta^{-1}(x), m_\theta^{-1}(y)}$.

This defines the same Brownian path measure as in, say, \cite{MoTboundary,ConfWeldDisks,ASYZ24}. Compared with their definition, we are using a different shear matrix 
in order to simplify later calculation. Note that any $M_\theta$ such that $M_\theta^2$ equals the desired covariance matrix of $L$ and $R$ will produce the same law on paths.

Let $B_1$, $B_2$ be measures on curves, where the ending points of samples from $B_1$ agrees with the starting points of samples from $B_2$. For  $(\eta_1,\eta_2)\sim B_1\times B_2$,  write $B_1\oplus B_2$ for the law of the concatenation $\eta_1\oplus \eta_2$. The following decomposition is essentially a Markov property for the Brownian path.

\begin{lemma}
    \label{lemma: first and last passage decomposition}
  For any $a,b,c>0$ and with $\mathbb{a}$ as in~\eqref{eq: gamma correlated BM},
     \begin{equation}
        \label{eq: B first passage decomposition}
        B^{\gamma, \bH}_{0, -a+ib} = B^{\gamma,  \bbR_+^2-c}_{0, -a+ib} + \mathbb{a} \int_0^\infty \big( B^{\gamma, \bbR_+^2-c}_{0, -c+i d} \oplus B^{\gamma, \bH}_{-c+i d, -a+ib} \big) \d d
    \end{equation}
    where $\bbR_+^2-c = (-c,\infty)\times(0,\infty)$.
\end{lemma}

\begin{proof}
 
Let $D = m_\theta^{-1}\bbH$, $S = m_\theta^{-1}(\{-c+id:d\geq 0\})$ and $D_0$ be the connected component of $D\backslash S$ containing 0. Let $y = m_\theta^{-1}(-a+bi)$. By
   ~\cite[Lemmas B.2, B.3]{MinBMCutPt},
    \begin{equation}\label{eq:pf-BM-decom-z}
        B^{D}_{0, y} = B^{D_0}_{0, y} + \int_S [B_{0, z}^{D_0} \oplus B^{D}_{z, y}]\d z.
    \end{equation}
    The claim~\eqref{eq: B first passage decomposition} then follows by applying the transform $m_\theta$ to~\eqref{eq:pf-BM-decom-z}, where the additional factor $\mathbb{a}$ is from change of variables $-c+id = m_\theta(z).$
\end{proof}

\section{Bijection and peanosphere convergence for annular map with orientation}
\label{sec: random conditioned orientation}

The goal of this section is to prove Theorem~\ref{thm: annular orientation scaling limit}. 
In Section~\ref{sec: annular orientations}, we define the pole-free almost acyclic annular maps. In Section~\ref{sec: bijection between maps and lattice paths}, we construct a weight-preserving bijection between the set of pole-free almost acyclic annular maps with a west crossing path and a certain set of lattice paths (Proposition~\ref{prop: discrete time bijection}), which proves the first assertion of  Theorem~\ref{thm: annular orientation scaling limit}. The main input to our proof is the bijection from~\cite{KMSW19}. 
In Section~\ref{subsec:height-func}, we prove the convergence of the lattice path under Brownian scaling and complete the proof of Theorem~\ref{thm: annular orientation scaling limit}.

\subsection{Pole-free almost acyclic annular maps}\label{sec: annular orientations}
A \emph{planar map} is a 
connected planar graph with finitely many edges together with a proper embedding into $\mathbb{S}^2 \cong \wh{\bC}:=\bC\cup\{\infty\}$. 
Two embeddings are equivalent if there exists an orientation-preserving homeomorphism from $\mathbb{S}^2$ to $\mathbb{S}^2$ that takes the first embedded graph to the second. All the planar maps we consider will be rooted, meaning that there is a distinguished oriented edge $\mathfrak e$. The tail vertex of $e$ is called the root vertex and the face to the right of $e$ is called the root face. An \emph{oriented planar map} is a planar map with an assignment of an orientation to each edge of the map. We will work with oriented planar maps. A \emph{source} (resp.\ \emph{sink}) is a vertex with only outgoing (resp.\ incoming) edges, and a vertex which is either a source or a sink is called a \emph{pole}.
A \textit{trail} in a planar map is a sequence of edges consisting of at least one edge, such that the ending vertex of one edge is the starting vertex of the next edge, and each edge only appears in the sequence once. We say that a vertex is \emph{in} a trail if it is the ending vertex or starting vertex of an edge in the trail. A \textit{path} is a trail such that every vertex in the trail is visited only once. 
A \textit{cycle} is a trail where the starting vertex agrees with the ending vertex. 
An oriented map is \emph{acyclic} if it does not contain any cycle.
A \emph{bipolar-oriented map} is an acyclic oriented planar map with a single source and single sink, such that the sink and the source  are incident to a common face. We  assume that the source is the root vertex and that the source and the sink are both incident to the root face.

An \textit{oriented annular map} is an oriented planar map with two distinguished faces $f_{\mathrm{inner}}$ and $ f_{\mathrm{outer}}$ with simple boundaries, where we assume that $f_{\mathrm{outer}}$ is the root face. We allow the case where $f_{\mathrm{inner}}$ and $ f_{\mathrm{outer}}$ share some vertices or edges. Every cycle in an annular map is embedded as a Jordan curve in $\mathbb{A}^2:=\mathbb{S}^2 \setminus (f_{\mathrm{inner}} \cup f_{\mathrm{outer}})$, where by an abuse of notation, $f_{\mathrm{inner}}$ and $f_{\mathrm{outer}}$ are identified with corresponding subsets on $\mathbb{S}^2$. 
A cycle in an annular map is called \textit{contractible} if and only if it is contractible as a loop on $\mathbb{S}^2 \setminus (f_{\mathrm{inner}} \cup f_{\mathrm{outer}})$, and \textit{non-contractible} otherwise. A \emph{crossing path} on an annular map is a path that starts from (a vertex on) one boundary component and ends at the other.

Let $C$ be a non-contractible cycle in an oriented annular map. Suppose the map is embedded in the plane such that $f_{\mathrm{outer}}$ contains $\infty$. We say that $C$ is oriented clockwise (resp.\ counterclockwise) if it is a clockwise (resp.\ counterclockwise) cycle around $f_{\mathrm{inner}}$ under this embedding.

We consider a particular class of oriented annular maps. Namely,
we say that an oriented annular map is \textit{pole-free} if it is source-free and sink-free and we say that the orientation is \emph{almost acyclic}  if it has no contractible cycle.
  We write $\cG^0$ for the set of pole-free annular maps with an almost acyclic orientation such that both inner and outer boundaries are oriented counterclockwise, and there exists a crossing path.

For each vertex $v$, a \emph{group} of incoming (resp.\ outgoing) edges refers to a maximal 
set of consecutive edges in cyclic order around $v$, such that all edges in the set are incoming (resp.\ outgoing) edges. See Figure \ref{fig: annular map example} for an illustration. It is shown in~\cite{KMSW19} that, for bipolar-oriented maps, for each vertex that is not equal to its source or sink, its incidental edges consist of exactly one group of incoming edges and exactly one group of outgoing edges.

\begin{figure}
    \centering
    \includegraphics[height=4cm]{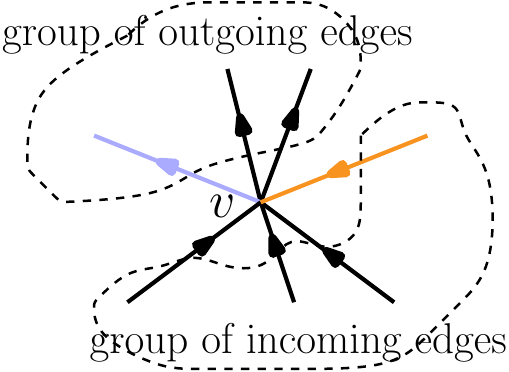}
    \hspace{2cm}
    \includegraphics[height=4cm]{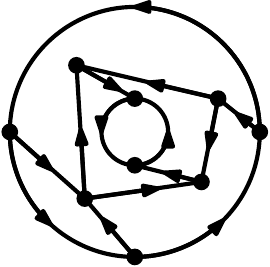}
    \caption{\textbf{Left:} A vertex in an annular map with a single group of incoming and outgoing edges. The blue (resp.\ orange) edge is the last edge in the incoming  outgoing (resp.\ outgoing  incoming) group of edges in the counterclockwise direction. \textbf{Right:} An example of a pole-free almost acyclic annular map.}
    \label{fig: annular map example}
\end{figure}

\begin{lemma}
    \label{lemma: good orientation two groups}
 Let $G\in\cG^0$. Then the same property above holds, i.e., for each vertex $v$, its incidental edges consist of exactly one group of incoming edges and exactly one group of outgoing edges.
\end{lemma}

Before proving the lemma, we define the operation of \emph{cutting} along a crossing path. 
    Let $G\in\cG^0$ and $p$ be a crossing path. For every vertex $v$ on $p$, every incidental edge that is not on $p$ is either on the \emph{left} of $p$ or \emph{right} of $p$. This is true because $p$ is directed; for the starting and ending vertices we also use the assumption that they are on the boundary of the distinguished faces.
    We then produce an oriented planar map $\wt{G}$ by adding two disjoint copies $\{p_\rmL, p_\rmR\}$ of $p$ to the map $G\setminus p$, so that for every vertex $v$ on $p$, edges on the left (resp.~right) of $v$ are connected to the copy of $v$ on $p_\rmL$ (resp.~$p_\rmR$) in $\wt{G}$. The planar graph $\wt{G}$ may be embedded in $\mathbb{S}^2$ by a modification of the embedding of $G$; see Figure \ref{fig: cutting along path}. Note that the definition above makes sense even if $p$ partially traces the outer boundary; see Figure \ref{fig: cutting along path with boundary}.

\begin{figure}[ht]
    \centering
    \includegraphics[height=3.5cm, page=1]{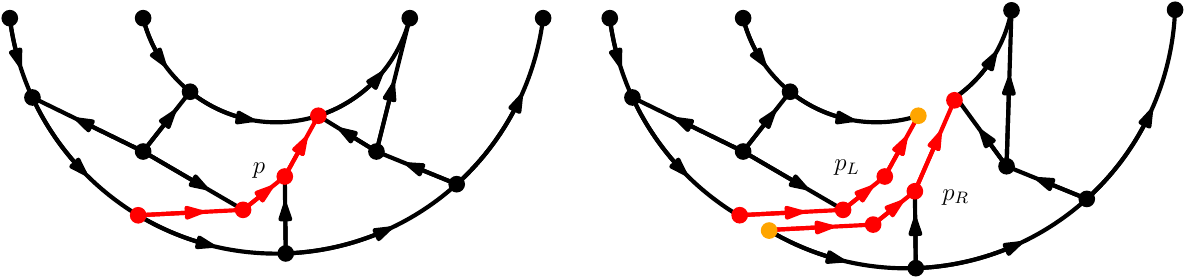}
    \caption{Cutting along a crossing path $p$ (shown in red) that connects the inner face and the outer face turns the annular map into a bipolar-oriented map, where the two poles are colored in orange. After cutting, each vertex on $p_\rmL$ has only one outgoing edge which is the one on $p_\rmL$. One should understand the maps obtained from cutting as bipolar-oriented maps with some nontrivial boundary conditions. See Figure \ref{fig: cutting dashed interface curve} for details on how this boundary condition would affect the interface curve. 
    }
    \label{fig: cutting along path}
\end{figure}


\begin{figure}[ht]
    \centering
    \includegraphics[height=3.5cm, page=2]{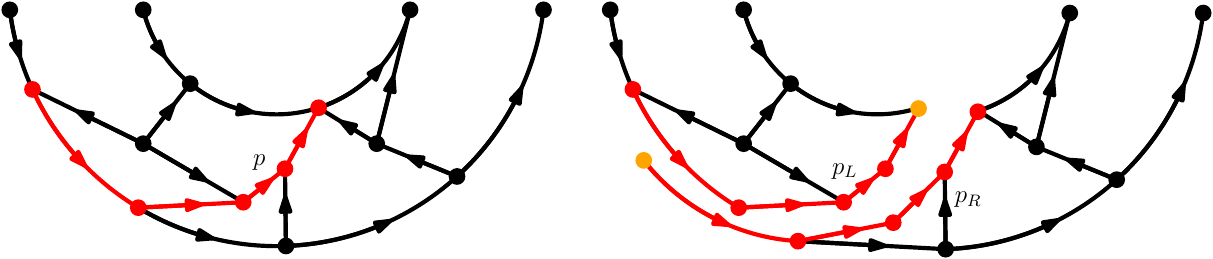}
    \caption{The path $p$ here traces part of the outer boundary. By our assumption, after cutting, the copy of the starting vertex on the right is still a source.
    }
    \label{fig: cutting along path with boundary}
\end{figure}

\begin{lemma}
\label{lemma: cut to be bipolar}
     Let $G\in\cG^0$, let $p$ be a crossing path on $G$, and let $\widetilde{G}$ be the map obtained by cutting $G$ along $p$. Then $\widetilde G$ is a bipolar-oriented map where the sink and the source are incident to the same face.
\end{lemma}

\begin{proof}
   A cycle in $\wt G$ would correspond to a contractible cycle in $G$, which implies that $\wt G$ cannot have any cycles.
    Next we verify that cutting yields exactly one source and one sink. Assume $p$ starts from vertex $v_1$ on the outer boundary and ends at vertex $v_2$ on the inner boundary. We write $v_1^\rmR$ for the copy of the vertex  $v_1$ on $p_\rmR$, and $v_2^\rmL$ for the copy of the vertex $v_2$ on $p_\rmL$.

    We first show that  $v_1^\rmR$ is a source in $\widetilde G$. 
    Let $e_L$, $e_R$ be the incoming and outgoing edges incident to $v_1$ on $\partial f_{\mathrm{outer}}$, respectively, and $e_1$ be the first edge on $p$. We claim that if $e$ is an edge incident to $v$ and $e_L$, $e_1$, $e$, $e_R$ are in cyclic order, then $e$ must be an outgoing edge. This implies that in $\wt{G}$, every edge incident to  $v_1^\rmR$ is outgoing, and hence $v_1^\rmR$ is a source in $\wt{G}$. 
    To prove the claim, we assume for the sake of contradiction that $e$ is an incoming edge. Since $G$ is source-free, the starting vertex of $e$ has at least one incoming edge, and we choose one edge $e'$ arbitrarily to extend $e$ to a trail consisting of $e$ and $e'$. Repeat this process until the trail contains an edge that starts on a vertex that is either on the boundary, on $p$, or on our trail. We denote the resulting train by $p'$. 
    No matter where $p'$ starts, there will be a contractible cycle by concatenating appropriately pieces of $p$, $p'$ and the boundary of $\widetilde G$, and this is a contradiction; see Figure \ref{fig: cut to create source}.
    A similar argument will show that $v_2^\rmL$ is a sink in $\wt{G}$. The source $v_2^\rmR$ and the sink $v_2^\rmL$ lie on a common face.

    \begin{figure}
        \centering
        \includegraphics[height = 6cm]{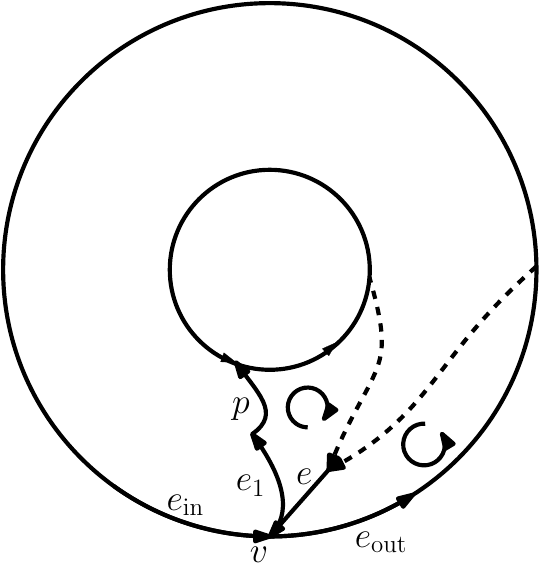}
        \hspace{1cm}
        \includegraphics[height = 6cm]{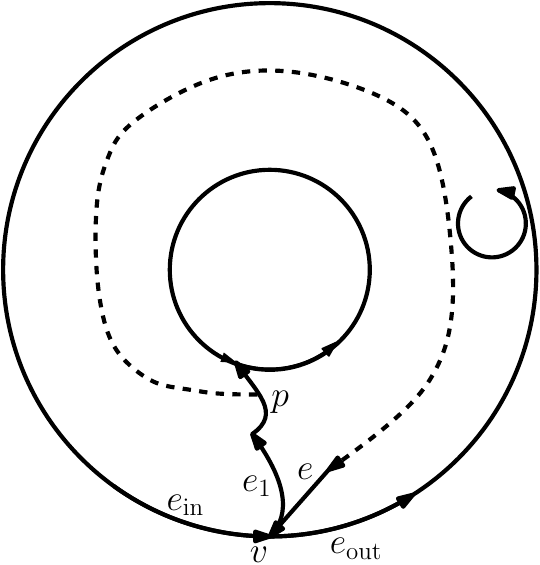}
        \caption{An illustration for the proof of Lemma \ref{lemma: cut to be bipolar}, more specifically the part where we prove that $v_1^\rmR$ is a source. If $e$ is an incoming edge of $v$, then there always exists a path $p'$, represented by dashed lines in the figures, such that $p'$ contains $v$ and part of $p'$ concatenated with part of the two boundaries and $p$ make a contractible cycle. This is a contradiction.} 
        \label{fig: cut to create source}
    \end{figure}
    
    To conclude the proof, we need to show that there is no other source or sink. By the definition of pole-free maps, vertices of $\wt G$ not on $p_\rmL$ or $p_\rmR$ can not be a source or a sink. Also, it is clear that  vertices on $p_\rmL$ or $p_\rmR$ other than $v_1^\rmR$ or $v_2^\rmL$ cannot be a source or a sink. 
\end{proof}


\begin{proof}[Proof of Lemma~\ref{lemma: good orientation two groups}]
    By Lemma~\ref{lemma: cut to be bipolar}, we can cut along the crossing path $p$ and obtain a bipolar-oriented map. The path $p$ corresponds to two paths in the new bipolar-oriented map which we call $p_\rmL$ and $p_\rmR$. If $v$ is not incident to $p$, then clearly it has the desired property. Now assume that $v$ is incident to $p$, and there are $4$ incidental edges $e_1, e_2,e_3,e_4$ in cyclic order that are oriented inward, outward, inward and outward, respectively. 
    It can happen that $e_{\mathrm{in}}$ and $e_{\mathrm{out}}$ coincides with some of the $4$ edges, but in all cases, incidental edges to $v$ in $\{e_1,e_2,e_3,e_4\}\cup \{e_{\mathrm{in}},e_{\mathrm{out}}\}$ will consist of at least $2$ groups of incoming edges and $2$ groups of outgoing edges in cyclic order. 
    Moreover, $e_{\mathrm{in}}$ and $e_{\mathrm{out}}$ must belong to distinct inward and outward groups. After cutting, at least one copy of $v$ on  $p_\rmL$ and $p_\rmR$ will have more than $2$ groups of alternating inward and outward edges, which gives a contradiction due to the first claim.
    See Figure~\ref{fig: bad vertex on path}.
\end{proof}
\begin{figure}[ht]
    \centering
    \includegraphics{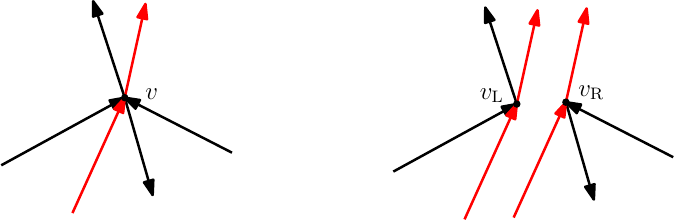}
    \caption{Suppose $v\in p$ does not satisfy the assertion of Lemma \ref{lemma: good orientation two groups}. 
    Then cutting along $p$, the vertex $v$ is ``split'' into two vertices $v_L$ and $v_R$, and we see that $v_R$, as a vertex in a bipolar-oriented planar map, does not satisfy this property, which makes a contradiction.}
    \label{fig: bad vertex on path}
\end{figure}
Suppose $G$ is in $\cG^0$ or that $G$ is a  bipolar-oriented map. 
As a consequence of Lemma \ref{lemma: good orientation two groups}, each vertex of $G$ has a single non-empty group of outgoing edges, and a single non-empty group of incoming edges. In counterclockwise order, we denote the last edge in the outgoing group to be the \textit{west edge} of the vertex, and the last edge in the incoming group to be the \textit{east edge} of the vertex. By our choice of orientation, for $G\in\cG^0$, each edge $e$ on $f_{\mathrm{inner}}$ (resp.\ $f_{\mathrm{outer}}$) is a west (resp.\ east) edge. Note that if we reverse the orientation of the edges, then the east edge of a vertex $v$ becomes the west edge of $v$ in the new map.
 West and east are the counterparts of \textit{north-west} and \textit{south-east} in~\cite{KMSW19}, but we choose our notation to be consistent with the imaginary geometry in the continuum, and with the notation in~\cite{GHS16}. 

 {
 A \emph{west trail} (resp.~\emph{east trail}) is a trail where each composing edge is a west (resp.~east) edge, and a  \emph{west path} (resp.~\emph{east path}) is a path where each composing edge is a west (resp.~east) edge.
}

\begin{proposition}
\label{prop: no loop in good annular map}
    Let $G\in\cG^0$ and let $p$ be a  west trail. Then $p$ cannot cross other west trails of $G$ or any east trail of $G$. The same holds if $G$ is a  bipolar-oriented map. 
    Furthermore, if there exists a  crossing west path in $G\in\cG^0$  (i.e., a path which is both crossing and west), 
    then there can be no west trail which forms a cycle before hitting the inner boundary.
\end{proposition}

\begin{proof}
   We work with the case where $G\in\cG^0$; the claim when $G$ is a bipolar orientation is similar. It is clear from definition that two west trails merge upon intersecting. Now assume $v$ is on a west trail $p = e_1\ldots e_n$ and on an east trail $p'=e_1'\ldots e_m'$, and assume $v$ is incident to $e_i$, $e_{i+1}$,  $e_j'$ and $e_{j+1}'$. 
   By the definition of the east edge and west edge, $e_{i+1},e'_{j+1},e'_{j},e_i$ must be in clockwise direction, 
   which implies that $p$ must bounce off $p'$ and cannot cross $p'$. For the last claim, let $p$ be a crossing west path and $\wt p$ be a west trail. If $\wt p$ forms a cycle before hitting the inner boundary, since $G$ has no contractible loops, $\wt p$ must hit and merge with $p$ before forming a cycle. Then if $\wt p$ further completes a cycle before hitting $f_{\mathrm{inner}}$, then $p$ can no longer hit  $f_{\mathrm{inner}}$, which contradicts with the assumption of a crossing path. This finishes the proof.
\end{proof}

\begin{figure}
    \centering
    \includegraphics[scale=0.6]{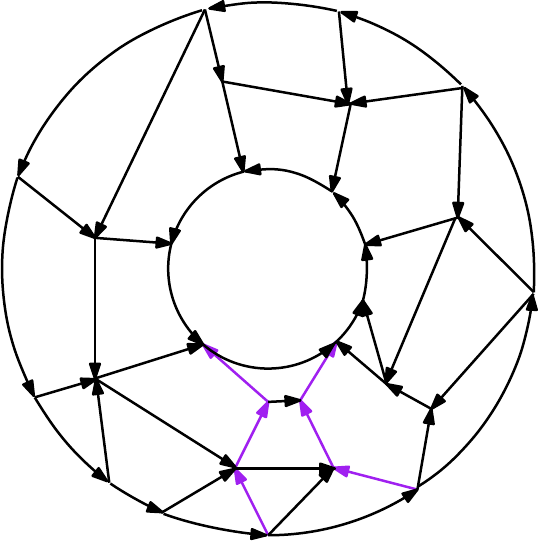}
    \hspace{1cm}
    \includegraphics[scale=0.6]{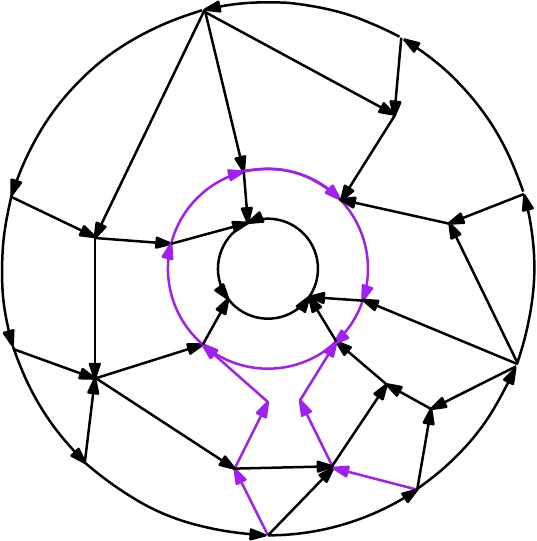}
    \caption{\textbf{Left:} Both purple paths are west paths that hits the inner boundary. \textbf{Right:} Both purple trails are west trails containing a non-contractible cycle.  By Proposition~\ref{prop: no loop in good annular map}, there does not exist any crossing west path.}
    \label{fig: group of incoming and outgoing edges}
\end{figure}

Note that there exists $G\in\cG^0$ such that there is no crossing west path; see the right part of Figure~\ref{fig: group of incoming and outgoing edges} for an example. Let $\cG\subseteq \cG^0$ be the set of maps in $\cG^0$ where there exists a crossing west path. 

\subsection{Bijection between pole-free almost acyclic annular maps and lattice paths}
\label{sec: bijection between maps and lattice paths}
Recall that $\cG$ is the set of pole-free almost acyclic annular maps with cyclically oriented boundary such that there exists a crossing west path. For $k\in \bN$, let $\cG_k$ be the subset of $\cG$ where every map has $k$ edges on the outer boundary. The goal of this subsection is to prove Proposition~\ref{prop: discrete time bijection}, where we construct a weight-preserving bijection between $\cG_k$ and $\mathcal H_k$, a set of lattice paths on $\bZ_{\geq 0}^2$.

We first define the type-$(i,j)$ faces for bipolar-oriented maps $\wt G$.
Let $f$ be a bounded face of $\wt G$, and let $\partial f $ be the set of edges incident to $f$. A vertex $v$ on the boundary of $f$ is called a \emph{maximal} (resp.~\emph{minimal}) vertex if both edges in $\partial f$ incident to $v$ are incoming (resp.~outgoing). Following~\cite[Section 2]{KMSW19}, $f$ has a unique maximal vertex and a unique minimal vertex. We say $f$ is of type-$(i,j)$, if, on the boundary, there are $i+1$ edges from minimum vertex to the maximum vertex in the clockwise direction, and $j+1$ edges from the minimum vertex to the maximum vertex in the counterclockwise direction. 
\begin{lemma}\label{lem:type-ij}
    Let $G\in\cG$. Then the notion of type-$(i,j)$ faces is well-defined for every face in $G\in\cG$ that is different from $f_{\mathrm{inner}}$ and $f_{\mathrm{outer}}$.
\end{lemma}
\begin{proof}
    This follows from the result for bipolar-oriented maps by cutting along a crossing path as in Lemma~\ref{lemma: cut to be bipolar}.
\end{proof}

    Now we review the theory of height functions in~\cite{KMSW19}.  Let $\wt{G}$ be a bipolar-oriented map. Its edge set $E(\wt{G})$ consists of the west edges $E_0(\wt{G})$ and the non-west edges $E_1(\wt{G})$. 
    For every $e\in E_1(\wt{G})$, we introduce an additional vertex $v'$ associated to $e$ embedded at the midpoint of $e$, as well as an additional edge starting from $v'$ and ending at the ending vertex of $e$, embedded in the natural way.
    Then there is a set of vertices $V'(\wt{G})$ and edges $E'(\wt{G})$ obtained in this way, each in bijective correspondence with $E_1(\wt{G})$.
    The \emph{west tree} associated to $\wt{G}$ is defined to be the planar tree $\mathcal{W}(\wt{G}) = (V(\wt{G})\cup V'(\wt{G}), E_0(\wt{G}) \cup E'(\wt{G}))$, where we impose the sink of $\wt{G}$ to be the root of $\mathcal{W}(\wt{G})$.

    If we reverse the orientations of each edge and perform the same procedure, we obtain an analogous tree rooted at the source, which we call the \emph{east tree}. We can further embed both trees simultaneously on $\mathbb S^2$ such that they do not cross each other; 
    c.f.~\cite[Figure 2]{KMSW19} and Figure~\ref{fig: cutting dashed interface curve}. 
    Each edge $e\in E(\wt G)$ corresponds to two edges, $e_{\mathrm W}$ in the west tree and $e_{\mathrm E}$ in the east tree. We can further define a path $\eta'{:[0,|E(G)|-1]\cap\bZ\to E(\wt G)}$ tracing the interface between the two trees with $|E(\wt{G})|-1$ steps, such that each step of the interface  either moves along an edge of the map, or across a type $(i,j)$ face of the map from the maximal vertex to the minimal vertex and then traverses an edge. The former is called an $m_e$-step, while the latter is called an $m_{i,j}$-step. See Figure \ref{fig: cutting dashed interface curve}. For each edge $e\in E(\wt G)$, we define its west height to be the graph distance from the ending vertex of $e$ to the sink in the west tree, and we define its east height to be the graph distance of the starting vertex of $e$ in the east tree to the source in the east tree. 
    We get a height function $(\wt{\sfL}_t, \wt{\sfR}_t)_{0\le t \le |E(G)|-1}$, so that $\wt{\sfL}_t$ (resp.~$\wt{\sfR}_t$) is the west (resp.~east) height of the edge corresponding to $\eta'(t)$.

By \cite[Section 2.1]{KMSW19}, if $\eta'(t) $  traverses an edge, then 
\begin{equation}
    \label{eq: edge step}
    (\widetilde{\mathsf L}_t,\widetilde{\mathsf R}_t) - (\widetilde{\mathsf L}_{t-1},\widetilde{\mathsf R}_{t-1}) = (-1,1),
\end{equation}
and if $\eta'(t) $ crosses a type-$(i,j)$ face and then traverses an edge, then
\begin{equation}
    \label{eq: face step}
    (\widetilde{\mathsf L}_t,\widetilde{\mathsf R}_t)-(\widetilde{\mathsf L}_{t-1},\widetilde{\mathsf R}_{t-1}) = (j,-i).
\end{equation}

This definition of the height functions of a bipolar-oriented map differs from the one in \cite{KMSW19} in that we have swapped the two coordinates. This is for consistency with the continuum, where the first (resp.\ second) coordinate measures the change in the left (resp.\ right) LQG boundary length. Its name is justified by its close relation to the height function of a tree, as in e.g.\ \cite[Section 1.1]{le2005random}. The function $\wt{\sfR}$ is exactly the height function of the east tree, while $\wt{\sfL}$ is the height function of the west tree up to a time-reversal and translation, since we are exploring the west tree from one of its leaves and in a reversed order.

 For a bipolar-oriented map or an annular map $G\in\mathcal{G}$, we set its weight to be~\eqref{eq: planar map weight}, where $F\degree(G)$ is $F(G)$ minus the root face for bipolar-oriented maps and  $F(G)\setminus\{f_{\mathrm{inner}},f_{\mathrm{outer}}\}$ for $G\in\cG$. Also recall the definition of the set of lattice walks $\mathcal{H}_{k}$ via the Conditions~(\ref{cond: Hk lmn cond 1})-(\ref{cond: Hk lmn cond 4}) as in Section~\ref{sec: intro precise statements and outline} and the weight $W(\sfZ)$ given by~\eqref{eq: lattice path weight}.

The main result of this section is the following proposition. 
\begin{proposition}
    \label{prop: discrete time bijection}
    For each $k\in \mathbb{N}$, there exists a weight-preserving bijection between $\mathcal{G}_{k}$ and $\mathcal{H}_{k}$, given by $ \xi\circ \psi\circ \phi$, where $\phi,\psi, \xi$ are introduced in Lemmas \ref{lemma: bijection step 1}, \ref{lemma: bijection step 2} and \ref{lemma: bijection step 3}, respectively. 
\end{proposition}
 
Our strategy is to use Lemma~\ref{lemma: cut to be bipolar} to obtain a bijection $\phi$ between annular maps in $\mathcal{G}_k$ and a family of bipolar-oriented maps $\wt{\mathcal{G}}_k$ by cutting along the crossing west path started from the root. The map $\psi$ is a bijection between $\wt{\mathcal{G}}_k$ and a class of lattice paths $\wt{\mathcal{H}}_k$  {which is a restriction of a bijection in} \cite[Theorem 1.1, 1.2]{KMSW19}. The final step is an elementary bijection $\xi$ between $\wt{\mathcal{H}}_k$ and $\mathcal{H}_k$ where we shift the lattice path and remove the last $m$ $(-1,1)$-steps. By removing these steps, the path $\eta'$ on the bipolar-oriented map becomes a loop after gluing.
See Figure~\ref{fig: discrete logics} for an overview.

\begin{figure}
    \centering
    \includegraphics{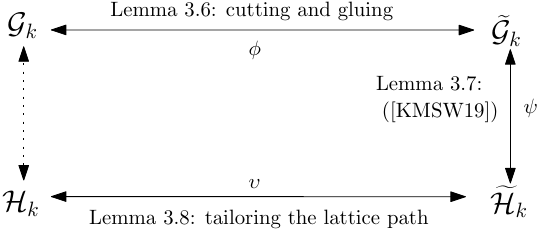}
    \caption{An overview of the proof of Proposition~\ref{prop: discrete time bijection}. } 
    \label{fig: discrete logics}
\end{figure}

For $G\in \cG$, we let $p$ be the west trail started from the root vertex, stopped upon hitting the inner boundary. By Proposition~\ref{prop: no loop in good annular map}, $p$ is a crossing west path. By Lemma~\ref{lemma: cut to be bipolar}, if we cut along $p$, we will get a bipolar-oriented map $\wt G$ with $\abs{E(G)}+m$ edges, where $m$ is the length of $p$. 

We now define the sets $\wt{\mathcal{G}}_k$ and $\wt{\mathcal{H}}_k$. For a bipolar-oriented map, we define its left (resp.\ right) boundary to be the west path (resp.\ east path) from the sink to the source. Let $\widetilde{\mathcal{G}}_{k}$ be the family of pairs $(\widetilde{G}, m)$, where $m \in \bN$ and $\widetilde{G}$ is a bipolar-oriented map 
that satisfy the following conditions. 
\begin{enumerate}
    \item There are $k+m$ edges on the right boundary.
    \label{cond: tilde Gk 1}
    \item The number of edges on the left boundary is greater or equal to $m$.
    \label{cond: tilde Gk 2}
    \item The uppermost $m$ edges on the right boundary are all west edges. 
    \label{cond: tilde Gk 3}
    \item { The uppermost edge on the left boundary does not coincide with the uppermost edge on the right boundary.} Equivalently, the sink is incident to at least 2 edges.
    \label{cond: tilde Gk 4}
\end{enumerate}
 
We emphasize that $\widetilde{\mathcal G}_k $ is not the usual class of bipolar-oriented maps with prescribed boundary lengths  and forms a subset of the set of bipolar-oriented maps studied in \cite{KMSW19}. Condition~\ref{cond: tilde Gk 3} is an additional boundary condition on the uppermost $m$ edges of the right boundary: for each vertex 
incident to one of these edges, the only outgoing edge is the boundary edge itself. As a consequence, once the interface curve reaches this part of the right boundary, it remains on the right boundary until it terminates at the sink. This boundary condition is compatible with the cutting and gluing for annular planar maps as in Lemma~\ref{lemma: cut to be bipolar} and Figure~\ref{fig: cutting dashed}.


\begin{figure}[t]
    \centering
    \includegraphics[scale=0.7]{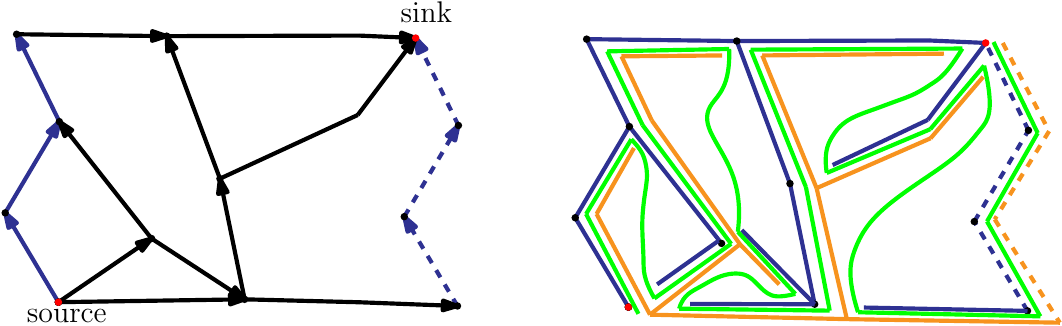}
    \caption{{\textbf{Left}: Illustration of $(G,m)\in\wt{\cG}_k$ for $m=3$ and $k=2$. The dashed edges represent the special boundary segment from Condition~\ref{cond: tilde Gk 3}. \textbf{Right}:} The green curve is the interface curve between the west tree and the east tree. The steps of the green curve are given by $m_e m_e m_{1,1} m_e m_e m_e m_{2,2} m_{1,0} m_e m_e m_e m_{1,1} m_{1,3}$.  For every vertex incident to the dashed segment, the only outgoing edge is the boundary edge itself. Therefore, once the interface curve reaches this part of the right boundary, it is forced to continue along the boundary and cannot leave it before reaching the sink. Furthermore, if we drop the last 3 steps as in Lemma~\ref{lemma: bijection step 3} and glue the dashed edges on the right to the solid edges on the left as in Figure~\ref{fig: cutting dashed}, the green curve becomes a loop on the annular map. }
    \label{fig: cutting dashed interface curve}
\end{figure}

Let $\widetilde{\mathcal{H}}_{k}(\ell, m,n)$ be the set of lattice paths $\wt{\sfZ} = (\widetilde{\mathsf L}_j, \widetilde{\mathsf R}_j)_{0\leq j\leq m+n}$ that satisfy the following conditions.
\begin{enumerate}[(i)]
    \item Each step of the lattice path is selected from $(-1, 1)$ and $(j, -i)$ for some $i$, $j\in \bZ_{\ge 0}$.
    \label{cond: Hk tilde cond 2}
    \item The last $m$ steps of the lattice path are all $(-1, 1)$.
    \label{cond: Hk tilde cond 3}
    \item We have $\wt{\sfZ}(0) = (\ell+m, 0)$ and $\wt{\sfZ}(n+m) = (0, k+m)$.
    \label{cond: Hk tilde cond 4}
    \item For any $j\in [0, n+m] \cap \bZ$, $\wt{\sfL}(j)\ge 0$ and $\wt{\sfR}(j) \ge 0$.
    \label{cond: Hk tilde cond 6}
    \item $\inf_{0 \le j < n+m} \wt{\sfL}(t) = 0$.
    \label{cond: Hk tilde cond 7}
\end{enumerate}


We define $\widetilde{\mathcal{H}}_{k}$ to be the collection of tuples $\wt{H} = (\wt{\sfZ},   m)$, where $\wt{\sfZ}\in \widetilde{\mathcal{H}}_{k}(\ell, m, n)$ and $\ell, m,n\in \bZ_+$. 

\begin{lemma}
    \label{lemma: bijection step 1}
    For any $k\in\mathbb N$ there is a bijection $\phi:\mathcal G_k\to \widetilde{\mathcal G}_k$ given by  {$\phi(G) = (\wt{G}, m)$, where $\wt{G}$ is the bipolar-oriented map we obtain when we cut $G$ along the west path starting from the root vertex, and $m$ is the length of the west path.} Moreover, $W(G) = b_e^{-m}W(\wt G)$.
\end{lemma}

\begin{lemma}
    \label{lemma: bijection step 2}
    For every $k\in \mathbb{N}$ there is a weight-preserving bijection $\psi: \widetilde{\mathcal{G}}_{k}\to \widetilde{\mathcal{H}}_{k}$ such that 
    $\psi((\wt{G}, m)) = (\sfZ,  m)$. Furthermore, $\sfZ\in \widetilde{\mathcal{H}}_{k}(\ell, m, n)$, where $n=|E(\wt{G})|-m$ and $\ell$ is such that the right boundary of $\wt{G}$ has $\ell+m$ edges. 
\end{lemma}

\begin{lemma}
    \label{lemma: bijection step 3}
    For every $k\in \mathbb{N}$ there is a bijection $ \xi: \widetilde{\mathcal{H}}_{k}\to\mathcal{H}_{k}$, where for $H=(\wt\sfZ,m)$, $\xi(H)$ is given by shifting and removing the last $m$ steps of $\wt\sfZ$. 
    Moreover,  $W(\xi(H)) = b_e^{-m}W(\wt\sfZ)$.
\end{lemma}

 
We start with the \emph{gluing} operation for bipolar-oriented maps, which is the inverse of the cutting operation defined  {above Lemma \ref{lemma: cut to be bipolar}}. 
For $(\widetilde{G}, m)\in \widetilde{\mathcal{G}}_k$, let $(e_1^\rmL, \dots, e_m^\rmL)$ be the first $m$ edges of the west path starting from the source, and $(e_1^\rmR, \dots, e_m^\rmR)$ be the last $m$ edges of the east path started from the source. Identifying $e_j^\rmL$ and $e_j^\rmR$ for each $j$ gives an annular map $G$, where $f_{\mathrm{outer}}$ (resp.~$f_{\mathrm{inner}}$) is chosen such that its boundary contains the non-identified edges of the right (resp.~left) boundary of $\widetilde{G}$. See Figure~\ref{fig: cutting dashed}.
\begin{figure}[ht]
    \centering
    \includegraphics[height=4.5cm]{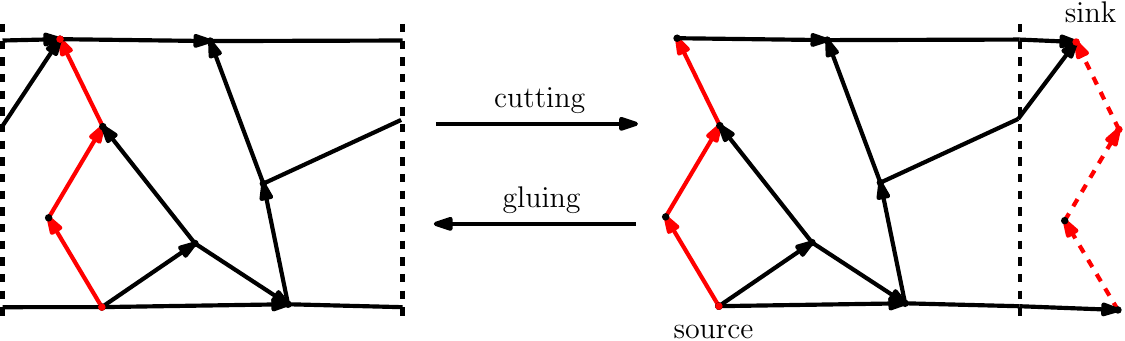}
    \caption{Cutting and gluing along the red path with $m=3$. We will see later that the path corresponds to a welding interface in the continuum.
    }
    \label{fig: cutting dashed}
\end{figure}

\begin{proof}[Proof of Lemma \ref{lemma: bijection step 1}]
    For $G\in \mathcal{G}_k$ with root vertex $v_0$, let $m$ be the length of the crossing west path from $v_0$. Cut along this path to obtain $\widetilde{G}$ and set $\phi(G)=(\widetilde{G}, m)$. The pair $(\widetilde{G}, m)$ clearly satisfies Conditions \ref{cond: tilde Gk 1}-\ref{cond: tilde Gk 2} for $\widetilde{\mathcal{G}}_{k}$. Condition~\ref{cond: tilde Gk 3} follows since the path we are cutting along is a west path. To verify Condition~\ref{cond: tilde Gk 4}, recall that we are cutting along a west path that stops upon hitting the inner boundary. Therefore after cutting, the sink must be incident to at least 2 edges: one edge from the inner cycle and another edge from the west path. Therefore the uppermost edge on the left boundary cannot coincide with the uppermost edge on the right boundary. See Figure~\ref{fig: cutting dashed}.

    Conversely, we will now check that gluing the bottommost $m$ edges $e^\rmL_1,\ldots,e^\rmL_m$ (bottom to top) on the left boundary to the uppermost $m$ edges $e^\rmR_1,\ldots,e^\rmR_m$ (bottom to top) on the right boundary of $\widetilde{G}_{k}$ yields an element in $\mathcal G_k$. The gluing operation identifies the sink (resp.~source) with a vertex that is not a sink (resp.~source). Therefore after gluing there is neither sink nor source. The gluing interface is a crossing west path  by Condition~\ref{cond: tilde Gk 3}. Moreover, we claim that there is no contractible cycle after gluing. 
    Otherwise, the cycle would either intersect the gluing interface or not intersect the gluing interface. 
    In the former case, the cycle must cross the gluing interface at least once from left and once from right. Thus one vertex of $e^R_1,\ldots,e^R_m$ has an outgoing edge that is not one of $e^R_k$, which contradicts with Condition~\ref{cond: tilde Gk 3}.
    In the latter case, we have a contractible cycle in $\widetilde{G}$, also making a contradiction. 
    The outer boundary of the glued map corresponds to the part of right boundary in $G$ that is not glued, and this segment has $k$ edges.  Thus we get an element in $\mathcal G_k$ after gluing. 
    
    It is clear that the two maps we constructed via cutting and gluing are inverses of each other and that their weights are related as desired, which concludes the proof. 
\end{proof}

\begin{proof}[Proof of Lemma~\ref{lemma: bijection step 2}]
  {
    Let $\widetilde{\mathcal{G}}_{k}'(\ell, m, n)$ be the set of bipolar-oriented maps $\wt{G}$ that satisfy Conditions \ref{cond: tilde Gk 1}-\ref{cond: tilde Gk 2} in the definition of $\widetilde{\mathcal{G}}_{k}$, such that there are $\ell+m$ edges on the left boundary and   $|E(\wt{G})| = n+m$.   Let $\widetilde{\mathcal{H}}_k'(\ell, m, n)$ be the family of lattice paths that satisfy Conditions~(\ref{cond: Hk tilde cond 2})-(\ref{cond: Hk tilde cond 6}) except for Condition~(\ref{cond: Hk tilde cond 3}) in the definition of $\widetilde{\mathcal{H}}_k(\ell, m, n)$. 
    It follows from \cite[Theorem 2]{KMSW19} that such $\widetilde{G}\in\widetilde{\mathcal{G}}_{k}'(\ell, m, n)$ is in bijective correspondence to a lattice path $(\widetilde{\sfL}, \widetilde{\sfR})\in \widetilde{\mathcal{H}}_k'(\ell, m, n)$ given by the bipolar height functions. 

  Next we restrict this bijection to $ \widetilde{\mathcal{G}}_k(\ell, m, n)$ and $ \widetilde{\mathcal{H}}_k(\ell, m, n)$. By Condition~\ref{cond: tilde Gk 3} in the definition of $\widetilde{\mathcal{G}}_{k}$, for each $\wt G\in \widetilde{\mathcal{G}}_{k}$ and any one of the uppermost $m$ edges of the right boundary started from $v$, there is no other outgoing edge started from $v$.  
  Once the interface curve reaches this segment, it is forced to follow the right boundary all the way to the sink, and the final $m$ steps of the bipolar height function are all edge steps, i.e.\ equal to $(-1,1)$, implying Condition~(\ref{cond: Hk tilde cond 3}) of $\wt{\mathcal{H}}_k(\ell, m, n)$. Furthermore, in the construction of the space-filling path, the last step in the walk that is not equal to $(-1,1)$ corresponding to the last edge on the right boundary that is not a west edge, and therefore Condition~(\ref{cond: Hk tilde cond 3}) for $\wt{\mathcal{H}}_k(\ell, m, n)$  implies Condition~\ref{cond: tilde Gk 3} for $\widetilde{\mathcal{G}}_{k}$ in this bijection.  
    
    To obtain a bijection between $\widetilde{\mathcal{G}}_k$ and $\widetilde{\mathcal{H}}_k$ (and hence a bijection between $\wt{\mathcal{G}}$ and $\wt{\mathcal{H}}$), it remains to show that if $\wt{G}\in \widetilde{\mathcal{G}}_{k}'(\ell, m, n)$ satisfies Condition~\ref{cond: tilde Gk 4} in the definition of $\widetilde{\mathcal{G}}_{k}$, then its corresponding lattice path in $\widetilde{\mathcal{H}}_{k}'$ satisfies Condition~(\ref{cond: Hk tilde cond 7}) of the definition of $\widetilde{\mathcal{H}}_{k}(\ell, m, n)$, and vice versa.
    Since the uppermost edge on the left boundary does not coincide with the uppermost edge on the right boundary, the interface curve must visit the uppermost edge on the left boundary before it traces the uppermost $m$ edges on the right boundary. Thus $\inf_{0\leq j < n+m} \sfL(j) = 0$. Similarly, if $\inf_{0\leq j < n+m} \sfL(j) = 0$, then the interface curve $\eta'$ must visit an edge incident to the sink before tracing the uppermost $m$ edges on the right boundary. This concludes the proof.
    }
\end{proof}

\begin{proof}[Proof of Lemma~\ref{lemma: bijection step 3}]
    Consider the bijective map from $\widetilde{\mathcal{H}}_{k}(\ell,m,n)$ to $\mathcal{H}_{k}(\ell,m,n)$  given by first translating the path so that it starts from the origin, and then removing the final $m$ steps. The weights under this bijection are related as desired, and we can extend it to a bijection between $\widetilde{\mathcal{H}}_{k}$ and  $\mathcal{H}_{k}$.
\end{proof}

\begin{proof}[Proof of Proposition~\ref{prop: discrete time bijection}]
   The proposition follows by combining Lemmas~\ref{lemma: bijection step 1}, \ref{lemma: bijection step 2} and \ref{lemma: bijection step 3}. 
\end{proof}

\subsection{Convergence of height functions}\label{subsec:height-func}

In this subsection, we prove that under the Brownian scaling, the height function of an annular map sampled according to weights converges to a correlated Brownian {path measure}, which implies Theorem \ref{thm: annular orientation scaling limit}.

    Recall the weights $W(G)$ in~\eqref{eq: planar map weight}.  We define a $\sigma$-finite measure $M_\cG$ on $\mathcal{G}$ (and, by Proposition~\ref{prop: discrete time bijection}, equivalently on $\mathcal{H}$), where for each $G\in \cG$, $M_\cG(\{G\})$ is equal to the weight $ W(G)$. We write $M_{\cG,k}$ for the restriction of $M_\cG$ to $\cG_k$. 
 For lattice paths $\sfZ=(\sfL_t, \sfR_t)_{0\le t\le n}$ in $\mathcal{H}$, by Condition~(\ref{cond: Hk lmn cond 4}) in the definition of $\mathcal{H}_k$, $\sfZ_n=(-\ell, k)$ for some $\ell$, $k\in\bN$. We write $\wh{\sfZ}$ for the lattice path such that $\wh{\sfZ}_t=\sfZ_{n-t}+(\ell, 0)$, i.e.~it is the time-reversal of $\sfZ$, translated so that it starts from $(0,k)$. The steps of  $\wh{\sfZ}$ are now in $\{(1,-1)\}\cup\{(-j,i)\,:\, b_{i,j}>0\}$. {The reason we work with the time-reversal of $\sfZ$ instead of $\sfZ$ is that the time-reversal haw the law of a particular stopped random walk, see Lemma \ref{lemma: random walk representation of Hk} below.}

Fix an admissible set of weights $\{b_{e}, b_{i, j}: i, j\in \bZ_{\ge 0}\}$ satisfying Conditions~(\ref{cond: b-A})-(\ref{cond: b-D}). Let $k\in\bN$. Consider a random walk $\sfS=(\sfU_t, \sfV_t)_{t\in \bZ_{\ge 0}}$ with i.i.d.\ steps that starts from $(0, k)$ such that $\sfS_{t+1} - \sfS_t = (1,-1)$ with probability $b_e$ and equals $(-j,i)$ with probability $b_{i,j}$. We write $\bbP_k$ for its law.  Let $\tau = \inf\{t: \sfV_t \le 0\}$, which is finite almost surely. We write $\sfS^{\tau}=(\sfU_t, \sfV_t)_{0\le t\le \tau}$ for the walk truncated at time $\tau$.

\begin{lemma}
    \label{lemma: random walk representation of Hk}
    In the setting above, conditioned on the positive probability event $\{\sfS_{\tau} \in \bN\times\{0\}\} =\{ V_\tau=0 \} $, the  law of $\sfS^\tau$ is proportional to the law of $\wh{\sfZ}$ under $M_{\cG,k}$. Furthermore, the measure $M_{\cG,k}$ is finite and its total mass is equal to $\bbP_k[\sfS_{\tau}\in \bN\times \{0\}]$. 
\end{lemma}
\begin{proof}
  It is immediate that the event $A:=\{\sfS_{\tau}\in \bN\times\{0\}\}$ has positive probability.  Let $\mathring\sfZ$ be a lattice path started from $(0,k)$ and staying above  $\bZ\times \{0\}$, whose steps consists of $(1, -1)$ and $(-j, i)$ with $i$, $j$  such that $b_{i, j}\neq 0$, and  stopped upon first hitting $\bZ\times \{0\}$. By construction, we have $\bP(\sfS^\tau=\mathring\sfZ) =b_e^{n_{e}}\prod_{i,j\geq 0}b_{i,j}^{n_{i, j}}$, where $n_e$ is the number of $(1 ,-1)$ steps and $n_{i, j}$ is the number of $(-j, i)$ steps in $\mathring{\sfZ}$. The claim then follows by varying the choices of $\mathring\sfZ$.
\end{proof}

The next proposition gives the scaling limit result in 2.\ of  Theorem \ref{thm: annular orientation scaling limit}.
\begin{proposition}
\label{prop: path scaling limit}
    Let $(\mathsf L^{(N)}_t,\mathsf R^{(N)}_t)_{0\leq t\leq \abs{E(G)}-1}$ be the height functions associated to an  annular map $G$ sampled from {$M_{\mathcal{G},{\lfloor b\sqrt{N}\rfloor }}$}. Then as $N\to \infty$, the law of $\frac{1}{\sqrt{N}}(\mathsf L_{\lfloor Nt\rfloor}^{(N)},\mathsf R_{\lfloor Nt\rfloor}^{(N)})_{0\leq \lfloor Nt\rfloor\leq \abs{E(G)}-1}$ converges weakly to $\int_0^\infty \d a B^{\gamma, \bH}_{0, -a+bi}$ with $\gamma$ given by~\eqref{eq: relation between gamma and weights}.
\end{proposition}

We will use the following lemma, which follows from Donsker's theorem, the strong Markov property of Brownian motion, and the fact that Brownian motion started from 0 takes negative values in any open interval containing zero.
\begin{lemma}\label{lemma: donsker}
    Let $b>0$ and let $\mathsf{S}^{(N)}:= (\sfU^{(N)} \sfV^{(N)})$ be a random walk on $\bbR^2$ such that 
    \begin{enumerate}
        \item $\sfS^{(N)}_0 = (0, \lfloor b\sqrt{N} \rfloor )$;
        \item $\sfS_{n+1}^{(N)}-\sfS_n^{(N)}$ is i.i.d.~with mean $0$, and its norm is bounded by some constant independent of $N$; 
        \item $\mathbb{E}[(\sfU^{(N)}_{n+1}-\sfU^{(N)}_{n})^2] = \mathbb{E}[(\sfV^{(N)}_{n+1}-\sfV^{(N)}_{n})^2] = a_{11} $ and $\mathbb{E}[(\sfU_{n+1}^{(N)} -\sfU_{n}^{(N)})(\sfV_{n+1}^{(N)} - \sfV_{n}^{(N)})] = a_{12}$ for some constants $a_{11}$ and $a_{12}$.
    \end{enumerate}
    Let $\tau_N = \inf \{ t:\sfV_{t}^{(N)} \leq 0\}$ and let $B_t = (U_t, V_t)$ be a two-dimensional Brownian motion started from $(0,b)$ with covariance matrix $\begin{pmatrix}
        a_{11} & a_{12}\\
        a_{12} & a_{11}
        \end{pmatrix}$.
        Then as $N\to \infty$,  $(\frac{\sfS_{\lfloor Nt\rfloor}}{\sqrt{N}})_{0\leq t\leq \tau_N} $ converges in distribution to $ (B_t)_{0\leq t\leq \tau} $ where $\tau: = \inf\{t: \sfV_t  \leq 0 \}$ with respect to the metric~\eqref{eq:curve-metric}.   
\end{lemma}
    

\begin{proof}[Proof of Proposition~\ref{prop: path scaling limit}]
    Let $\wh\sfZ^{(N)} = (\wh \sfL^{(N)}_t, \wh\sfR^{(N)}_t)$ be the reversal of $\sfZ^{(N)}$, shifted such that it starts from $(0,\lfloor b\sqrt N \rfloor)$. By the definition of the measure $B^{\gamma, \bH}_{0, -a+bi}$, it suffices to show that the law of $N^{-1/2}\wh{\sfZ}^{(N)}$ converges as $N\to\infty$ to $\int_0^\infty \d a B^{\gamma, \bH_{}}_{bi, a}$. Since the steps that $\wh\sfZ^{(N)}$ can take are restricted to $\{(1,-1)\}\cup\{(-i,j):i,j\geq 0\}$, we see that $\inf\{t: \wh\sfR^{(N)}_t\leq 0\} = \inf\{t: \wh\sfR^{(N)}_t= 0\}$. The claim then follows by Lemma~\ref{lemma: random walk representation of Hk} and Lemma~\ref{lemma: donsker}. 
\end{proof}
\begin{proof}[Proof of Theorem \ref{thm: annular orientation scaling limit}]
    Proposition \ref{prop: discrete time bijection} implies that $\mathcal G_k$ and $\mathcal H_k$ are related by a weight-preserving bijection, and Proposition \ref{prop: path scaling limit} gives the scaling limit result described in Theorem \ref{thm: annular orientation scaling limit}. The fact that $\gamma$ can take every value in $[0,\sqrt{2})$ follows from the fact that 
    \[
    \frac{2b_e + 2\sum_{i,j \in \bZ_{\ge 0}} ijb_{i,j}}{2b_e + \sum_{i,j\in \bZ_{\ge 0}}(i^2+j^2) b_{i,j}}
    \]
    may take every value in $[0, 1)$ for $b_e$, $b_{i,j}$ satisfying Conditions~(\ref{cond: b-A})-(\ref{cond: b-D}). 
\end{proof}

\section{Construction of the limiting curve-decorated surface}
\label{sec: resampling property of scaling limit}

The goal of this and the next two sections is to prove Theorem~\ref{thm: quantum annulus mot}. 
We conformally weld a quantum cell and a quantum disk as in Figure~\ref{fig: two ways of welding} to get an LQG surface with two welding interfaces and annular topology, whose law is denoted as $\MA^{1,1\dagger}$, and we will later (in Section~\ref{sec: resampling characterization scaling limit}) prove that this welded surface can be descirbed in terms of  $\LF_\tau$ and mating-of-trees as in Theorem~\ref{thm: quantum annulus mot}.  In Section~\ref{sec: qt mot}, we present the mating-of-trees description of the various LQG surfaces we are considering in terms of  Brownian paths. In Section~\ref{sec: proof sec 4 main thm}, we work with the aforementioned conformal welding of a quantum cell and a quantum disk, and prove a resampling property of the interfaces (Proposition~\ref{prop:curve-resample}). 

\subsection{Mating-of-trees descriptions of LQG surfaces}
\label{sec: qt mot}

In~\cite{MoT}, it is shown that one can encode a whole-plane space-filling $\SLE_{\kappa'}$ decorated LQG surface called \emph{$\gamma$-quantum cone} by a 2D Brownian motion. The precise definition of the $\gamma$-quantum cone, which was introduced in~\cite[Definition 4.10]{MoT}, will not be used in our paper. The mating-of-trees framework was later extended to Brownian path measures and other LQG surfaces~\cite{MoT,MS19Finite,MoTboundary,ConfWeldDisks,ASYZ24}. In this subsection we collect several mating-of-trees results that we will use, starting with  the celebrated mating-of-trees theorem for $\gamma$-quantum cones.

\begin{theorem}[\cite{MoT, GHMS17}]
    \label{thm: quantum cone mot}
    Let $\cC=(\bC, \phi, 0, \infty)$ be a $\gamma$-quantum cone decorated with an independent whole-plane space-filling $\SLE_{\kappa'}$ $\eta'$ from $\infty$ to $\infty$, parametrized according to the quantum area and $\eta'(0)=0$. The boundary length process $(L_t, R_t)_{t\in\bbR}$, which tracks the change of the left and right boundaries of $\eta'((-\infty,t])$ relative to time $0$,  is a bi-infinite  Brownian motion with covariance~\eqref{eq: gamma correlated BM}. Moreover, there is a measurable function $F_\infty$ such that  $(\bC, \phi, \eta',0, \infty){/}{\sim_\gamma} = F_\infty((L,R))$ a.s.
\end{theorem}

\begin{definition}\label{def:Quantum-cell}
    Let $\mathcal{C}=(\bC, \phi, \eta')$ be the curve-decorated $\gamma$-quantum cone with boundary length process $(L, R)$ as in Theorem~\ref{thm: quantum cone mot}.
    An area-$a$ quantum cell is the   curve-decorated LQG surface 
    $\cC_a=(\eta'([0, a]), \phi,\eta'|_{[0, a]},\\   \eta'(0), \eta'(a), x_\rmL, x_\rmR)/\mathord\sim_\gamma$, where $x_\rmL=\eta'(t_\rmL)$, $x_\rmR=\eta'(t_\rmR)$, $t_\rmL$ (resp.~$t_\rmR$) is the time when $L_t$ (resp.~$R_t$) attains its infimum for $0\le t\le a$. 
    We write $\MC_a$ for the law of the curve-decorated LQG surface $\cC_a$. 
\end{definition}
We call the four boundary arcs from $\eta'(0)$ to $x_\rmL$,  from $x_\rmL$ to $\eta'(a)$, from $\eta'(a)$ to $x_\rmR$ and from $x_\rmR$ to $\eta'(0)$ the bottom left, top left, top right, bottom right boundary arcs of $\cC_a$, respectively. We write $\MC = \int_0^\infty \MC_a\,\d a$, and let $\MC = \int_{\bbR_+^4} \MC(\ell_1,\ell_2,\ell_3,\ell_4)\;\d\ell_1\,\d\ell_2,\d\ell_3\,\d\ell_4$ be the disintegration such that the bottom left, top left, top right, bottom right boundary arcs have quantum lengths $\ell_1,\ell_2,\ell_3,\ell_4$, respectively, when sampled from $\MC(\ell_1,\ell_2,\ell_3,\ell_4)$. This disintegration can be defined on the level of Brownian paths. 

Note that as explained in~\cite[Section 2.4]{AY23wholeplane},   the quantum cell $\cC_a$ considered above is measurable with respect to $(L,R)|_{[0,a]}$. We write $F$ for the map such that $F((L,R)|_{[0,a]}) = \cC_a$, a.s. 
Furthermore, as explained in~\cite[Section 3.3.1]{ASYZ24}, $F$ satisfies reversibility and  concatenation compatibility, and can therefore be extended to Brownian path trajectories by conformal welding.

\begin{definition}\label{def:MD}
    Let $\cD=(\bH, \phi,  0, \infty)$ be an embedding of a quantum disk from $\mathcal{M}_2^{\mathrm{disk}}(\frac{\gamma^2}{2};a,b)$ as in Definition~\ref{def-quantum-disk} and~\eqref{eq:QD-dis}, and let $\eta'$ be a  space-filling  $\SLE_{\kappa'}(0;\kappa'/2-4)$ from 0 to $\infty$ independent from $\phi$ and  parameterized by LQG area with respect to $\phi$. We write   $\MD(a, b)$ for the law of the LQG surface $(\cD,\eta'){/}{\sim_\gamma}$. Consider the left boundary of $\eta'([0,t])$. We let $L_t$ be the quantum length of the segment inside $\bbH$ plus the quantum length of $(-\infty,0)\backslash \eta'([0,t])$. Let $R_t$ be the quantum length of the right boundary of $\eta'([0,t])$.  
\end{definition}
\begin{theorem} 
    \label{thm: quantum disk mot}
    For any $a, b \in (0, \infty)$, the law of the boundary length process $(L, R)$ of $\MD(a, b)$ as in Definition~\ref{def:MD} is $C_\gamma B_{a, bi}^{\gamma, \bR^2_+}$ for some $C_\gamma >0$ depending only on $\gamma$. Moreover, the curve-decorated LQG surface $(\cD,\eta'){/}{\sim_\gamma}$ can a.s.\  be recovered from $(L,R)$ through the map $F$ in Definition~\ref{def:Quantum-cell}.
\end{theorem}

\begin{proof}
    The first claim is from {\cite[Remark 7.10]{ConfWeldDisks}}. For the second claim, let $T$ be the duration of $(L,R)$. By the definition of the above boundary length process, for each rational interval $[a,b]\subseteq (0,T)$, $(\eta'([a,b]),\phi,\eta'|_{[a,b]})$ can a.s.\  be recovered from the map $F$, and the claim further follows by sending $a\downarrow 0$, $b\uparrow T$ and using the concatenation compatibility of $F$~\cite[Lemma 2.15]{AY23wholeplane}.
\end{proof}

Next we consider the boundary length processes associated with quantum triangles. We start with the case where $\gamma\in(\sqrt{2},2)$. Embed a sample from $\QT(2-\frac{\gamma^2}2, \gamma^2, \frac{\gamma^2}2)$ as $(D, \phi, -i, \infty, -1)$ such that the points $-i, \infty, -1$ correspond to the vertices having weights $2-\frac{\gamma^2}2, \gamma^2, \frac{\gamma^2}2$ respectively, the connected component of $D$ having $-1$ on its boundary is $\bbH$, and $\ol{D \backslash \bbH} \cap \bbR = 0$. See Figure \ref{fig: thick and thin QT}.  Independently sample a space-filling $\SLE_{\kappa'}(\frac{\kappa'}2-4;\frac{\kappa'}2-4)$ curve in the connected components of $D \backslash \bbH$ where the force points are located immediately to the left and right of the starting point, and concatenate them then to get a space-filling curve from $0$ to $-i$. Further independently sample a space-filling $\SLE_{\kappa'}(\frac{\kappa'}{2}-4;\frac{\kappa'}{2}-4)$ curve in $\bbH$ from $-1$ to $0$ with force points at $\infty$ and $(-1)^+$. Let $\eta'$ be the concatenation of these curves parametrized by quantum area, such that it is a space-filling curve in $D$ from $-1$ to $-i$.  Let $T$ be the duration of $\eta'$ and let $\tau$ be the first time $\eta'$ hits $\infty$. For $t \leq T$ let $R_t$ be the  quantum length of the right boundary arc of $\eta'([0,t])$. For $t \leq \tau$, consider the left boundary arc of $\eta'([0,t])$; let $L_t^+$ (resp.\ $L_t^-$) be the quantum length of the segment inside $D$ (resp.\ on $\partial D$), and let $L_t = L_t^+ - L_t^-$. For $t \in (\tau, T]$ let $L_t$ be $L_\tau$ plus the quantum length of the left  boundary arc of $\eta'([\tau,t])$. We write $\MT(a,b,c)$ for the law of the curve-decorated LQG surface $(D,\phi,-i,\infty,-1,\eta'){/}{\sim_\gamma}$ considered as above, where we further disintegrate over $\QT(2-\frac{\gamma^2}2, \gamma^2, \frac{\gamma^2}2)$ such that the boundary arcs of $D$ from $-i$ to $\infty$, from $-i$ to $-1$ and from $-1$ to $\infty$ have length $a,b,c$, respectively. For $\gamma\in(0,\sqrt{2})$, we embed a sample from $\QT(2-\frac{\gamma^2}2, \gamma^2, \frac{\gamma^2}2)$ as $(\bbH, \phi, 0, \infty, -1)$ (i.e., there will no longer be beads $D\backslash\bbH$ in the above definition) and define $(L,R)$ and $\MT(a,b,c)$ in the same way.

\begin{figure}
    \centering
    \includegraphics[width=0.8\linewidth]{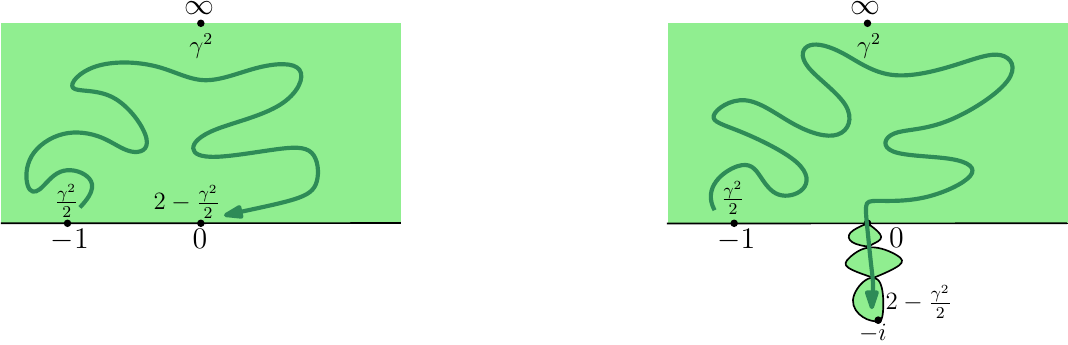}
    \caption{\textbf{Left}: Illustration for $\mathrm{QT}(2-\frac{\gamma^2}{2}, \gamma^2, \frac{\gamma^2}{2})$ along with the space-filling curve $\eta'$, embedded as $(\mathbb H, \phi, 0, \infty, -1)$, for $0<\gamma < \sqrt{2}$; \textbf{Right}: Illustration for $\mathrm{QT}(2-\frac{\gamma^2}{2}, \gamma^2, \frac{\gamma^2}{2})$, embedded as $(D, \phi, -i, \infty, -1)$, for $\gamma > \sqrt{2}$.  }
    \label{fig: thick and thin QT}
\end{figure}

\begin{theorem}\label{thm:ASYZ3.9}
    Let $\gamma\in(0,2)\backslash\{\sqrt2\}$. There is a constant $C_\gamma>0$ such that for every $b>0$, the above boundary length process for 
    \begin{equation}\label{eq:MT-law}
         \int_0^\infty \MT(\ell,b,\ell+a)\,\d\ell
    \end{equation}
   has law $C_\gamma B_{0,-a+bi}^{\gamma;\bbH}$. Furthermore, it is possible to define the map $F$ introduced below Definition~\ref{def:Quantum-cell} such that it a.s.\ recovers the decorated LQG surface from its boundary length process.
\end{theorem}

\begin{proof}
    By~\cite[Proposition 3.9]{ASYZ24} (where we swap the role of $L$ and $R$ there), the boundary length process for $\int_{0}^\infty\int_0^\infty \MT(a',b,c')\,\d a'\,\d c'$ has the law $C_\gamma\int_{-\infty}^\infty B_{0,r+bi}^{\gamma;\bbH}\,\d r$. From the definition of the boundary length process, if $T$ is the duration of $(L,R)$, then in the above two integrals $r=L_T = a'-c'$. The first claim then follows by a change of variables $(\ell,a)=(a',c'-a')$ and disintegration over $a$. The second claim is from the last claim of~\cite[Proposition 3.9]{ASYZ24}.
\end{proof}

\subsection{The limiting surface and an interface resampling property}
\label{sec: proof sec 4 main thm}

We start with the following result. See Figure \ref{fig: cell welding} for an illustration.

\begin{proposition}\label{prop: MC+MD}
   Let $a,b>0$. Let $(D,\phi,-i,\infty,-1,\eta')$ (for $\gamma\in(\sqrt2,2)$) or $(\bbH,\phi,0,\infty,-1,\eta')$ (for $\gamma\in(0,\sqrt2)$) be an embedding of the surface from 
     \begin{equation}\label{eq:MT-law-1}
         \int_0^\infty a\MT(\ell,b,\ell+a)\,\d\ell
    \end{equation}
    as described in the paragraph before Theorem~\ref{thm:ASYZ3.9}. Pick $r$ from the uniform probability measure on $[0,a]$, and mark the point $w$ on $(-\infty,-1)$ such that  $(-\infty,w)$ has quantum length $r$.  Let $T$ be the duration of $\eta'$ and $\tau$ be the time when $\eta'$ hits $w$. Let $\cD_1 = (\eta'([0,\tau]),\phi,\eta|_{[0,\tau]}) {/}{\sim_\gamma}$ and $\cD_2 = (\eta'([\tau,T]),\phi,\eta|_{[\tau,T]}) {/}{\sim_\gamma}$. Then for some constant $C\in(0,\infty)$, the LQG surfaces $(\cD_1,\cD_2)$ has the law
    \begin{equation}\label{eq:MT-law-2}
        C\int_0^a\int_0^\infty\int_0^b \int_0^\infty\MD(a+\ell-r;b+s-t) \times \MC(r,\ell,t,s)\,\d \ell\,\d t\,\d s\,\d r.
    \end{equation}
\end{proposition}

\begin{figure}[t]
    \centering
    \includegraphics[scale=0.9]{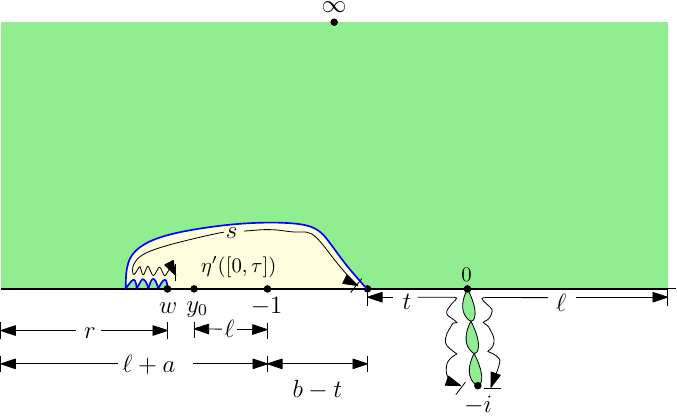}
    \caption{This figure shows the boundary lengths in the conformal welding in Proposition~\ref{prop: MC+MD} in the case where $\gamma > \sqrt{2}$. The light yellow part corresponds to $\mathrm{MD}(a+\ell-r;b+s-t)$ and the light green part corresponds to $\mathrm{MC}(r,\ell,t,s)$. In the $\gamma < \sqrt{2}$ case, there are no  beaded disks and the blue interface is disjoint from $\bbR$ except for its starting and ending points. In the definition of $\MA^{1,1\dagger}$, we shall further glue the arc from $-i$ to $+\infty$ with the arc from $-1$ to $y_0$ to obtain a surface with annular topology as in Figure~\ref{fig: two ways of welding}.} 
    \label{fig: cell welding}
\end{figure}

Proposition~\ref{prop: MC+MD} is a straightforward consequence of Theorem~\ref{thm:ASYZ3.9} and the following Brownian path decomposition. 

\begin{proposition}\label{prop:BM-MC+MD}
    Let $a,b>0$, and $Z=(L,R)$ be a sample from 
     $a B^{\gamma, \bH}_{0, -a+ib}$ of (random) duration $T$ such that $L_T=-a$. Let $L_{\mathrm{inf}}=\inf_{0\leq u\leq T}L_u$, and $c$ be sampled from the uniform probability measure on $(-L_{\mathrm{inf}}-a,-L_{\mathrm{inf}})$. Let $\tau = \inf\{u\in[0,T]:L_u=-c\}$. Let $Z^1 = Z|_{[0,\tau]}+c$ and $Z^2(t) = Z(t+\tau)-Z(\tau)$ for $t\in[0,T-\tau]$. Then for some constant $C_\gamma>0$ depending only on $\gamma$, $(Z^1,Z^2) = ((L^1,R^1),(L^2,R^2))$ has the joint law 
    \begin{equation}\label{eq:prop:BM-MC+MD}
          C\int_0^\infty \int_0^\infty \mathds{1}_E\big(B_{c,id}^{\gamma,\bbR_+^2}\times B^{\gamma,\bC}_{0,(c-a)+(b-d)i}\big)\,\d d\,\d c,
    \end{equation}
    where   $$E=\{L_{T_2}^2- \inf_{t\in [0,T_2]} L_t^2< c \}\cap \{-\inf_{t\in[0,T_2]}R_t^2\leq d \}; \ \ T_2 \text{ is the duration of }Z^2.  $$
\end{proposition}

\begin{proof}
   For given $c\in(0,\infty)$, consider the events
    \begin{equation*}
        E_c = \{Z: -L_{\mathrm{inf}}-a < c < -L_{\mathrm{inf}}\}; \ \ \ \  E_\rmR = \{Z: \inf_{0 \le u \le T}R_u \geq 0\}.
    \end{equation*}

    By definition,  the joint  law of $(Z,c)$ is 
    \begin{equation}
     \label{eq: B first passage decomposition-000}
     ( \frac{1}{a} \mathds{1}_{-L_{\mathrm{inf}}-a< c< -L_{\mathrm{inf}}} \d c)\cdot aB^{\gamma, \bH}_{0, -a+ib}  = 
       \mathds{1}_{E_c} B^{\gamma, \bH}_{0, -a+ib} \,\d c .
    \end{equation}
     By Lemma~\ref{lemma: first and last passage decomposition}, where we decompose the path according to the time $\tau$,
    \begin{equation}
        \label{eq: B first passage decomposition-001}
        B^{\gamma, \bH}_{0, -a+ib} = B^{\gamma,  \bbR_+^2-c}_{0, -a+ib} + \mathbb{a} \int_0^\infty \big( B^{\gamma, \bbR_+^2-c}_{0, -c+i d} \oplus B^{\gamma, \bH}_{-c+i d, -a+ib} \big) \d d.
    \end{equation}
Under the event $E_c$, $L_{\mathrm{inf}}<-c$, and by definition  $B^{\gamma,  \bbR_+^2-c}_{0, -a+ib}[E_c]=0$. Furthermore, $B^{\gamma, \bbR_+^2-c}_{0, -c+i d} \oplus B^{\gamma, \bH}_{-c+i d, -a+ib} [\cdot] = B^{\gamma, \bbR_+^2-c}_{0, -c+i d} \oplus B^{\gamma, \bC}_{-c+i d, -a+ib} [\cdot\cap  E_\rmR]$. Combining~\eqref{eq: B first passage decomposition-000} with~\eqref{eq: B first passage decomposition-001}, the joint law of $(Z,c)$ is
\begin{equation}
     \mathds{1}_{E_c} B^{\gamma, \bH}_{0, -a+ib}\,\d c = \mathbb{a}   \int_0^\infty \mathds{1}_{ E_c\cap E_\rmR} \big( B^{\gamma, \bbR_+^2-c}_{0, -c+i d} \oplus B^{\gamma, \bC}_{-c+i d, -a+ib} \big)  \d d\,\d c.
\end{equation}
 On the other hand, in our definition of $Z^1$ and $Z^2$, the event $E_c\cap E_\rmR$ for $(Z,c)$ in the above decomposition is precisely the event $E$ for $(Z^1,Z^2)$. This finishes the proof.
\end{proof}

\begin{proof}[Proof of Proposition~\ref{prop: MC+MD}]
    Let $Z$ be the boundary length process. Following from our definition of the boundary length processes for samples from $\MT$, the quantum length of $(-\infty,-1)$ is $-L_{\mathrm{inf}} = -\inf_{0\leq u\leq T}L_u = a+\ell$. The quantum length $c$ of $(w,-1)$ is uniformly chosen in $(-L_{\mathrm{inf}}-a,-L_{\mathrm{inf}})$, and $\tau=\inf\{u\geq 0:L_u=-c\}$.  Therefore by Theorem~\ref{thm: quantum disk mot}, Proposition~\ref{prop:BM-MC+MD} and the concatenation property of the map $F$, the surfaces $(\cD_1,\cD_2)$ are samples from $\MD$ and $\MC$ with appropriate boundary lengths.  Furthermore, observe that for a cell from $\MC$ generated by a Brownian path $(\wt L_t,\wt R_t)|_{t\in[0,\wt T]}$, its top left and bottom right boundary arcs have lengths $\wt L_{\wt T}-\inf_{0\leq u\leq \wt T}\wt L_u$ and $-\inf _{0\leq u\leq \wt T}\wt R_u$, respectively. Let $r,t,s$ be the bottom left, top right and bottom right arc lengths of $\cD_2$, respectively, and note that the top left boundary arc of $\cD_2$ has quantum length $\ell$. Then the left (resp.\ right) boundary arc of $\cD_1$ has quantum length $c=a+\ell-r$ (resp.\ $d=s+b-t$). By~\eqref{eq:prop:BM-MC+MD}, as we further disintegrate over the bottom right boundary arc length for $\cD_2$, the joint law of $(\cD_1,\cD_2)$ is now a constant times
    \begin{equation}
        \int_0^\infty\int_0^\infty\int_0^\infty \int_0^\infty \mathds{1}_{\ell<c;s\leq d} \MD(c;d)\times \MC(a+\ell-c,\ell,b+s-d,s)\, \d \ell\, \d s\, \d d\, \d c.
    \end{equation}
    The claim then follows from a change of variables $r=a+\ell-c$, $t=b+s-d$.
\end{proof}

Given a pair of certain LQG surfaces, following ~\cite{She16a,MoT}, there exists a way to \emph{conformally weld} them together according to the length measure; see e.g.~\cite[Section 4.1]{AHS24} and~\cite[Section 4.1]{QuanTrig} for more explanation. Proposition~\ref{prop: MC+MD} can be viewed as the conformal welding result of a quantum cell $\cD_2$ from $\MC$ with a quantum disk $\cD_1$ from $\MD$ sampled from~\eqref{eq:MT-law-2}, where we weld the bottom right side of the quantum cell $\cD_2$ to the right side of   the quantum disk $\cD_1$ to get a quantum triangle according to the boundary quantum length measure. By symmetry and the reversibility of the map $F$, one can also define the conformal welding of $\cD_1$ and $\cD_2$ where we weld the  top left boundary arc of the quantum cell $\cD_2$ to the left boundary arc of the quantum disk $\cD_1$. 

Let $a,b>0$. Let $\cT$ be a space-filling $\SLE_{\kappa'}$ decorated quantum triangle sampled from~\eqref{eq:MT-law}, such that the boundary arc joining the weight $\gamma^2$ and weight $\frac{\gamma^2}{2}$ vertices is longer than the boundary arc joining the weight $2-\frac{\gamma^2}{2}$ and weight $\frac{\gamma^2}{2}$ vertices. We conformally weld these two boundary arcs together according to the LQG boundary length as in Figure~\ref{fig:weld-annulus-8}, such that the weight $2-\frac{\gamma^2}{2}$ vertex is identified with the weight $\frac{\gamma^2}{2}$ vertex. The LQG surface we get from this welding has annular topology with boundary lengths $a$ and $b$. Indeed, if $\gamma\in(\sqrt2,2)$, then the beaded part of $\cT$ is glued to the thick quantum triangle part of $\cT$ as in Definition~\ref{def-qt-thin}, so the resulting surface is   connected. Moreover, by Proposition~\ref{prop: MC+MD}, $\cT$ can be viewed as the welding of $\cD_1$ and $\cD_2$ along their left side and bottom left side, and therefore by symmetry and the discussion in the previous paragraph the above conformal welding for $\cT$ is well-defined. We embed the output curve-decorated LQG surface, whose law is denoted by $\MA^{1, 0\dagger}(a,b)$, as $(\cA_\tau,\phi,\eta,-i,\eta') {/}{\sim_\gamma}$ such that the welding interface $\eta$ starts from $-i$, and $\eta'$ is the space-filling loop on $\cA_\tau$ carried from the space-filling curve on $\cT$ as in the definition of $\MT$. For $(\phi,\eta,\eta')$ sampled from $\MA^{1, 0\dagger}(a,b)$ and embedded as above, we weight their joint law by $a$, which is the quantum length of $C_\tau$, and sample another point $y\in C_\tau$ by the probability measure proportional to the LQG length measure with respect to $\phi$. Further let $\wt\eta$ be the part of the boundary of $\eta'$ stopped upon hitting $y$ in the interior of $A_\tau\backslash\eta$. We write $\MA^{1, 1\dagger}(a,b)$ for the law of the curve-decorated LQG surface  $(\cA_\tau,\phi,\eta,-i,\wt\eta,y,\eta') {/}{\sim_\gamma}$.  By Theorem~\ref{thm:ASYZ3.9}, $\|\MA^{1, 1\dagger}(a,b)\| = C_\gamma a\|B_{0,-a+bi}^{\gamma;\bbH}\|<\infty$. Note that by Proposition~\ref{prop: MC+MD}, a surface from $\MA^{1, 1\dagger}(a,b)$ can be viewed as  the conformal welding of $\cD_1$ and $\cD_2$  sampled from~\eqref{eq:MT-law-2}
 where we  weld the bottom right side of  $\cD_2$ to the right side of    $\cD_1$, and also the  top right boundary arc of  $\cD_2$ to the left boundary arc of  $\cD_1$. See Figure \ref{fig: two ways of welding} for an illustration.

\begin{figure}
    \centering
    \begin{tabular}{cc}
         \includegraphics[width=0.5\linewidth]{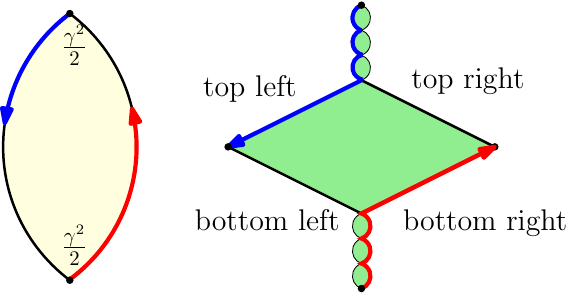}    & \includegraphics[scale=0.66]{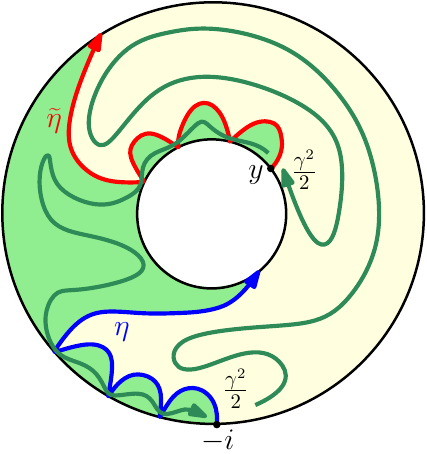}  \\
    \end{tabular}
    \caption{The welding in the definition of $\MA^{1, 1\dagger}$ for $\gamma \in( \sqrt{2},2)$. There are two ways of welding $\mathcal D_1$ (a quantum disk, as left) and $\mathcal D_2$ (a quantum cell, as right). One way is to weld along the blue part, and the other is to weld along the red part. The two welding operations are commutative, and when we weld along one boundary arc, we get a quantum triangle, as in Figure \ref{fig: cell welding}. For $\gamma < \sqrt{2}$, there are no beaded disks, $\cD_2$ is simply connected, and the interfaces $\eta,\wt\eta$ are disjoint from the boundary except for the starting and ending points. }
    \label{fig: two ways of welding}
\end{figure}

\begin{proposition}\label{prop:curve-resample}
Under the finite measure $\MA^{1, 1\dagger}(a,b)$, we have the following description of the conditional law of $\eta$ and $\wt\eta$ given $\phi$.
\begin{enumerate}[(i)]
    \item Given $\phi$ and $\eta$, the conditional law of $(y,\wt\eta)$ is as follows. The distribution of $y\in  C_\tau$ is the probability measure proportional to the LQG length measure with respect to $\phi$; given $y$, $\wt\eta$  has the law as      $\SLE_\kappa(-\frac{\kappa}{2},\kappa-2,-\frac{\kappa}{2};\frac{\kappa}{2}-2)$ in $\cA_\tau\backslash\eta$ from $y$ and targeted at $(-i)^+$, with  the force points   located at $y^-$, the endpoint of $\eta$, the leftmost point of $\eta\cap C_0$ and $y^+$, respectively. 
    \item  Conditioned on $(y,\wt\eta)$ and $\phi$, $\eta$ has the law as   $\SLE_\kappa(-\frac{\kappa}{2},\kappa-2,-\frac{\kappa}{2};\frac{\kappa}{2}-2)$ in $\cA_\tau\backslash\wt\eta$ from $-i$ and targeted at $y^+$, with the force points located at $(-i)^-$, the endpoint of $\wt\eta$, the the leftmost point of $\wt\eta\cap C_\tau$ and $(-i)^+$, respectively.
\end{enumerate}
\end{proposition}

Note that in the above $\SLE_\kappa(\underline{\rho})$ description of the conditional law of $\wt\eta$ given $\eta$ and $\phi$, $\wt\eta$ hits its continuation threshold upon hitting $C_0$, and the same holds for $\wt\eta$ and $\eta$ swapped.

\begin{proof}
The first claim follows from the definition of $\MA^{1, 1\dagger}(a,b)$ along with the imaginary flow line description in Theorem~\ref{thm-ig1}.   The second claim follows by Proposition~\ref{prop: MC+MD} and applying a map $z\mapsto e^{-2\pi \tau}/z$ and a left-right reflection symmetry to Proposition~\ref{prop: MC+MD}.
\end{proof}

\section{Imaginary geometry in the annulus}\label{sec: IG annulus}


 In this section we will construct the variant of imaginary geometry on the annulus that appears in Theorem \ref{thm: quantum annulus mot}.
We work with an annulus Dirichlet GFF with explicit boundary values. We construct the flow line starting from the outer boundary in Proposition~\ref{prop:def-flow-line-strip}, and in Lemma~\ref{lem:flow-line-continuity} we prove that it is a continuous curve that either crosses the annulus and hits the inner boundary or merges with itself to form a non-contractable loop. Furthermore, in Lemma~\ref{lem:ig-def-SLEloop} and Definition~\ref{def:ig-flow-line-annulus}, under the positive probability event \textbf{$E^\cros$} that there is an  angle $\frac{\pi}{2}$ crossing flow line, we construct the space-filling annulus $\SLE_{\kappa'}$ loop in the annulus, which is a novel object. Finally in Lemma~\ref{lem:ig-curve-resample}, we prove a resampling property of annulus imaginary geometry flow lines similar to Proposition~\ref{prop:curve-resample}, which will allow us to identify the space-filling curve $\eta'$ in $\MA^{1,1,\dagger}$ with the space-filling annulus $\SLE_{\kappa'}$ loop in Definition~\ref{def:ig-flow-line-annulus} in the next section.

Imaginary geometry in non-simply connected domains has already been considered in the papers~\cite{ImagGeo1,ImagGeo4}, but for a different field than what we consider here, and the long-time behavior of flow lines as well as the space-filling counterflowlines were not discussed. In~\cite{ImagGeo4}, the imaginary geometry fields are of the form $h-\alpha\arg(\cdot)-\beta\log|\cdot|$ with $\alpha>-\chi$, and there is a $2\pi(\chi+\alpha)$-branch cut coming from the $\arg(\cdot)$ function. On the other hand, the field we are working on corresponds to the $\alpha=-\chi$ case, where there is no branch cut and results from~\cite{ImagGeo4} are not directly applicable. To circumvent this, we will first work with the universal cover (strip) and later project back to the annulus, and the fields on the annulus are multi-valued functions.

We fix the modulus $\tau\in (0, \infty)$, and recall that $\cS_\tau = \bbR\times (0,\tau)$.
The map $p:z\mapsto e^{i2\pi z}$ provides a holomorphic universal covering map $\cS_\tau\to \cA_\tau$ for any $\tau\in (0, \infty)$.  Throughout this  section we assume $\kappa\in(0,4)$, and $\kappa',\lambda,\chi,\lambda'$ are as in~\eqref{eq:ig-parameters}. Let $h_0$ be a zero boundary GFF in $\cA_\tau$, and let $h^\rmS  = h_0\circ p+\lambda'-\chi\pi$. We will work with the periodic field $h^\rmS$ in the universal cover $\cS_\tau$, and later project back to $\cA_\tau$. Note that by the rotation invariance in law of the GFF on $\cA_\tau$, $h^\rmS$ is translation invariant in law, in the sense that $h^\rmS(\cdot+x)\overset{d}{=}h^\rmS(\cdot)$ for any deterministic $x$.

\begin{proposition}\label{prop:def-flow-line-strip}
    Let $h^\rmS$ be the field as above and $\theta\in \bbR$ such that $\lambda'-\chi\pi+\theta\chi\in(-\lambda,\lambda)$. Let $x\in\bbR$. There exists a unique coupling between $(h^\rmS,\eta)$ where $\eta$ is a curve starting from $x$,  and stopped until either hitting $\bbR+i\tau$ or one of its translates $\eta+k$ where $k\in\bbZ$, such that the following holds:
    \begin{enumerate}[(i)]
        \item For each stopping time $\sigma$ of $\eta$ before its duration, $\eta([0,\sigma])$ is a local set of $h^\rmS$, in the sense that the conditional law of $h^\rmS$ given $\eta|_{[0,\sigma]}$ has the following description. Let $\cS_{\eta,\sigma} = \cS_\tau\backslash \cup_{k\in\bbZ}(k+\eta([0,\sigma]))$, and $\cA_{\eta,\sigma}$ be the annulus such that there is a conformal map $\varphi_{\eta,\sigma}$ with $\cA_{\eta,\sigma} = p\circ\varphi_{\eta,\sigma}\circ \cS_{\eta,\sigma}$. Let $\mathfrak{h}_{\eta,\sigma}$ be the harmonic function on $\cS_{\eta,\sigma}$ whose boundary values agrees with $h^\rmS$ on $\partial\cS_\tau$ and is given by the flow line boundary conditions on $ \cup_{k\in\bbZ}(k+\eta([0,\sigma]))$ as in Figure~\ref{fig: flow line boundary data} and~\cite[Figure 1.10]{ImagGeo1}. Then $h^\rmS|_{\cS_{\eta,\sigma}} =  {h}_{\eta,\sigma}\circ p\circ \varphi_{\eta,\sigma} + \mathfrak{h}_{\eta,\sigma}$, where ${h}_{\eta,\sigma}$ is a zero boundary GFF in $\cA_{\eta,\sigma}$ independent from $\mathfrak{h}_{\eta,\sigma}$.
        \item The curve $\eta$ is measurable with respect to $h^\rmS$.
    \end{enumerate}
\end{proposition}
We call the curve $\eta$ in the proposition above \emph{the $\theta$-angle flow line} of $h^\rmS$ started from $x$. Rotating by $\pi$ and applying~\eqref{eq:ig-change-coord}, we can also define $\theta$-angle flow lines of $h^\rmS$ started from $x+i\tau$.

\begin{proof}
    Let $\wt h^\rmS_1:=\wt h^\rmS$ be a Dirichlet GFF on $\cS_\tau$ with the same boundary value as $h^\rmS$. By~\cite[Proposition 3.4]{ImagGeo1} with the field $h_0$ on $\cA_\tau$ and the conformal invariance of GFF, for each rectangle $\cR_y:=[y-\frac{2}{5},y+\frac{2}{5}]\times [0,\tau]$, the law of the fields $h^\rmS|_{\cR_y}$ and $\wt h^\rmS|_{\cR_y}$ are mutually absolutely continuous. Now consider the $\theta$-angle  flow line $\wt\eta$ of $\wt h^\rmS$ as discussed in Section~\ref{sec: prelim ig} started from $x$, targeted at $+\infty$, and stopped until exiting $\wt \cR_x:=[x-\frac{1}{3},x+\frac{1}{3}]\times [0,\tau]$. Following the characterization of the local sets in~\cite[Lemma 3.6(ii)]{ImagGeo1},  $\wt\eta$ is measurable with respect to $\wt h^\rmS|_{\cR_x}.$ Then following~\cite[Proposition 2.16]{ImagGeo4}\footnote{This proposition is stated for the whole plane GFF modulo global additive constants, and we can first project to the annulus and apply the analogous statement for Dirichlet GFF on domains and then lift to the universal cover.}, using the aforementioned local absolute continuity between $h^\rmS$ and $\wt h^\rmS$, we can generate the flow line $\eta$ of $h^\rmS$ with properties (i) and (ii) up until time $\sigma_1$ when $\eta$ exits $\wt \cR_x$.  Assume that we have constructed $\eta$ up until time $\sigma_j$ before its duration with $\mathrm{Re} \,\eta(\sigma_j) = y_j\in \{x-j/3,x-(j-1)/3,...,x,...,x+j/3\}$. Consider the Dirichlet GFF $\wt h^\rmS_{j+1}$ on $\cS_\tau\backslash (\cup_{k=-3j}^{3j}(\eta([0,\sigma_j])+k))$ with flow line boundary conditions on $\cup_{k=-3j}^{3j}(\eta([0,\sigma_j])+k)$ and the same boundary value as $h^\rmS$ on $\partial \cS_\tau$. Then we can use the same local absolute continuity argument to compare the conditional law of $h^\rmS|_{\cR_{y_j}}$ given $\eta|_{[0,\sigma_j]}$ with the law of $\wt h^\rmS_j|_{\cR_{y_j}}$ and further extend $\eta$ to time $\sigma_{j+1}$ when $\eta$ either hits one of its translates or exits $\cR_{y_j}$. By an iteration of this, we have shown the existence of the flow lines until its duration as desired.  

    To prove the uniqueness of the above coupling, we adapt the arguments in~\cite[proof of Proposition 7.33]{ImagGeo1} and~\cite[proof of Theorem 1.1]{ImagGeo4}. Suppose $\wh\eta$ is another random curve satisfying the above two properties. We fix angles $\theta_1<\theta_2<...<\theta_n$ such that for the flow lines $\wt\eta_j$ of $\wt h^\rmS$ started from $x$ with angle $\theta_j$, $\wt \eta_j$ a.s.\ intersects $\wt \eta_{j+1}$ infinitely many times for each $j=0,...,n$ where $\eta_0=\wt\eta_0=(x,\infty)$ and $\eta_{n+1}=\wt\eta_{n+1}=(-\infty,x)$~\cite[Theorem 1.5 and Lemma 7.34]{ImagGeo1}. Without loss of generality assume $\theta=\theta_{j_0}$. Let $\eta_j$ be the flow lines of $h^\rmS$ with angle $\theta_j$ constructed as in the preceding paragraph. Then using absolute continuity between the laws of $h^\rmS|_{\cR_x}$ and $\wt h^\rmS|_{\cR_x}$ and the flow line interacting rules in~\cite[Theorem 1.5]{ImagGeo1}, almost surely, for each $j=0,...,n$, before exiting $\cR_x$, $\eta_j$ stays to the right of $\eta_{j+1}$, and for every $r>0$, $\eta_j$ hits and bounces off $\eta_{j+1}$ before exiting $B(x;r)$.    From the locality property (i), by~\cite[Lemma 3.6(ii)]{ImagGeo1} and property (ii), if we let $\wh\sigma_x$ be the time when $\wh \eta$ exists $\wt \cR_x$, then $\wh\eta|_{[0,\wt\sigma_x]}$ is measurable with respect  to $h^\rmS|_{ \cR_x}$. Therefore by~\cite[Proposition 6.1 and Lemma 6.2]{ImagGeo1}, the same set of flow line intersecting rules as described in~\cite[Theorem 1.5]{ImagGeo1} applies for $(\eta_1,...,\eta_n,\wh\eta)$ before any of those curves exits $\cR_x$. In particular, for $j>j_0$ (resp.\ $j<j_0$), if $\eta_j$ hits $\wh \eta$ on the right (resp.\ left) side of $\wh\eta$, then $\eta_j$ crosses $\wh\eta$ and never crosses back, while for $j=j_0$, $\eta_j$ merges with $\wh\eta$ upon intersecting. In particular, for each $j=0,...,n$, $\wh\eta$ can only intersect the interior of one pocket formed by $\eta_j$ and $\eta_{j+1}$, and $\wh\eta$ can only intersect a deterministic finite number of connected components of $\cR_x\backslash (\cup_{j=1}^n\eta_j)$ and further merge into $\eta_{j_0}$ before exiting $\cR_x$. This implies that $\wh\eta$ agrees with $\eta_{j_0}$ before exiting $\cR_x$, and the same arguments can be extended to all times.
\end{proof}
Note that from the measurability property (ii) in Proposition~\ref{prop:def-flow-line-strip} and our construction of $h^\rmS$, for $k\in\bbZ$, if the angle $\theta$ flow line of $h^\rmS$ started from $x\in\bbR$ is $\eta$, then the angle $\theta$ flow line of $h^\rmS$ started from $x+k$ is $\eta+k$.

 Next we prove the following end point continuity of flow lines. For simplicity we assume $\theta=\pm\frac{\pi}{2}$. See Figure~\ref{fig:flowline-cts} for an illustration.
\begin{lemma}\label{lem:flow-line-continuity}
    Let $\eta$ be the angle $\theta=\pm\frac{\pi}{2}$ flow line started from $x\in\bbR$ as in Proposition~\ref{prop:def-flow-line-strip}. Then almost surely, $\eta$ either hits $\bbR+i\tau$ or one of its translates $\eta+k$ for $k\in\bbZ$, and is a compact simple curve.
\end{lemma}
\begin{proof}
    Let $h^\rmS$, $\wt h^\rmS$ and $\cR_x$ be as in the proof of Proposition~\ref{prop:def-flow-line-strip}. 
    We assume on the contrary and assume that with positive probability, $\eta$ never hits  any of its translates $\{\eta+k:k\in\bbZ\}$ or $\bbR+i\tau$. By the periodicity of $h^\rmS$, under this event, the same holds for all $\eta+k$ where $k\in\bbZ$.   Then from our construction of $\eta$, 
    under this event,  there exists some $z_0\in \bbR^2$, such that for some $\delta_0\in(0,0.1)$, there are infinitely many segments of $\cup_{k\in\bbZ}(\eta+k)$ starting at some point on $\partial B(z_0;2\delta_0)$, crossing the annulus $B(z_0;2\delta_0)\backslash B(z_0;\delta_0)$ and further leaving $B(z_0;2\delta_0)$ at some point on $\partial B(z_0;2\delta_0)$. Without loss of generality assume that $\mathrm{Re}\, z_0\in (-0.1,0.1)$. 
    Consider the set $\Lambda_{\wt h^\rmS}$ of the flow lines of $\wt h^\rmS|_{\cR_0}$ started from rational points in $\cR_0$ of angle $\pm\frac{\pi}{2}$ stopped upon leaving $\cR_0$, and define $\Lambda_{h^\rmS}$ analogously via the local absolute continuity between the laws of $\wt h^\rmS$ and $h^\rmS$. Then following the same proof as in~\cite[Proposition 3.5]{ImagGeo4}, as we condition on $\eta$ as well as its translates and read off from the flow line boundary conditions of $\eta+k$, flow lines of angles $\pm\frac{\pi}{2}$ do not cross each other. Thus almost surely, all the flow lines of $h^\rmS$ of angle $\pm\frac{\pi}{2}$ started from rational points do not cross any of $\eta+k$.  Consider any two disjoint connected components $D$ and $D'$ of $B(z_0;2\delta_0)\backslash \cup_{k\in\bbZ}(\eta+k)$. Then if $\eta_1,\eta_2\in\Lambda_{h^\rmS}$ such that the starting point of $\eta_1$ (resp.\ $\eta_2$) is in $D$ (resp.\ $D'$), then $\eta_1$ and $\eta_2$ are disjoint from each other before leaving $B(z_0;2\delta_0)$. This further implies with positive probability, that there exists a sequence $\{z_n\}$ with rational coordinates converging to some $z_0'$ such that flow lines of $h^\rmS|_{\cR_0}$ of angle $\pm \frac{\pi}{2}$ started from different $z_j$ are disjoint from each other before exiting $B(z_0;2\delta_0)$, and the same holds with positive probability for $\wt h^\rmS|_{\cR_0}$ by local absolute continuity. Consider the space-filling counterflow line $\wt\eta'$ of $\wt h^\rmS$ defined in terms of angle $\pm\frac{\pi}{2}$ flow lines of $\wt h^\rmS$ as in Section~\ref{sec: prelim ig} from $-\infty$ to $+\infty$, which is a space-filling $\SLE_{\kappa'}(\frac{\kappa'}{2}-4;0)$ curve. Since almost surely, $\wt\eta'$ visits each rational point only once,  without loss of generality assume that $\wt\eta'$ visits $z_1,z_2,z_3,...$ in order; otherwise we look into the time reversal of $\wt\eta'$. Then from the construction of $\wt\eta'$ and our assumption on $\{z_n\}$, there exists rational points $\{w_n\}\subseteq \cR_0\backslash B(z_0;2\delta_0)$ such that $\wt\eta'$ visits the points $z_1,w_1,z_2,w_2,z_3,w_3,....$ in order. This contradicts with the continuity of $\wt \eta'$ as in~\cite[Theorem 4.12]{ImagGeo4}.
 \end{proof}

\begin{figure}
    \centering
    \includegraphics[scale=0.8]{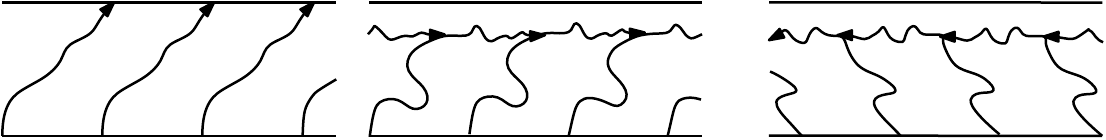}
    \caption{Flow line $\eta$ either hits   $\mathbb R + i\tau$, or merges into one of its translations $\eta + k$. }
    \label{fig:flowline-cts}
\end{figure}

We call flow lines of $h^\rmS$ of angle $\frac{\pi}{2}$ started from $x\in\bbR$  the \emph{west-going flow line} of $h^\rmS$ from $x$, and we denote it by $\eta_x^\rmW$. Likewise, we call flow lines of $h^\rmS$ of angle $-\frac{\pi}2$ started from $y+i\tau$   where $y\in\bbR$ the \emph{east-going flow line} of $h^\rmS$ from $y +i\tau$, and denote it by $\eta_{y+i\tau}^\rmE$. Using the change of coordinates formula~\eqref{eq:ig-change-coord}, east-going flow lines have the same law as the west-going flow lines up to a rotation by 180 degrees. Following the interaction rule of the flow lines in~\cite[Theorem 1.5]{ImagGeo1}, by our choice of boundary data, $\bbR+i\tau$ (resp.\ $\bbR$) can be viewed as a west-going flow line targeted at $+\infty$ (resp.\ east-going flow line targeted at $-\infty$).  In particular, if we view $\eta_x^\rmW$ as a flow line targeted at $+\infty$, then  $\eta_x^\rmW$ merges into $\bbR+i\tau$ upon intersecting, while if we view $\eta_{x+i\tau}^\rmE$ as a flow line targeted at $-\infty$, then $\eta_{x+i\tau}^\rmE$ merges into $\bbR$ upon intersecting.

\begin{lemma}\label{lem:ig-flow-line-cross}
  Let $x,y\in\bbR$ be deterministic.  With positive probability, $\eta_x^\rmW$ hits  $\bbR+i\tau$. Under this event, almost surely, the following holds:
  \begin{enumerate}[(i)]
      \item  $\eta_x^\rmW$ is disjoint from $\bbR+i\tau$ except for its ending point. Furthermore, $\eta_x^\rmW$ is disjoint from $(x,\infty)$, and  if $\kappa\in(0,2]$, then $\eta_x^\rmW$ is also disjoint from $(-\infty,x)$;
      \item $\eta_x^\rmW$ is disjoint from all of its translates $\{\eta_x^{\rmW}+k:k\in\bbZ\}$;
      \item  For  each connected component $D$ of $\cS_\tau\backslash (\cup_{k\in\bbZ}(\eta_x^\rmW+k) )$, $h^\rmS|_D$ has the conditional law as a zero boundary GFF on $D$ plus a harmonic function which is equal to $h^\rmS$ on $\partial D\cap\partial \cS_\tau$ and has flow line boundary conditions on $\partial D\cap (\cup_{k\in\bbZ}(\eta_x^\rmW+k))$. 
      \item We have  $\eta_y^\rmW$ hits $\bbR+i\tau$, and $\eta_{y+i\tau}^\rmE$ hits $\bbR$.
      \item The same conclusion holds for the conditional law of $h^\rmS$ given $\eta_{y+i\tau}^\rmE$ in connected components of $\cS_\tau\backslash (\cup_{k\in\bbZ}(\eta_{y+i\tau}^\rmE+k) )$ as in part (iii).
  \end{enumerate}
\end{lemma}

\begin{proof}
    Let $\wt h^\rmS$ be the same as in the proof of Proposition~\ref{prop:def-flow-line-strip}. Then following identical arguments as in~\cite[Lemma 3.9]{ImagGeo4}, with positive probability,  the angle $\frac{\pi}{2}$ flow line of $\wt h^\rmS$ started from $x$ stays within the 0.01-neighborhood of $\{x+yi:y\in[0,\tau]\}$ and merges with $\bbR+i\tau$. The  claim that $\eta_x^\rmW$ hits  $\bbR+i\tau$ with positive probability further follows from the local absolute continuity between the laws of $h^\rmS$ and $\wt h^\rmS$. 
    
    Now we prove the claims (i)-(v). The first half of  claim (i) is immediate from our construction of $\eta_x^\rmW$ in Proposition~\ref{prop:def-flow-line-strip} since $\eta_x^\rmW$ is terminated upon hitting $\bbR+i\tau$. The second half follows
    from the flow line interaction rules in~\cite[Theorem 1.5]{ImagGeo1} and~\cite[Theorem 1.9]{ImagGeo4} along with local absolute continuity. Claim (ii) is also a straightforward consequence of our construction that flow lines are terminated upon hitting $\bbR+i\tau$ or its translates. 
    Claim (iii) is then a consequence of part (i) of Proposition~\ref{prop:def-flow-line-strip} along with the continuity of $\eta$ from Lemma~\ref{lem:flow-line-continuity} and~\cite[Proposition 3.8]{ImagGeo1}. 

    To prove claim (iv), if  $y-x\in\bbZ$, then $\eta_x^\rmW$ is a translate of  $\eta_0^\rmW$ and the conclusion holds.  Now assume $y-x\notin \bbZ$ and we pick $k\in\bbZ$ such that $x+k<y<x+k+1$.  Following the locality of the flow lines as in Proposition~\ref{prop:def-flow-line-strip}, conditioned on $\eta_{x+k}^\rmW$ and $\eta_{ x+k+1}^\rmW$,  $\eta_y^\rmW$ has the law of the angle $\frac{\pi}{2}$ flow line of the field $h^\rmS$  in the domain bounded by $\eta_{x+k}^\rmW$, $\eta_{ x+k+1}^\rmW$ and $\partial S_\tau$. On the other hand, from the flow line interaction rules in~\cite[Theorem 1.5]{ImagGeo1} for the field as described in claim (iii), $\eta_y^\rmW$ merges into $\eta_{x+k}^\rmW$, $\eta_{x+k+1}^\rmW$ or $\bbR+i\tau$ upon hitting and does not cross  $\eta_{x+k}^\rmW$ and $\eta_{x+k+1}^\rmW$. This gives the first half of the claim. The second half  follows similarly by working with the events where $y+i\tau$ is between  $\eta_{x'}^\rmW$ and $\eta_{x'+1}^\rmW$ where $x'\in\bbQ$. Claim (v) then follows from part (i) of Proposition~\ref{prop:def-flow-line-strip} and claim (iv).
\end{proof}



\begin{figure}
    \centering
    \includegraphics[scale=0.66]{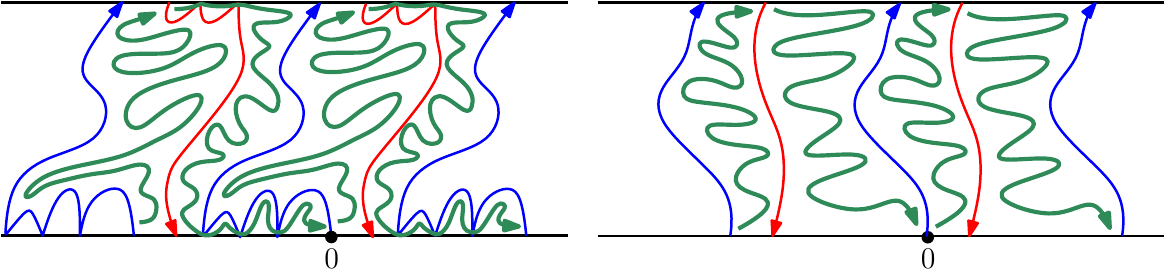}
    \caption{An illustration of the definition of the space-filling curve $\eta'$, where the blue curves are  $\eta_0^\rmW$ and its translates, the red curves are $\eta_{y+i\tau}^\rmE$ and its translates, and $\eta'$ is the concatenation of the green curves. In the left panel $\kappa\in (2,4)$ and in the right panel $\kappa\in (0,2]$. The space-filling $\SLE_\kappa$ loop in the annulus is defined to be the image of the green curve under the map $p:z\mapsto e^{2\pi iz}$.}
    \label{fig:space-sf-def}
\end{figure}

Now we work on the event $E^\cros$ where $\eta_0^\rmW$ hits $\bbR+i\tau$.  We define the \emph{counterclockwise space-filling counterflow line of $h^\rmS$} as follows. Let $x\in\bbR$. Let $\cS_{\eta_x^\rmW}$ be the (beaded if $\kappa\in (2,4)$) domain  which is formed by the right side of $\eta_x^\rmW$, the left side of $\eta_x^\rmW+1$ and $\partial\cS_\tau$. Consider the space-filling counterflow line of $h^\rmS$ in $\cS_{\eta_x^\rmW}$ from $x$ to $x+1$ defined in the usual sense as in Section~\ref{sec: prelim ig} and~\cite[Section 1.2.3]{ImagGeo4}, where if $\kappa\in(2,4)$ $\eta'$ is a concatenation of a number of space-filling curves. Note that by the flow line boundary conditions, the law of   $\eta'$ inside $\cS_{\eta_x^\rmW}$ is precisely the $\SLE_{\kappa'}(\frac{\kappa'}{2}-4;\frac{\kappa'}{2}-4)$ as described in the paragraph before Theorem~\ref{thm:ASYZ3.9}.
Let $\eta'$ be the concatenation of $\eta'$ and its  translates $\{\eta'+k:k\in\bbZ\}$. See Figure~\ref{fig:space-sf-def} for an illustration.

\begin{lemma}\label{lem:ig-def-SLEloop}
    The counterclockwise space-filling counterflow line $\eta'$ above is invariant of the choice of $x$. Moreover, for any deterministic $y\in\bbR$, the boundary of $\eta'$ stopped upon hitting $y$ is $\eta_y^\rmW$, while the boundary of $\eta'$ stopped upon hitting $y+i\tau$ is $\eta_{y+i\tau}^\rmE$.
\end{lemma}

\begin{proof}
    We start with the first half of the second claim. If $y\in x+\bbZ$, then the claim is clear from our construction. Without loss of generality assume $y\in(x,x+1)$. From our construction of the space-filling counterflow lines, conditioned on $\eta_x^\rmW$ and  $\eta_{x+1}^\rmW$, the boundary of $\eta'$ stopped upon hitting $y$ is the angle $\frac{\pi}{2}$ flow line of $h^\rmS|_{\cS_{\eta_x^\rmW}}$ started from $y$, which coincides with the description of the conditional law of $\eta_y^\rmW$ given $\eta_x^\rmW$ and $\eta_{x+1}^\rmW$. This proves the first half of the second claim. Now the restriction of $\eta'$ to the times when $\eta'$ is between $\eta_x^\rmW$ (resp.\ $\eta_{x+1}^\rmW$) and $\eta_y^\rmW$ is the space-filling counterflow line of $h^\rmS$ in the region bounded by $\eta_x^\rmW$ (resp.\ $\eta_{x+1}^\rmW$), $\eta_y^\rmW$, $\bbR$ and $\bbR+i\tau$ from $x$ to $y$ (resp.\ $y$ to $x+1$), and from this description we conclude that the space-filling curve $\eta'$ agrees with the version defined using $y$ in place of $x$, finishing the proof for the first claim. The other half of the second claim follows similarly by working  with the events where $y+i\tau$ is between the $\eta_{x}^\rmW$ and $\eta_{x+1}^\rmW$ where $x\in\bbQ$. 
\end{proof}

Now we move back to the setting of the annulus. Recall that $h^\rmS=h_0\circ p+\lambda'-\chi\pi$ where $h_0$ is a zero boundary GFF in $\cA_\tau$. We write $h^{\mathrm{A}} = h_0+\chi\arg(\cdot)+\lambda'-\chi\pi$, where we view the $\arg(\cdot)$ as a multi-valued function on the annulus. 
\begin{definition}\label{def:ig-flow-line-annulus}
    For $x\in\partial \cA_\tau$, let $x'\in\cS_\tau$ such that $p(x')=x$. With slight abuse of notation, we define the west-going flow line $\eta_x^\rmW$ of $h^{\mathrm{A}}$ started from $x$ to be equal to $p(\eta_{x'}^\rmW)$ where  $\eta_{x'}^\rmW$ is the west-going flow line of $h^\rmS$ started from $x'$, and the east-going flow line  $\eta_x^\rmE$ of $h^{\mathrm{A}}$ started from $x$ to be equal to $p(\eta_{x'}^\rmE)$ where  $\eta_{x'}^\rmE$ is the east-going flow line of $h^\rmS$ started from $x'$. Let $E^\cros$ be the event where the west-going flow line of $h^{\mathrm{A}}$ started from $1$ hits $C_\tau$. Conditioned on  the event $E^\cros$, we define the counterclockwise space-filling $\SLE_{\kappa'}$ loop associated with $h^{\mathrm{A}}$ by taking the projection under $p$ of the space-filling curve $\eta'$ for $h^\rmS$ on $\cS_\tau$ in Lemma~\ref{lem:ig-def-SLEloop}, and write $\SLE_{\kappa',\tau}^{\mathrm{loop}}$ for its law. We write $\IG^*_\tau$ for the law of $h^{\mathrm{A}}$ conditioned on the positive probability event $E^\cros$.
\end{definition}

\begin{lemma}\label{lem:ig-curve-resample}
  Let $x\in C_0$ and $y\in C_\tau$.  Under the probability measure $\IG_\tau^*$, the conditional law of $\eta_y^\rmE$ given $\eta_x^\rmW$ and the conditional law of $\eta_x^\rmW$ given $\eta_y^\rmE$, agree with the description of the conditional laws of $\eta$ given $\wt\eta$ and $\wt\eta$ given $\eta$ as in Proposition~\ref{prop:curve-resample}. Furthermore, let $\eta'$ be the counterclockwise space-filling $\SLE_{\kappa'}$ loop associated with $h^{\mathrm{A}}$ parameterized by Lebesgue area, and let $\tau_x$ be the time when $\eta'$ hits $x$. Then $((-i)x^{-1}\eta_x^\rmW,(-i)x^{-1}\eta'(\cdot+\tau_x))\overset{d}{=} (\eta_{-i}^\rmW,\eta')$ where we set $\eta'(t+\mathrm{area}(\cA_\tau)) = \eta'(t)$ for $t\in\bbR$.
\end{lemma}

\begin{proof}
    The first claim follows from the claim (iii) and claim (v)  in Lemma~\ref{lem:ig-flow-line-cross} in the strip setting along with Definition~\ref{def:ig-flow-line-annulus}. The second claim follows from the translation invariance in law of the field $h^\rmS$.
\end{proof}

\section{Characterization of the limiting surface via resampling}
\label{sec: resampling characterization scaling limit}
In this section, we prove Theorem~\ref{thm: quantum annulus mot}. Recall that the measure $\MA^{1,0,\dagger}$ defined in Section~\ref{sec: proof sec 4 main thm} describes the conformal welding of a weight $(2-\frac{\gamma^2}{2},\gamma^2,\frac{\gamma^2}{2})$ quantum triangle to itself. Then in Section~\ref{subsec:RS-curve} we prove that the interface in $\MA^{1,0,\dagger}$ can be described in terms of  imaginary geometry flow lines as in Definition~\ref{def:ig-flow-line-annulus}, while in Section~\ref{subsec:RS-field}, we prove that the field in $\MA^{1,0,\dagger}$ is a Liouville field on the annulus. 
Our proof is based on a Markov chain resampling argument, where we build an irreducible Markov kernel such that the law of the interface and field from the conformal welding in $\MA^{1,0,\dagger}$ and the law of the imaginary geometry flow line in Definition~\ref{def:ig-flow-line-annulus} and the Liouville field on the annulus are both invariant measures, which allows us to deduce that the two laws are up to a constant thanks to the uniqueness of invariant measures for irreducible Markov chains~\cite{MT12}. 

Resampling uniqueness properties of SLE curves have been considered in~\cite{ImagGeo2,Yu22,Zhan23,ang2024radial}. One key difference in our setting is that we will with the annulus and we also resample the starting point of one interface according to the LQG length measure to transition between curves of different homotopy classes. The characterization of the field has been considered in~\cite{ModAnn,QuanTrig,ang2024radial}, which considered the resampling of the field in different deterministic domains. Our resampling argument in Section~\ref{subsec:RS-field} will involve resampling of the field in random domains, where we will additionally induce a rotation and use a rotation invariance of the field in $\MA^{1,0,\dagger}$.

\subsection{Characterization of the interface law}\label{subsec:RS-curve}
Recall the curve-decorated surface $\MA^{1,1,\dagger}$ and its embedding $(\cA_\tau,\phi,\eta,-i,\wt\eta,y,\eta') {/}{\sim_\gamma}$ as described in Section~\ref{sec: proof sec 4 main thm} where $|y|=e^{-2\pi\tau}$. We write $\cL_0^\phi$ (resp.\ $\cL_1^\phi$) for the LQG boundary length measure on $C_0=\partial\bbD$ (resp.\ $C_\tau=\{|z|=e^{-2\pi\tau}\}$) with respect to $\phi$. Also recall the probability measure $\IG_\tau^*$ from Definition~\ref{def:ig-flow-line-annulus}. We write $\mathsf m_y$ for the joint law of the flow lines $\eta_{-i}^\rmW$  and $\eta_y^\rmE$ under  $\IG_\tau^*$. The goal of this subsection is to prove the following.

\begin{proposition}\label{prop:pf-curve-resample}
    Let $(\cA_\tau,\phi,\eta,-i,\wt\eta,y,\eta') {/}{\sim_\gamma}$ be the embedding of  $\MA^{1,1,\dagger}$ as above. The conditional law of $(\eta,\wt\eta,y)$ given $\phi$ is given by the probability measure proportional to $\cL_1^\phi(\d y)\times \mathsf m_y$.
\end{proposition}

In Proposition~\ref{prop:curve-resample}, we derived a resampling property for the conditional law of $(\eta,\wt\eta,y)$ given $\phi$. Based on this, we define a Markov kernel $\Lambda^\curve$ as follows. Fix the modulus $\tau$. Consider the space $S^\curve$ of tuples $(\eta,\wt\eta,y)$, where $y\in C_\tau$, $\eta$ is a simple curve in $\ol{\cA_\tau}$ started from $(-i)$
and ending at a point in $C_\tau$, while $\wt\eta$ is a simple curve in $\ol{\cA_\tau}$ started from $y$ and ending at a point in $C_0$, such that $\eta,\wt\eta$ are disjoint from each other, and $\eta$ (resp.\ $\wt\eta$) is disjoint from $C_\tau$ (resp.\ $C_0$) except at its ending point. For each tuple $\omega=(\eta,\wt\eta,y)$, let $\Lambda^\curve(\omega,\cdot)$ be the law of $\omega_1=(\eta_1,\wt\eta_1,y_1)$, where we first sample $y_1$ from the probability measure proportional to $\cL_1^\phi$, and then sample $\wt\eta_1$ in $\cA_\tau\backslash \eta$ as an $\SLE_\kappa(-\frac{\kappa}{2},\kappa-2,-\frac{\kappa}{2};\frac{\kappa}{2}-2)$  curve from $y_1$ and targeted at  $(-i)^+$, and further sample $\eta_1$ in $\cA_\tau\backslash \wt\eta_1$ as an $\SLE_\kappa(-\frac{\kappa}{2},\kappa-2,-\frac{\kappa}{2};\frac{\kappa}{2}-2)$ curve from $-i$ and targeted at $y_1^+$. See Figure~\ref{fig:curve-resample} for an illustration.

\begin{figure}
    \centering
    \includegraphics[scale=0.65]{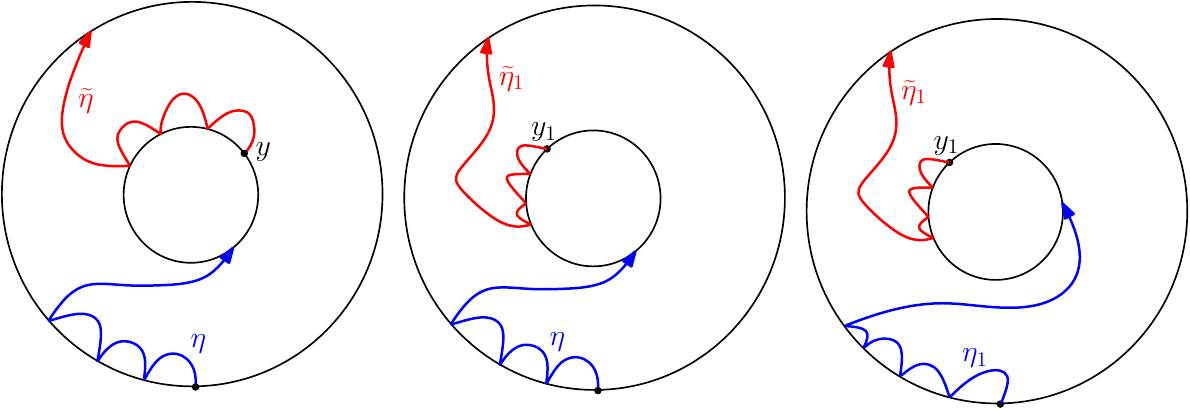}
    \caption{A single step in the resampling of the interfaces in the Markov kernel $\Lambda^\curve$.}
    \label{fig:curve-resample}
\end{figure}

\begin{lemma}\label{lem:curve-resample-0}
    The conditional law of $(\eta,\wt\eta,y)$ given $\phi$ under the measure $\MA^{1,1,\dagger}(a,b)$ and the measure $\cL_1^\phi(\d y)\times \mathsf m_y$ are both invariant measures of the Markov kernel $\Lambda^\curve$.
\end{lemma}

\begin{proof}
The first claim is a direct consequence of Proposition~\ref{prop:curve-resample}, and the second claim follows from Lemma~\ref{lem:ig-curve-resample}.
\end{proof}

The kernel $\Lambda^\curve$ defines a Markov chain $\{(\eta_n,\wt\eta_n,y_n):n\in\bN\}$ on $S^\curve$, which can be viewed as a metric space via Hausdorff distance between curves.   To prove Proposition~\ref{prop:pf-curve-resample}, we will prove that the Markov chain induced by $\Lambda^\curve$ is {irreducible}, in the sense that for any $(\eta_0,\wt\eta_0,y_0)$ and any measurable set $A\subseteq S^\curve$ such that $(\cL_1^\phi(\d y)\times \mathsf m_y)[A]>0$, there is some deterministic $N>0$ (possibly depending on $(\eta_0,\wt\eta_0,y_0)$) such that $\bbP[(\eta_N,\wt\eta_N,y_N)\in A]>0$. The key of our proof is the following result on $\SLE_\kappa(\underline{\rho})$, which is a consequence of ~\cite[Lemma 3.9]{ImagGeo4} and~\cite[Lemma 3.1]{Yu22}.

\begin{lemma}\label{lem:SLE-positive-prob}
  Let $a_2<a_1<0$.  Let $\eta$ be an $\SLE_\kappa(-\frac{\kappa}{2},\kappa-2,-\frac{\kappa}{2};\frac{\kappa}{2}-2)$  curve on $\bbH$ from 0 and targeted at $\infty$ with force points $0^-,a_1,a_2;0^+$, respectively. Let $\vartheta$ be any deterministic simple curve in $\ol \bbH$ from 0 to some point $a\in (-\infty,a_2)$. Then for any $\e>0$, with positive probability, $\eta$ stays within the $\e$-neighborhood of $\vartheta$. 
\end{lemma}

\begin{proof}[Proof of Proposition~\ref{prop:pf-curve-resample}]
    We first prove the irreducibility of $\Lambda^\curve$. Let $\Omega_1,\Omega_2\subsetneq \cA_\tau$   be some fixed pair of disjoint subdomains containing a pair of paths $(\vartheta,\wt\vartheta)$ in $\cA_\tau$ from $C_0$ to $C_\tau$ where $\vartheta$ starts from $-i$.  Let $G$ be a graph whose vertex set  is equal to $S^\curve$, and two configurations $(\eta_0,\wt\eta_0,y_0)$ and $(\eta_1,\wt\eta_1,y_1)$ is connected by an edge if and only if the curve $\eta_1$ is disjoint from $\eta_0$. In other words, we are able to move from  $(\eta_0,\wt\eta_0,y_0)$ to $(\eta_1,\wt\eta_1,y_1)$ by a single step described by the Markov kernel $\Lambda^\curve$. For $\ell\in\bR$, let $G^\ell$ be the subgraph of $G$ such that the winding number of $\eta_0$ around 0 (with fraction also counted) is equal to $\ell$. Then one can deduce from~\cite[Lemma 5.2]{ImagGeo4} that  $G^\ell$ is connected. On the other hand, for $\delta\in(0,1/4]$, it is straightforward to check that if we set $\vartheta$ (resp.\ $\vartheta_1$) be the projection of the line segment connecting $3/4$ and $3/4+\ell+i\tau$ (resp.\ $3/4-\delta$ and  $3/4-\delta+\ell+i\tau$)  under the map $p:z\mapsto e^{2\pi i \tau z}$, $\wt\vartheta$ be the projection of the line segment  connecting $1/4+\ell+i\tau$ and $1/4$ (resp.\ $-1/8+\ell+i\tau$ and $-1/8$) under $p$, then $(p(\ell+1/4),\vartheta_0,\wt\vartheta_0)\in G^\ell$ is connected with $(p(\ell-1/8),\vartheta_1,\wt\vartheta_1)\in G^{\ell-\delta}$ by an edge in $G$. This implies that the graph $G$ is connected. Further by Lemma~\ref{lem:SLE-positive-prob} and the fact that almost surely, every segment of $C_\tau$ has positive measure under $\cL^\phi_1$, one can prove that for each initial configuration $(\eta_0,\wt\eta_0,y_0)$, there is a deterministic $N>0$ such that $\bbP[\eta_N\in\Omega_1,\wt\eta_N\in\Omega_2]>0$. Further using the fact that GFF determines flow lines~\cite[Theorem 1.2]{ImagGeo1}, for any  $A\subseteq S^\curve$ such that $(\cL_1^\phi(\d y)\times \mathsf m_y )[A\cap \{\eta\in \Omega_1, \ \wt\eta\in\Omega_2\}]>0$, we have $\bbP\big[ \{(\eta_{N+1},\wt\eta_{N+1},y_{N+1})\in A \} \cap\{\eta_{N+1}\in\Omega_1,\wt\eta_{N+1}\in\Omega_2\}\big]>0$. This verifies the irreducibility of $\Lambda^\curve$. Now by~\cite[Proposition 10.1.1, Theorem 10.0.1]{MT12}, the Markov chain $\{( \eta_n,\wt\eta_n,y_n):n\in\bN\}$ is recurrent, and further by~\cite[Proposition 4.2.1]{MT12}, the invariant measure of this Markov chain is unique up to a multiplicative constant. This along with Lemma~\ref{lem:curve-resample-0} finish the proof.
\end{proof}

The following is an immediate consequence of Proposition~\ref{prop:pf-curve-resample}.

\begin{proposition}\label{prop:pf-curve-resample-1}
    Let  $(\cA_\tau,\phi,\eta,-i,\eta') {/}{\sim_\gamma}$ be the embedding of an LQG surface from $\MA^{1,0,\dagger}(a,b)$. Then conditioned on the modulus $\tau$, $(\eta,\eta')$ is independent from $\phi$. Moreover, $(\eta,\eta')$ has the law as that of the west-going flow line $\eta_{-i}^\rmW$ and the counterclockwise space-filling  loop $\SLE_{\kappa',\tau}^{\mathrm{loop}}$ under the measure $\IG_\tau^*$. Finally, for any $\e>0$, with positive probability (depending on $\e$), $\eta$ is contained within the $\e$-neighborhood of $\{-ir:e^{-2\pi \tau}<r<1\}$.
\end{proposition}

\begin{proof}
     Using the construction of $\MA^{1,1,\dagger}$ from $\MA^{1,0,\dagger}$, it is immediate from Proposition~\ref{prop:pf-curve-resample} that conditioned on $\tau$, $\eta$ is independent from $\phi$ and has the same law as $\eta_{-i}^\rmW$ under $\IG_\tau^*$. The claim on the space-filling curve $\eta'$ follows from the construction of the space-filling curve in $\cS_\tau$ and $\cA_\tau$. The last claim follows from the proof of Proposition~\ref{prop:pf-curve-resample}, where we take $\Omega_1$ to be  the $\e$-neighborhood of $\{-ir:e^{-2\pi \tau}<r<1\}$ there.
\end{proof}

\subsection{Characterization of the field law}\label{subsec:RS-field}

In this subsection, we further prove that the field from the conformal welding as in the definition of $\MA^\dagger$ can be described in terms of $\LF_\tau$, and hence complete the proof of Theorem~\ref{thm: quantum annulus mot}. Following the strategy of~\cite{ModAnn,QuanTrig}, we will construct a Markov kernel on the space of distributions and prove that both of the laws of the fields from $\MA^\dagger$ and the Liouville field on the annulus are stationary under this Markov kernel, and conclude the proof via resampling uniqueness.

We start with the following resampling property for the marked points on $\MA^{1,0,\dagger}$.

\begin{lemma}\label{lem:M-resample}
    Let $(\cA_\tau,\phi,\eta,-i,\eta') {/}{\sim_\gamma}$ be the embedding of an LQG surface from $\MA^{1,0,\dagger}(a,b)$. Let $x$ be a point on $C_0 = \partial\bbD$ sampled according to the probability measure proportional to $\cL^\phi_0$. Let $\wh\phi(\cdot) = \phi(x\cdot)$, and $\wh \eta$ be $(-i)x^{-1}$ times the west-going flow line started from $x$. Let $(-i)x^{-1}\eta'(\cdot+\tau_x)$ be defined via $\eta'$ and $x$ as in Lemma~\ref{lem:ig-curve-resample}. Then the law of $(\cA_\tau,\wh\phi,\wh\eta,-i,(-i)x^{-1}\eta'(\cdot+\tau_x)) {/}{\sim_\gamma}$ is still $\MA^{1,0,\dagger}(a,b)$.
\end{lemma}

\begin{proof}
    Following Proposition~\ref{prop:pf-curve-resample-1}, $(\wh\eta,\eta')$ has the desired law independent from $\wh\phi$ thanks to the second claim of Lemma~\ref{lem:ig-curve-resample}. To see the claim on  the field law, we first weight the law of $(\cA_\tau,\phi,\eta,-i,\eta') {/}{\sim_\gamma}$ by $b$ and sample $y\in C_\tau$ from the probability measure proportional to $\cL_1^\phi$. Then $(\cA_\tau,\phi,-i,y) {/}{\sim_\gamma}$ has the same law as the field and marked points in $\MA^{1,1,\dagger}(a,b)$, which can be understood as the conformal welding of a pair of surfaces from~\eqref{eq:MT-law-2} by Proposition~\ref{prop: MC+MD}. Following the symmetry in the above conformal welding, if we forget about the marked point $-i$, then the surfaces $(\cA_\tau,\phi,-i){/}{\sim_\gamma}$ has the same law as $(\cA_\tau,\phi,y){/}{\sim_\gamma}$. This implies that if we further sample $x$ from the probability measure proportional to $\cL_0^\phi$, then the law of $(\cA_\tau,\phi,y,-i) {/}{\sim_\gamma}$ agrees with the law of $(\cA_\tau,\phi,y,x) {/}{\sim_\gamma}$ when viewed as marked LQG surfaces. This finishes the proof.
\end{proof}

Now we consider the following Markov kernel on the space of generalized functions on $\cA_\tau$ such that $C_0$ has finite LQG length with respect to $\phi$. Let $\phi_0$ be a field on $\cA_\tau$. We first sample a point $x$ according to the probability measure proportional to $\cL_0^{\phi_0}$, and set $\wh\phi_1 = \phi_0((-i)x^{-1}\cdot)$. Let $D=\cA_\tau$, $A = D\backslash\{re^{i\theta}:r\in[e^{-2\pi \tau},1],\theta\in [5/4\pi,7/4\pi]\}$ be as in Lemma~\ref{lem:GFF-Markov-0}. Let $\phi_1|_{D\backslash A} = \wh\phi_1|_{D\backslash A}$ and $\phi_1|_{A } = h+\mathfrak{h}$ where $h$ is a GFF on $A$ with zero boundary condition on $\partial A\backslash\partial D$ and free boundary condition on $\partial A\cap \partial D$, and $\mathfrak{h}$ is the harmonic extension of   $ \wh\phi_1|_{D\backslash A}$ onto $A$ with zero normal derivative on $\partial A\cap \partial D$. We write $\Lambda^{\field}(\phi_0,\cdot)$ for the law of $\phi_1$.  For an open interval $I$ with compact closure in $(0,\infty)$, we write $\Lambda^{\field}_I(\phi_0,\cdot)$ for the measure $\Lambda^{\field}(\phi_0,\cdot)$ conditioned on $\{\| \cL_0^\phi \|,\|\cL_1^\phi\|\in I\}$.

\begin{figure}
    \centering
    \includegraphics[scale=0.6]{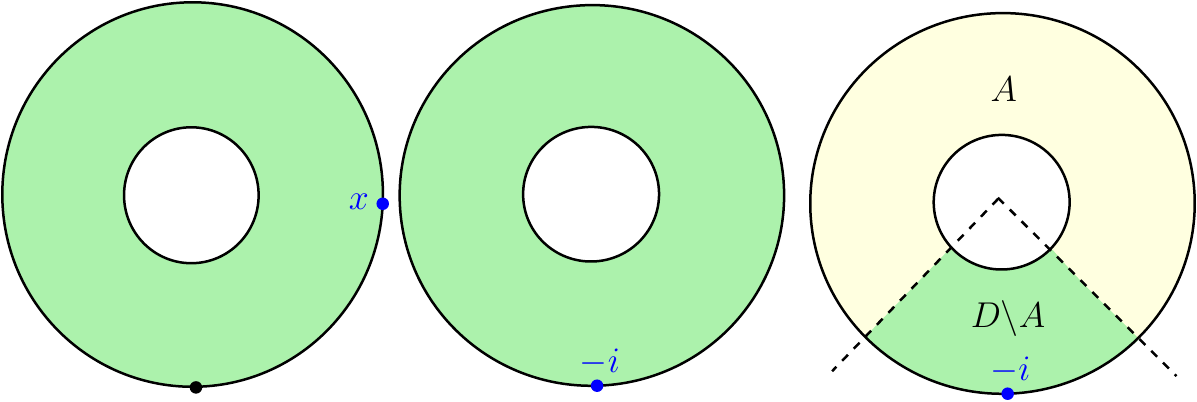}
    \caption{An illustration of the a single step in the kernel $\Lambda^\field$, where we first sample the blue point according to the probability measure proportional to the LQG length measure, rotate and forget about the black marked point, and further resample the field inside $A$ given its value on $D\backslash A$. In Lemma~\ref{lem:M-Markov}, we use the independence between the field and the interface in Proposition~\ref{prop:pf-curve-resample-1} to prove that the measure $M_\tau$, which describes the law of the field under the conformal welding when conditioned on the modulus to be $\tau$, is invariant under $\Lambda^\field$.}
    \label{fig:field-resample}
\end{figure}

Let $\MA^{1,0\dagger} = \int_0^\infty\int_0^\infty \MA^{1,0,\dagger}(a,b)\,\d a \d b$ be the $\sigma$-finite measure, and $(\cA_\tau,\phi,\eta,-i,\eta') {/}{\sim_\gamma}$ be an embedding of a surface from $\MA^{1,0\dagger}$. By Proposition~\ref{prop:pf-curve-resample-1}, the law of $(\phi,\eta,\eta')$ can be expressed as 
\begin{equation}\label{eq:M-decomposition}
    \int_0^\infty M_\tau\times \mathsf m_\tau\, \wh m(\d\tau),
\end{equation} 
where $\wh m(\d \tau)$ is a $\sigma$-finite measure on $(0,\infty)$ describing the law of the modulus $\tau$, $M_\tau$ is a $\sigma$-finite measure describing the law of the field $\phi$, and $\mathsf m_\tau$ is the probability measure describing the law of $(\eta,\eta')$ as in Proposition~\ref{prop:pf-curve-resample-1}. The following is the analog of Lemma~\ref{lem:GFF-Markov} for the measure $M_\tau$.

\begin{lemma}\label{lem:M-Markov}
    For $\wh m$-a.e.\ $\tau$, we have the following domain Markov property for $M_\tau$. Let $D=\cA_\tau$ and $A$ be as in Lemma~\ref{lem:GFF-Markov-0}. Let $(\cA_\tau,\phi_0,-i)$ be sampled according to $M_\tau$. Then conditioned on $\phi_0|_{D\backslash A}$, $\phi_0|_A\overset{d}{=} h+\mathfrak{h}$, where $h$ is a GFF on $A$ with zero boundary condition on $\partial A\backslash\partial D$ and free boundary condition on $\partial A\cap \partial D$, and $\mathfrak{h}$ is the harmonic extension of   $ \phi_0|_{D\backslash A}$ onto $A$ with zero normal derivative on $\partial A\cap \partial D$.   
\end{lemma}

We will use the following lemma, which is a generalization of~\cite[Lemma 5.5]{QuanTrig} to the setting where one of the insertion is possibly $Q^-$.
 
\begin{lemma}\label{lem:M-Markov-1}
    Let $\beta_1,\beta_2<Q$ and $\beta_3\in(-\infty,Q)\cup\{Q^-\}$. Let $D_0$ be a deterministic neighborhood of $[-1,0]$. Let $\psi\sim\LF_\bbH^{(\beta_1,0),(\beta_2,\infty),(\beta_3,-1)}$, and let $S=S(\psi)\subseteq \bbH$ be a bounded random set measurable with respect to $\psi$ such that $\ol S \cap D_0=\emptyset$. Suppose for any open set $U\subseteq\bbH$, the event $\{(\bbH\backslash S)\subseteq U\}$ is measurable with respect to $\psi|_U$. Then conditioned on $(S,\psi|_S)$ and $\{S\neq\emptyset\}$ in the sense of Definition~\ref{defn: conditioning Markov kernel definition}, we have $\psi|_S \overset{d}{=} h+\mathfrak{h}$  where $h$ is a GFF on $S$ with zero (resp.\ free) boundary conditions on $\partial S \cap \bbH$ (resp.\ $\partial S \cap \bbR$), $\mathfrak h$ is the harmonic extension of $\psi|_{\bbH \backslash S}$ to $S$ with normal derivative zero on $\partial S \cap \bbR$.
\end{lemma}

\begin{proof}
    We first consider the case where $\beta_3<Q$. Let $\rho$ be a deterministic measure on $D_0$ satisfying~\eqref{eq: log integrability}. Recall that in our definition of the Liouville field $\LF_\bbH$, the GFF $h$ is normalized to take zero average on the unit semicircle $C_0\cap\bbH$. By~\cite[Proposition 2.8]{ModAnn}, there exists some constant $Z_\rho$ such that if we write $\bbP_{\bbH,\rho}$ for the free boundary GFF on $\bbH$ normalized such that $\int_\bbH h\d \rho = 0$, then for $(h,\rho)\sim \bbP_{\bbH,\rho}\times [Z_\rho e^{-Qc}\d c]$, then the law of $\phi=h+\mathbf{c}-Q G_\rho(\cdot,\infty)$ is equal to $\LF_\bbH$, where $G_\rho$ is the covariance kernel for $\rho$. By further weighting the law of $\phi$ appropriately and adapting the argument in~\cite[Lemma 3.6]{AHS24}, it follows that for some constant $Z_\rho>0$, we can express a sample $\psi$ from $\LF_\bbH^{(\beta_1,0),(\beta_2,\infty),(\beta_3,-1)}$ as   
    \begin{equation}\label{eq:local-set-LF-1}
        \psi= h+\mathbf{c}+ \frac{\beta_1}{2}G_\rho(\cdot,0)+(\frac{\beta_2}{2}-Q)G_\rho(\cdot,\infty)+\frac{\beta_3}{2}G_\rho(\cdot,-1)
    \end{equation}
     with $(h,\mathbf{c})\sim \bbP_{\bbH,\rho}\times [Z_\rho e^{(\beta_1+\beta_2+\beta_3-2Q)c/2}\d c]$. Since we have assumed that $\ol S\cap D_0=\emptyset$, $\bbH\backslash S$ is now a local set for $h$, and the claim for $\beta_3<Q$ now follows from the decomposition~\eqref{eq:local-set-LF-1} together with the characterization of local sets in~\cite[Lemma 3.9]{SS13}. The $\beta_3=Q^-$ case follows by taking the limit $\lim_{\beta_3\uparrow Q^-}\frac{1}{Q-\beta_3}\LF_\bbH^{(\beta_1,0),(\beta_2,\infty),(\beta_3,-1)}$ following~\cite[Proposition 2.32]{QuanTrig}. 
\end{proof}

\begin{proof}[Proof of Lemma~\ref{lem:M-Markov}]
    We only prove the case where $\gamma\in(\sqrt2,2)$; the proof for $\gamma\in(0,\sqrt2)$ is the same and actually simpler since the quantum triangle in~\eqref{eq:MT-law} is simply connected. We follow the same methods as the proof of~\cite[Lemma 5.6]{QuanTrig}, which proved the same result for the setting where the quantum triangle is simply connected.. Let  $(\cA_\tau,\phi,\eta,-i,\eta') {/}{\sim_\gamma}$ be an embedding of a surface from $\MA^{1,0\dagger}$, $D_\eta$ be the connected component of $\cA_\tau\backslash\eta$ whose boundary contains $C_\tau$, and $ f_\eta: D_\eta\to\bbH$ be the conformal map sending $-i$ to $-1$ and the left side of $\eta$ to $[0,\infty)$. Let $\psi = \phi\circ f_\eta^{-1}+Q\log|(f_\eta^{-1})'|$, $\cD_1 = (\bbH,\psi,0,-1,\infty){/}{\sim_\gamma}$ and $\cD_2 = (\cA_\tau\backslash D_\eta,\phi){/}{\sim_\gamma}$. Then following the definition of $\MA^{1,0\dagger}$, $(\cD_1,\cD_2)$ form a quantum triangle from $\QT(2-\frac{\gamma^2}{2},\frac{\gamma^2}{2},\gamma^2)$ restricted to the event where the boundary arc connecting the weight $2-\frac{\gamma^2}{2}$ and ${\gamma^2}$ vertices has shorter length than the boundary arc connecting the weight $\frac{\gamma^2}{2}$ and weight $\gamma^2$ vertices, and $\cD_1$ is conditionally independent from $\cD_2$ on this event. Moreover, since the welding of $(\cD_1,\cD_2)$ is according to the LQG boundary length, if we mark the point  $x_0$ on the boundary arc $(-\infty,-1)$ of $\cD_1$ such that it is identified with the weight $\gamma^2$ vertex (which is embedded at $\infty$), the modulus $\tau$ along with the interface $\eta$ are measurable with respect to $(\cL^\psi|_{(-\infty,x_0]},\cD_2)$, where $\cL^\psi|_{(-\infty,x_0]}$ is the LQG length measure on $(-\infty,x_0]$ with respect to $\phi$. Now we remove this conditioning on $(\cD_1,\cD_2)$, and define  $S = f_\eta(A)$ if the welding is well-defined and $\eta$ is within the $0.01 e^{-\tau-\tau^{-1}}$-neighborhood of $\{-ir:e^{-2\pi \tau}<r<1\}$, and set $S=\emptyset$ otherwise. By the above measurability of the modulus and the interface,  we see that Lemma~\ref{lem:M-Markov-1} is applicable to the set $S$. Furthermore, by   Proposition~\ref{prop:pf-curve-resample-1}, the interface $\eta$ is independent from $\phi$ and has positive probability to stay within $A$. Lemma~\ref{lem:M-Markov} then follows from Lemma~\ref{lem:M-Markov-1} together with  conformal invariance of GFF.
\end{proof}

\begin{lemma}\label{lem:M-Markov-2}
    Let $\ol\Lambda^\field_I(\phi_0,\d \phi_2) = \iint \Lambda^\field_I(\phi_0,\d \phi_1)\,\Lambda^\field_I(\phi_1,\d \phi_2).$ Let $\LF_{\tau,I}^{(\gamma,-i)}$ be the finite measure which is given by the restriction of $\LF_{\tau}^{(\gamma,-i)}$ on the event $\{\|\cL_0^\phi\|,\|\cL^\phi_1\|\in I\}$. Then for any initial configuration $\phi_0$ with $\{\|\cL^{\phi_0}_0\|,\|\cL^{\phi_0}_1\|\in I\}$, the measure $\LF_{\tau,I}^{(\gamma,-i)}$ is absolutely continuous with respect to $\ol\Lambda^\field_I(\phi_0,\d \phi_2)$.
\end{lemma}

\begin{proof}
  By classical results on GFF, e.g.,~\cite[Proposition 3.4]{ImagGeo1}\footnote{This proposition is stated for Dirichlet GFF, and the same applies for the mixed or free boundary GFF as well.}, for $\phi_1$ from $\Lambda^\field_I(\phi_0,\d \phi_1)$ and $\wh\phi_1$ from $\LF_{\tau,I}^{(\gamma,-i)}$, the law of $\phi_1|_{\bbH\cap \cA_\tau}$ and the law of $\wh\phi_1|_{\bbH\cap \cA_\tau}$ are mutually absolutely continuous. Now as we sample $\phi_2$ from $\Lambda^\field_I(\phi_1,\d \phi_2)$ as in the definition of $\Lambda^\field_I$, we work with the positive probability event that the point $x$ satisfies $\arg x\in [0.49\pi,0.51\pi]$. Then following Proposition~\ref{prop: adding a gamma singularity to QA}, the law of $\phi_2|_{\cA_\tau\backslash A}$ and the law of   $\wh\phi_1|_{\cA_\tau\backslash A}$ are mutually absolutely continuous, and the claim follows by another application of ~\cite[Proposition 3.4]{ImagGeo1}.
\end{proof}

\begin{lemma}\label{lem:M-law}
    For $\wh m$-a.e.\ $\tau$, there is a  constant $c(\tau)\in(0,\infty)$ such that the measure $M_\tau$ is equal to $c(\tau)\LF_\tau^{(\gamma,-i)}$.
\end{lemma}

\begin{proof}
  By Proposition~\ref{prop: adding a gamma singularity to QA} and Lemma~\ref{lem:GFF-Markov}, the  finite measure $\LF_{\tau,I}^{(\gamma,-i)}$ an invariance measure with respect to $\ol\Lambda_I^\field$. We write $M_{\tau,I}$ for the restriction of the measure $M_\tau$ onto the event $\{\|\cL_0^\phi\|,\|\cL_1^\phi\|\in I \}$.  Following Lemma~\ref{lem:M-resample} and Lemma~\ref{lem:M-Markov}, the  finite measure $M_{\tau,I}$ is also invariant under $\ol\Lambda^\field$. By Lemma~\ref{lem:M-Markov-2}, the  kernel $\ol\Lambda_I^\field$ is irreducible. Therefore  by~\cite[Proposition 10.1.1, Theorem 10.0.1, Proposition 4.2.1]{MT12},   the invariant measure of this Markov chain is unique up to a multiplicative constant, and for some $c(\tau,I)\in(0,\infty)$,  $M_{\tau,I} = c(\tau,I)\LF^{(\gamma,-i)}_{\tau,I}$. For $I'\subsetneq I$, by further restricting to the event $\{\|\cL_0^\phi\|,\|\cL_1^\phi\|\in I' \}$, we have $c(\tau,I)=c(\tau,I')$. This completes the proof.
\end{proof}

\begin{proof}[Proof of Theorem~\ref{thm: quantum annulus mot}]
    Recall the decomposition in~\eqref{eq:M-decomposition}. With slight abuse of notation, we continue to write $\IG_\tau^*$ for the joint law of the west-going flow line $\eta_{-i}^\rmW$ and the space-filling loop $\eta'$ under the probability measure $\IG^*_\tau$ as in Definition~\ref{def:ig-flow-line-annulus}. Then following Proposition~\ref{prop:pf-curve-resample-1} along with Lemma~\ref{lem:M-law}, we have 
    \begin{equation}\label{lem:pf-thm-1.6}
        \MA^{1,0,\dagger} = \int_0^\infty c(\tau)\LF_\tau^{(\gamma,-i)}\times \IG^*_\tau \, \wh m(\d \tau):=  \int_0^\infty  \LF_\tau^{(\gamma,-i)}\times \IG^*_\tau \,   m(\d \tau),
    \end{equation}
    where we view the right hand side as a measure on curve-decorated LQG surfaces. In particular, if we set the measure $m$ in~\eqref{eq: MA 01 defn intro}, then   $ \MA^{1,0,\dagger} =  \MA^{1,0}$ if we forget about the interface $\eta$ in $ \MA^{1,0,\dagger}$. Furthermore, following our definition of the boundary length process $(L,R)$ in the annulus, if we view $\eta'$ as a loop rooted at $-i$, then $(L,R)$ agrees with the one in the definition of $\MT$ in the paragraph before Theorem~\ref{thm:ASYZ3.9}. Moreover, since by  Theorem~\ref{thm:ASYZ3.9}, the process $(L,R)$ a.s.\ determines the quantum triangle $\cT$ from~\eqref{eq:MT-law}  whereas $\MA^{1,0,\dagger}(a,b)$ describes the law of the conformal welding of $\cT$ to itself, we conclude that the LQG surface from $ \MA^{1,0,\dagger}$ is a.s.\ determined by the the boundary length process $(L,R)$ in the annulus. This completes the proof.
\end{proof}

\section{Law of the modulus}
\label{sec: law of the random modulus}
In this section, we prove Theorem~\ref{thm: law of the random modulus} following a similar approach to \cite{ModAnn}. On one hand, the law of the boundary lengths for $\MA^{ 1,0}=\MA^{ 1,0,\dagger}$ can be computed via Brownian paths thanks to Theorem~\ref{thm:ASYZ3.9}. On the other hand, following the $\LF_\tau$ description of $\MA^{ 1,0}$, one can also compute  the law of the boundary lengths for $\MA^{ 1,0}$ using integrability of Liouville conformal field theory on the annulus~\cite{LCFTAnnulus,ModAnn}. By comparing the above two ways of computation, we obtain,  following the approach in \cite{ModAnn}, the expression of $m(\d\tau)$.
This is done in Section~\ref{sec: law of the random modulus calculation}. In Section~\ref{app: gamma = 0 limit}, we study the law $m(\d\tau)$ in the semiclassical limit.

\subsection{Law of the random modulus}
\label{sec: law of the random modulus calculation}

We start with the following computation involving Brownian path measures.

\begin{proposition}\label{prop:BM-path-MA}
    There exists a constant $C_\gamma>0$, such that 
    \begin{equation}\label{eq:MA-10-length}
        \|\MA^{1, 0 }(a, b)\| = C_\gamma\frac{b}{a^2+b^2-2ab\cos\theta}.
    \end{equation}
\end{proposition}
\begin{proof}
    By Theorem~\ref{thm: quantum annulus mot} and the reversibility of the Brownian path measure, it suffices to show 
    \begin{equation*}
        \|B^{\gamma, \bH}_{0, -a+bi} \| = \|B^{\gamma, \bH}_{ -a+bi,0} \| = C_\gamma \frac{b}{a^2+b^2-2ab\cos\theta}.
    \end{equation*}

   By translation invariance we have $ \|B^{\gamma, \bH}_{ -a+bi,0} \| = \|B^{\gamma, \bH}_{ bi,a} \| $. Consider a Brownian motion $(L,R)$ with covariance~\eqref{eq: gamma correlated BM} started at $bi$, and write $\bbP_{bi}$ for its law. Recall that $\|B^{\gamma, \bH}_{ bi,a} \|$ is equal to  the exit density of $a\in\partial\bbH$ of the Brownian motion $(L,R)$ started at $bi$. Therefore
    \begin{equation}\label{eq:BM-ptn-fn}
       \|B^{\gamma, \bH}_{ -a+bi,0} \| = \|B^{\gamma, \bH}_{ bi,a} \|  = \frac{\partial }{\partial a}\mathbb{P}_{bi}\ab[(L,R) \text{ exits $\mathbb{H}$ from $[a,\infty)$}],
    \end{equation}
    Under $\bbP_{bi}$, we can write $(L_t,R_t) = (\mathbb{a}(X_t\sin\theta- Y_t\cos\theta), \mathbb{a}Y_t)$, where $(X,Y)$ is standard 2D Brownian motion started at $\mathbb{a}^{-1}b\cot\theta + \mathbb{a}^{-1}bi$. The event that $(L,R)$ exits $\mathbb{H}$ from $[a,\infty)$ is equivalent to the event that $(X,Y)$ exits $\bbH$ from $[\mathbb{a}^{-1}a/\sin\theta,\infty)$.  By standard results on exit probability of 2D Brownian motion, this event has probability $\frac{1}{\pi} \arctan\frac{b\sin\theta}{a-b\cos\theta}.$ Therefore we conclude the proof by~\eqref{eq:BM-ptn-fn}.
\end{proof}

Now we set $$\MA =\int_0^\infty \LF_\tau \IG^*_\tau\, m(\d\tau)$$ and view it as a measure on curve-decorated LQG surfaces. Following Proposition~\ref{prop: adding a gamma singularity to QA}, if we weight the law of a surface  from $\MA$ embedded  as $(\cA_\tau,\phi)$ by the $\|\cL_0^\phi\|$, and sample a point on $C_0$ according to the measure $\cL_0^\phi$, then the marked surface has law $\MA^{1,0}$.  Therefore by~\eqref{eq:MA-10-length}, for some $C_\gamma\in(0,\infty)$, we have 
\begin{equation}
    \label{eq: MA boundary law}
    \ab\|\MA(a, b)\| = \frac{C_\gamma}{a^2+b^2-2ab\cos\theta}.
\end{equation}


The following theorem was obtained with input from Liouville CFT on the annulus \cite{LCFTAnnulus}. We restate it in our context.
\begin{theorem}[{\cite[Theorem 1.7]{ModAnn}}]
    \label{thm: modulus LCFT}
    Let $L_0$, $L_1$ be the quantum length of the outer and inner boundary, respectively, and $\langle \cdot \rangle^\gamma_{\MA}$ be the integration with respect to $\MA$ with certain $\gamma\in(0, 2)$. We have
    \begin{equation*}
        \int_0^\infty \exp\ab(-\theta x^2 \tau)\,m(\d\tau) =
        \frac{2\sinh \ab(\theta x)}{\pi\gamma x \Gamma(1+ix)} 
        \langle L_1 e^{-L_1} L_0^{ix} \rangle^\gamma_{\MA},
    \end{equation*}
    for any $x\in\bR$.
\end{theorem}

\begin{proof}[Proof of Theorem~\ref{thm: law of the random modulus}]
Throughout the proof, $C_\gamma\in(0,\infty)$ is some constant depending only on $\gamma$ that may vary line by line.    By \eqref{eq: MA boundary law},
\begin{equation}
    \label{eq: expectation intermediate}
    \langle L_1 e^{-L_1} L_0^{ix} \rangle^\gamma_{\MA} = 
    C_\gamma\int_0^\infty \d a \int_0^\infty \d b \frac{ a e^{-a} b^{ix}}{a^2+b^2-2ab\cos\theta}.
\end{equation}
Setting $b=at$ and noting that for $\theta\in (0,\pi)$
\begin{equation*}
    \int_0^\infty a^{ix} e^{-a} \d a = \Gamma(1+ix), \quad
    \int_0^\infty \frac{t^{ix}}{1+2t\cos\theta+t^2}\d t = \frac{\pi\sinh(\theta x)}{\sin\theta\sinh(\pi x)},
\end{equation*}
we have
\begin{equation*}
    \eqref{eq: expectation intermediate} = \frac{\Gamma(1+ix)\sin\theta}{\theta} \int_0^\infty \frac{t^{ix}}{1+2t\cos(\pi-\theta)+t^2}\d t = C_\gamma \frac{\Gamma(1+ix)\pi\sinh((\pi-\theta) x)}{\theta\sinh(\pi x)}.
\end{equation*}

Combining Theorem~\ref{thm: modulus LCFT} and the last equation, we conclude that,
\begin{equation}
    \label{eq: modulus first expression}
    \int_0^\infty \exp\ab(-\theta x^2\tau)\,m(\d\tau)
    =  C_\gamma \frac{\sinh(\theta x)\sinh((\pi-\theta)x)}{x \sinh(\pi x)},
\end{equation}

For $\ell\in (0, \infty)$, $s$, $\beta\in\bC$ such that $\re s>0$ and $|\re \beta| + |\im \beta|\le \ell$, by e.g.~\cite[Chapter 20]{NIST:DLMF}, we have the following identity for the Jacobi theta function with imaginary arguments
\begin{equation}\label{eq:JAcobi-1}
    \int_0^\infty e^{-st}\vartheta_4\ab(\frac{\beta\pi}{2\ell} \bigg\vert \frac{i\pi t}{\ell^2}) \d t = \frac{\ell\cosh(\beta\sqrt{s})}{\sqrt{s}\sinh(\ell\sqrt{s})},
\end{equation}
where we use the convention $q = e^{i\pi t}$,
\begin{equation*}
    \vartheta_4(z | t) = \vartheta_4(z, q) = 1+2\sum_{n=1}^\infty (-1)^n q^{n^2} \cos(2 n z).
\end{equation*}

Note that $\sinh (\theta x) \sinh((\pi-\theta) x) = \frac{1}{2}[\cosh(\pi x)-\cosh((\pi-2\theta)x)]$.
Let $s = \theta x^2$. Then  
\begin{equation}\label{eq:JAcobi-2}
  \int_0^\infty e^{-st}\,m(\d t)
    = C_\gamma \frac{
        \cosh(\pi \sqrt{\frac{s}{\theta}}) - \cosh((\pi-2\theta)\cdot\sqrt{\frac{s}{\theta}})
    }{
        \sqrt{s} \sinh\ab(\frac{\pi}{\sqrt{\theta}}\sqrt{s})
    }.
\end{equation}
Picking $\beta=\frac{\pi}{\sqrt{\theta}}$ or $\frac{-2\theta + \pi}{\sqrt{\theta}}$ and $\ell=\frac{\pi}{\sqrt{\theta}}$, comparing~\eqref{eq:JAcobi-1}  and~\eqref{eq:JAcobi-2}, we get
\begin{equation}\label{eq:mdtau}
\begin{split}
    m(\d \tau) &= 
    C_\gamma\ab[
        \vartheta_4\ab(
            \frac{\pi}{2} \bigg\vert 
            \frac{i\theta\tau}{\pi}
        ) - \vartheta_4\ab(
            \frac{\pi - 2\theta}{2} \bigg\vert \frac{i\theta\tau}{\pi}
        )
    ]\d\tau\\
    &= C_\gamma \ab[\vartheta_4\ab(\frac{\pi}{2}, e^{-\theta\tau})-\vartheta_4\ab(\frac{\pi}{2}-\theta, e^{-\theta\tau})]\d\tau \\ 
    &= C_\gamma \ab[ \vartheta_3(0,e^{-\theta\tau}) - \vartheta_3(\theta, e^{-\theta\tau}) ]\\
    &= C_\gamma \sum_{n=1}^\infty  e^{-n^2\theta\tau} \sin^2(n\theta) \d\tau
    \end{split}
\end{equation} 
where we have used
\begin{equation*}
    \vartheta_3(z | t) = \vartheta_3(z, q) = 1+2\sum_{n=1}^\infty q^{n^2}\cos(2 n z) = \vartheta_4\ab(z +\frac{1}{2}\pi \big| t).
\end{equation*}
It is clear from~\eqref{eq:mdtau} that the measure $m$ is finite. This finishes the proof.
\end{proof}

\subsection{Semiclassical limit of \texorpdfstring{$m(\d\tau)$}{m(d tau)}}
\label{app: gamma = 0 limit}

Recall that $g(\theta,\tau) = \vartheta_3(0,e^{-\theta\tau}) - \vartheta_3(\theta, e^{-\theta \tau})$. We prove the following proposition, which roughly says that the modulus $\tau$ converges to 0 in probability as $\gamma\downarrow 0$.
\begin{proposition}
    For every $\delta>0$,
    \begin{equation*}
        \lim_{\theta\downarrow 0} \frac{\int_{\delta}^\infty g(\theta, \tau)\d\tau}{\int_{0}^\infty g(\theta, \tau)\d\tau} = 0.
    \end{equation*}
\end{proposition}
\begin{proof}
    By \cite[Equation 20.2.3]{NIST:DLMF},
    \begin{equation*}
        g(\theta,\tau) = 2\sum_{n=1}^\infty e^{-n^2\theta\tau}(1-\cos(2n\theta)). 
    \end{equation*}
    Note that each term in the summation is positive and decays exponentially with respect to $\tau$ for $\theta>0$. By Fubini--Tonelli,
    \begin{equation*}
        \int_{\delta}^\infty g(\theta,\tau)\d\tau 
        = \sum_{n=1}^\infty 2(1-\cos(2n\theta))\int_{\delta}^\infty e^{-n^2\theta\tau}\d\tau 
        = \sum_{n=1}^\infty \frac{2e^{-n^2\theta\delta}(1-\cos(2n\theta))}{n^2\theta}.
    \end{equation*}
    Let $M_\theta = \lfloor \theta^{-2/3} \rfloor$. Then
    \begin{equation*}
    \sum_{n=1}^\infty \frac{2e^{-n^2\theta\delta}(1-\cos(2n\theta))}{n^2\theta} = 
    \underbrace{\sum_{n=1}^{M_\theta} \frac{2e^{-n^2\theta\delta}(1-\cos(2n\theta))}{n^2\theta}}_{S_1} + 
    \underbrace{\sum_{n=M_\theta+1}^\infty \frac{2e^{-n^2\theta\delta}(1-\cos(2n\theta))}{n^2\theta}}_{S_2}.
    \end{equation*}
    Since $1-\cos x\leq \frac{x^2}{2}$, we have $1-\cos(2n\theta)\leq 2n^2\theta^2$, and thus
    \begin{equation}
        \label{eq: S1 estimate}
        S_1 \le \sum_{n=1}^{M_\theta} \frac{4n^2\theta^2 e^{-n^2\theta\delta}}{n^2\theta}
        \le M_\theta \cdot 4\theta \le 4\theta^{1/3}.
    \end{equation}

    On the other hand, using the crude estimate $|1-\cos(2n\theta)|<2$, we set $x=\theta^{1/2} n$ and obtain
    \begin{equation}
        \label{eq: S2 estimate}
        S_2 \le \sum_{n=M_\theta+1}^\infty \frac{4 e^{-\delta x^2}}{x^2} \le
        4 \theta^{-1/2} \int_{\theta^{-1/6}}^\infty x^{-2} e^{-\delta x^2}\d x \le 4\theta^{-1/2} e^{-\delta\theta^{-1/3}} \int_{\theta^{-1/6}}^\infty x^{-2}\d x =4\theta^{-1/3} e^{-\delta\theta^{-1/3}}.
    \end{equation}

    Combining \eqref{eq: S1 estimate} and \eqref{eq: S2 estimate}, we conclude that for any $\delta>0$,
    \begin{equation*}
        \lim_{\theta\downarrow 0} \int_{\delta}^\infty g(\theta, \tau)\d\tau = 0.
    \end{equation*}

    Now we give a lower bound for $\int_0^\delta g(\theta,\tau)\d\tau$. 
    Using Fubini--Tonelli again,
    \begin{equation}
        \label{eq: g 0 delta integral estimate}
        \int_0^\delta g(\theta,\tau)\d\tau = 2\sum_{n=1}^\infty \frac{1-e^{-n^2\theta\delta}}{n^2\theta}(1-\cos(2n\theta)).
    \end{equation}
    
    We keep only the terms where $\frac{\pi}{6\theta} \le n \le \frac{\pi}{3\theta}$. In this range, $1-\cos(2 n\theta) \ge \frac{1}{2}$ and $1-e^{-n^2\theta\delta} \ge 1- e^{-\pi^2\delta/(9\theta)} \ge \frac{1}{2}$ for $\theta$ sufficiently small. As a result,
    \begin{equation*}
        \eqref{eq: g 0 delta integral estimate}
        \ge 2\sum_{n=\lceil\frac{\pi}{6\theta}\rceil}^{\lfloor \pi/3\theta \rfloor} \frac{(1-e^{-n^2\theta\delta}) (1-\cos(2n\theta))}{n^2\theta} \ge 2 \ab(\frac{\pi}{6\theta}-2)\frac{1}{4}\frac{9\theta}{\pi^2}=\frac{3}{4\pi}-\frac{9\theta}{\pi^2} > \frac{1}{2\pi},
    \end{equation*}
    where in the last inequality we also assumed that $\theta$ is sufficiently small.

    We thus conclude that for each fixed $\delta > 0$, 
    \begin{equation*}
        \lim_{\theta\downarrow 0} \frac{\int_\delta^\infty g(\theta,\tau)d\tau}{\int_0^\infty g(\theta, \tau) d\tau} = 0.
    \end{equation*}
\end{proof}

\begin{remark}\label{remark: gammato0}
 If one condition on the boundary length $b$ in the quantum annulus from $\MA$ to be equal to 1, then following~\eqref{eq:MA-10-length}, as $\gamma\downarrow 0$, the other side length $a$ tends to 1 in probability as well, which is in accordance with the fact that $\tau\downarrow0$ in probability as  $\gamma\downarrow 0$.   It is natural to ask if there is any nontrivial limits as we send $\gamma \downarrow 0$ and scale differently. For instance, one can require that the LQG area of the surface is equal to 1. Since the LQG area of the quantum annulus from $\MA$ is equal to the duration of the Brownian path, as in the proof of Proposition~\ref{prop:BM-path-MA}, the joint law of the boundary lengths $(a,b)$ and the area $A$ for a sample from $\MA$ is proportional to the Brownian path measure of the standard 2D Brownian motion in $\bbH$ from $\frac{b}{\sqrt 2}(\cos\theta (\sin\theta)^{-1/2}+(\sin\theta)^{1/2}i)$ to $\frac{a}{\sqrt{2 \sin\theta}}$ of duration $A$. It is then possible to prove that as $\theta\downarrow 0$, we have the following limits:
 \begin{itemize}
     \item conditioned on $A=1$,  $a,b\to\infty$ and $b/a\to1$;
     \item conditioned on $A=1,b=1$, then $a\to 1$;
     \item conditioned on $A> (\sin\theta)^{-1/2}$ and $b=1$, then $a$ converges in law to a random distribution.
 \end{itemize}
 We further comment that it is possible to work with the joint law of $(a,b,A,\tau)$ under the measure $\MA$ using the integrability of $\LF_\tau$ from~\cite{LCFTAnnulus}; see e.g.~\cite[Equation (3.2)]{ModAnn}.
\end{remark}








\printbibliography

\end{document}